\documentclass[oneside,english]{amsart}
\usepackage[T1]{fontenc}
\usepackage[latin9]{inputenc}
\usepackage{color}
\usepackage{amsthm}
\usepackage{amssymb}

\makeatletter

\let\SF@@footnote\footnote
\def\footnote{\ifx\protect\@typeset@protect
    \expandafter\SF@@footnote
  \else
    \expandafter\SF@gobble@opt
  \fi
}
\expandafter\def\csname SF@gobble@opt \endcsname{\@ifnextchar[
  \SF@gobble@twobracket
  \@gobble
}
\edef\SF@gobble@opt{\noexpand\protect
  \expandafter\noexpand\csname SF@gobble@opt \endcsname}
\def\SF@gobble@twobracket[#1]#2{}

\numberwithin{equation}{section}
\numberwithin{figure}{section}

\usepackage{amsfonts}
\usepackage{amsmath}
\usepackage{amsfonts}\usepackage{amsthm}\usepackage{color}\usepackage{enumitem}\usepackage{bm}\usepackage{cancel}\usepackage{thmtools}
\usepackage{hyperref}
\def\theenumi{\arabic{enumi}}

\def\theenumii{\alph{enumii}}
\def\p@enumii{\theenumi.}

\def\theenumiii{\arabic{enumiii}}
\def\p@enumiii{(\theenumi)(\theenumii)}

\def\p@enumiv{\p@enumiii.\theenumiii}
\newtheorem{theorem}{Theorem}[section]

\newtheorem{assumption}[theorem]{Assumption}

\newtheorem{corollary}[theorem]{Corollary}

\newtheorem{fact}[theorem]{Fact}
\newtheorem{lemma}[theorem]{Lemma}

\newtheorem{proposition}[theorem]{Proposition}

\theoremstyle{definition}
\newtheorem{definition}[theorem]{Definition}
\newtheorem{notation}[theorem]{Notation}
\newtheorem{remark}[theorem]{Remark}
\newtheorem{example}[theorem]{Example}

\makeatother

\usepackage{babel}
\begin{document}
\title{On rigid stabilizers and invariant random subgroups of groups of homeomorphisms }
\author{Tianyi Zheng}
\address{Tianyi Zheng, Department of Mathematics, UC San Diego, 9500 Gilman
Dr, La Jolla, CA 92093, USA. Email address: tzheng2@math.ucsd.edu.}
\begin{abstract}
A generalization of the double commutator lemma for normal subgroups
is shown for invariant random subgroups of a countable group acting
faithfully on a Hausdorff space. As an application, we classify ergodic
invariant random subgroups of topological full groups of Cantor minimal
$\mathbb{Z}^{d}$-systems. Another corollary is that for an ergodic
invariant random subgroup of a branch group, a.e. subgroup $H$ must
contain derived subgroups of certain rigid stabilizers. Such results
can be applied to show that every IRS of the first Grigorchuk group
is co-sofic.
\end{abstract}

\maketitle

\section{Introduction}

Let $G$ be a locally compact group and denote by ${\rm Sub}(G)$
the space of closed subgroups of $G$ equipped with the Chabauty topology.
An \emph{invariant random subgroup} (IRS) of $G$ is a Borel probability
measure on ${\rm Sub}(G)$ which is invariant under conjugation by
$G$. The term IRS was coined in Ab\'ert, Glasner and Vir\'ag \cite{AGV}.
Independently it was considered by Vershik in \cite{Vershik1,Vershik2}
(in terms of totally non-free actions) and Bowen in \cite{Bowen}
(in terms of random coset spaces).

Invariant random subgroups are closely related to probability measure
preserving (p.m.p.\@) actions. Given a p.m.p.\@ action $G\curvearrowright(X,m)$,
the pushforward of the probability measure $m$ under the stabilizer
map $x\mapsto{\rm St}_{G}(x)$ gives rise to an IRS, which we refer
to as the stabilizer IRS of the action $G\curvearrowright(X,m)$.
It is known that all IRSs arise in this way (\cite[Proposition 14]{AGV}),
and moreover, an ergodic IRS arises as the stabilizer IRS of an ergodic
p.m.p.\@ action (\cite[Proposition 3.5]{CreutzPeterson}). The study
of stabilizers of measure-preserving actions has been around for decades;
it dates back to the work of Moore, see \cite[Chapter 2]{AM66}. 

In \cite{BergeronGaboriau} it is observed that invariant random subgroups
behave more similarly to normal subgroups than arbitrary subgroups.
This theme has been developed tremendously in recent years, see for
instances results in \cite{7S,AGV,Vershik1,Vershik2}, also the survey
\cite{Gelander1} and references therein. Considerations in the present
work are also motivated by this observation. 

The following elementary lemma for normal subgroups has appeared in
various forms in many contexts. We refer to it as the double commutator
lemma because in its proof one takes commutators of the form $[\alpha,[\beta,\gamma]]$.
It is used in \cite[Section 7]{Gri-justinfinite} to show a criterion
for a branch group to be just infinite. This lemma is also useful
in the proofs of simplicity for certain classes of groups of homeomorphisms,
see for instances \cite{MatuiRemarkI,Nek-finitepresent,Nek-simple}. 

Given $G\curvearrowright X$ and $U\subseteq X$, denote by $R_{G}(U)$
the \emph{rigid stabilizer} of $U$, that is $R_{G}(U)=\{g\in G:\ x\cdot g=x\mbox{ for all }x\in X\setminus U\}$.
Throughout the paper, group actions are right actions. The commutator
of two group elements is $[g,h]=ghg^{-1}h^{-1}$. 

\begin{lemma}[Double commutator lemma for normal subgroups, {\cite[Lemma 4.1]{Nek-finitepresent}}]

Let $G$ be a group acting faithfully by homeomorphisms on a Hausdorff
space $X$, and $N$ be a non-trivial normal subgroup of $G$. Then
there exists a non-empty open subset $U\subset X$ such that $[R_{G}(U),R_{G}(U)]\le N$.

\end{lemma}

We show that for countable groups, to some extent the double commutator
lemma for normal subgroups applies to IRSs as well. 

\begin{theorem}[Double commutator lemma for IRS]\label{inclu-rigid}

Suppose $G$ is a countable group acting faithfully on a second countable
Hausdorff space $X$ by homeomorphisms. Let $\mu$ be an ergodic IRS
of $G$.
\begin{description}
\item [{(i)}] If $\mu\neq\delta_{\{id\}}$, then for $\mu$-a.e. subgroup
$H$, there exists a non-empty open set $U$ of $X$ such that $[R_{G}(U),R_{G}(U)]\le H$.
\item [{(ii)}] Suppose in addition that for any open set $V$ of $X$,
$R_{G}(V)$ has no fixed point in $V$. Then $\mu$-a.e. subgroup
$H$ satisfies the following property: if $x\in X$ is not a fixed
point of $H$, then there exists an open neighborhood $U$ of $x$
such that $[R_{G}(U),R_{G}(U)]\le H$. 
\end{description}
\end{theorem}

Roughly speaking, Theorem \ref{inclu-rigid} implies that IRSs share
the same property of containing derived subgroups of rigid stabilizers
as normal subgroups do, except that IRSs are allowed to choose fixed
point sets on $X$. 

Theorem \ref{inclu-rigid} is a useful tool for classification of
IRSs of a given group $\Gamma$ acting by homeomorphisms on $X$.
We discuss two classes of countable groups in this work: weakly branch
groups and topological full groups. Recent work of Becker, Lubotzky
and Thom \cite{BLT19} gives a characterization, for amenable groups,
of permutation stability (\emph{P-stability} for short, which means
asymptotic homomorphisms to permutation groups are close to actual
homomorphisms, see Definition \ref{def-P-stable}) in terms of IRSs.
For a discrete group, an IRS $\mu$ is said to be \emph{co-sofic},
if it is a weak-$\ast$ limit of IRSs supported on finite index subgroups,
see \cite[Section 2.5]{Gelander1}. Theorem 1.3 (ii) in \cite{BLT19}
asserts that if $G$ is a finitely generated amenable group, then
$G$ is P-stable if and only if every IRS of $G$ is co-sofic. This
result provides further motivation to understand IRSs of various finitely
generated amenable groups: with structural information one might be
able to check if the co-soficity criterion is satisfied. As an application
of Theorem \ref{inclu-rigid}, we show that all IRSs of the first
Grigorchuk group are co-sofic, answering a question in \cite{BLT19},
see more details below. 

\subsection*{Applications to topological full groups\label{subsec:full-intro}}

Topological full groups of $\mathbb{Z}$-actions are introduced in
\cite{GPS}, in connection to the theory of topological orbit equivalence
of minimal homeomorphisms of the Cantor set. Topological full groups
for \'etale groupoids are investigated in the series of works \cite{MatuiRemarkI,Matui2012,Matui2015}.
In \cite{MatuiRemarkI} it is shown that the derived subgroup of the
topological full group of a minimal subshift is simple. Amenability
of such groups is conjectured in \cite{Grigorchuk-Medynets} and proved
in \cite{JuschenkoMonod}, providing first examples of finitely generated
infinite simple amenable groups. 

When $\mathcal{G}$ is a groupoid of germs with unit space $X$ (relevant
definitions are recalled in Subsection \ref{subsec:full-def}), its
topological full group $\mathsf{F}(\mathcal{G})$ is defined as the
group of all homeomorphisms of $X$ whose germs belong to $\mathcal{G}$.
In \cite{Nek-simple}, for every groupoid of germs $\mathcal{G}$,
two normal subgroups $\mathsf{S}(\mathcal{G})$ and $\mathsf{A}(\mathcal{G})$
of the topological full group of $\mathcal{G}$ are defined, which
are analogues of the symmetric and alternating groups. Following \cite{MatteBon},
we call $\mathsf{S}(\mathcal{G})$ and $\mathsf{A}(\mathcal{G})$
the symmetric and alternating full group of $\mathcal{G}$ respectively.
\cite[Theorem 1.1]{Nek-simple} states that if $\mathcal{G}$ is a\emph{
}minimal groupoid of germs with the space of units homeomorphic to
the Cantor set (where \emph{minimal} means every orbit is dense),
then the alternating full group $\mathsf{A}(\mathcal{G})$ is simple.
In \cite{Nek-palindromic}, first examples of simple groups of intermediate
volume growth are constructed, these groups are alternating full groups
of certain fragmentation of dihedral group actions on Cantor sets.
To the best of our knowledge, so far in all known examples, $\mathsf{A}(\mathcal{G})$
coincides with the derived subgroup of $\mathsf{F}(\mathcal{G})$. 

The IRSs of the infinite finitary permutation group ${\rm Sym}_{f}(\mathbb{N})$
are classified in \cite{Vershik2}. In a similar way one can classify
IRSs of the infinite finitary alternating group ${\rm Alt}_{f}(\mathbb{N})$,
see \cite{Thomas-Tucker-Drob2}. The alternating group ${\rm Alt}_{f}(\mathbb{N})$
is an example of so called finitary locally finite simple groups.
Another important class of locally finite simple groups are the inductive
limits of direct product of alternating groups with respect to block
diagonal embeddings (called \emph{LDA-groups}), see \cite{LeinenPuglisi,LN}.
Indecomposable characters and IRSs of LDA-groups are classified in
\cite{Dudko-Medynets1,Dudko-Medynets3}. Independently, IRSs of inductive
limits of finite alternating groups are investigated in \cite{Thomas-Tucker-Drob1,Thomas-Tucker-Drob2}. 

An LDA-group can be identified as the alternating full group of its
associated AF-groupoid. We are interested in IRSs of the topological
full group $\mathsf{F}(\mathcal{G})$ and the alternating full group
$\mathsf{A}(\mathcal{G})$ for more general minimal groupoids of germs.
Combined with algebraic properties of the alternating full group $\mathsf{A}(\mathcal{G})$,
we have the following consequence of Theorem \ref{inclu-rigid}. Definitions
of a groupoid of germs $\mathcal{G}$, the topological full group
$\mathsf{F}(\mathcal{G})$ and the alternating full group $\mathsf{A}(\mathcal{G})$
are recalled in Subsection \ref{subsec:full-def}. 

\begin{theorem}\label{full-1}

Let $\mathcal{G}$ be a minimal groupoid of germs with unit space
$X=\mathcal{G}^{(0)}$ homeomorphic to the Cantor set. Suppose $\Gamma$
is a group such that $\mathsf{A}(\mathcal{G})\le\Gamma\le\mathsf{F}(\mathcal{G}).$
Let $\mu$ be an ergodic IRS of $\Gamma$. Then $\mu$-a.e. subgroup
$H$ contains the rigid stabilizer in $\mathsf{A}(\mathcal{G})$ of
the complement of its fixed point set, that is
\[
H\ge R_{\mathsf{A}(\mathcal{G})}\left(X\setminus{\rm Fix}_{X}(H)\right),
\]
where ${\rm Fix}_{X}(H)=\{x\in X:\ x\cdot h=x\mbox{ for all }h\in H\}$. 

\end{theorem}

In the setting of Theorem \ref{full-1}, denote by $F(X)$ the space
of closed subsets of $X$ equipped with the Vietoris topology. It
is routine to check that the map $H\mapsto{\rm Fix}_{X}(H)$ is Borel
measurable. To classify ergodic IRSs of $\mathsf{A}(\mathcal{G})$
in this situation, the task is reduced to classify possible distributions
of ${\rm Fix}_{X}(H)$, which are ergodic $G$-invariant Borel probability
measures on the space $F(X)$. In particular, if the only $\mathsf{A}(\mathcal{G})$-invariant
ergodic measures on $F(X)$ are the $\delta$-mass at the empty set
$\emptyset$ or the full set $X$, then the only IRSs of $\mathsf{A}(\mathcal{G})$
are the trivial ones: $\delta_{\{id\}}$ and $\delta_{\mathsf{A}(\mathcal{G})}$. 

In what follows we consider examples where $\mathsf{A}(\mathcal{G})$
admits invariant probability measures on $F(X)$ other than $\delta_{\emptyset}$
and $\delta_{X}$. In the special case where $\mathsf{A}(\mathcal{G})$
contains an LDA-subgroup which acts minimally on $X$, one can gain
information on such invariant measures with the help of the pointwise
ergodic theorem for inductive limit of finite groups from \cite{Vershik74,Olshanski-Vershik},
see Corollary \ref{full}. The idea of applying pointwise ergodic
theorem to study IRSs of the inductive limit of finite groups appeared
in \cite{Thomas-Tucker-Drob1,Thomas-Tucker-Drob2}.

In the case of the topological full group $\Gamma$ of a minimal $\mathbb{Z}^{d}$-Cantor
system, we can apply results on construction of Bratteli diagrams
for minimal $\mathbb{Z}^{d}$ actions from \cite{Forrest} to transfer
information on ergodic invariant measures on $F(X)$ under action
of certain LDA subgroup back to $\Gamma$. We obtain the following
classification:

\begin{corollary}\label{classificationZd} 

Let $\Gamma$ be the topological full group of a minimal action of
$\mathbb{Z}^{d}$ on the Cantor set $X$. The list of ergodic IRSs
of the derived subgroup $\Gamma'=[\Gamma,\Gamma]$ is 
\begin{description}
\item [{(i)}] (atomic ones) $\delta_{\{id\}}$, $\delta_{\Gamma'}$. 
\item [{(ii)}] (non-atomic ones) pushforward under the map 
\begin{align*}
X^{k} & \to{\rm Sub}(\Gamma')\\
(x_{1},,\ldots x_{k}) & \mapsto\cap_{i=1}^{k}{\rm St}_{\Gamma'}(x_{i})
\end{align*}
of measure $\mu_{1}\times\ldots\times\mu_{k}$ on $X^{k}$, $k\in\mathbb{N}$,
where each $\mu_{i}$ is an ergodic $\mathbb{Z}^{d}$-invariant measure
on $X$. 
\end{description}
\end{corollary}

Corollary \ref{classificationZd} extends classification results on
IRSs of LDA-groups established in the works \cite{Dudko-Medynets1,Dudko-Medynets3,Thomas-Tucker-Drob1,Thomas-Tucker-Drob2}
to minimal $\mathbb{Z}^{d}$-Cantor systems. Such a classification
shows that the alternating full group of a minimal $\mathbb{Z}^{d}$-action
on the Cantor set $X$ does not admit non-trivial IRSs other than
these stabilizer IRSs of diagonal actions on $X^{k}$, $k\in\mathbb{N}$.
In particular, although $\Gamma'$ admits an infinite collection of
non-atomic IRSs, it does not have a \textquotedbl zoo\textquotedbl{}
of IRSs like non-abelian free groups \cite{Bowen2} or the lamplighter
group \cite{BGR}. As indicated above, ingredients that go into the
proof of Corollary \ref{classificationZd} are: results from \cite{Forrest},
pointwise ergodic theorem for locally finite groups and Theorem \ref{full-1},
which in turn relies on Theorem \ref{inclu-rigid} and general properties
of $\mathsf{A}(\mathcal{G})$ as shown in \cite{Nek-simple}. 

Examples of topological full groups of minimal $\mathbb{Z}^{d}$-Cantor
systems include groups of interval exchange transformations, see \cite[Subsection 5.3]{JMMS};
and groups associated with the Penrose tilings introduced in \cite{CKN}.
As cited earlier, by \cite{JuschenkoMonod}, the topological full
group of a minimal $\mathbb{Z}$-Cantor system is amenable. For minimal
$\mathbb{\mathbb{Z}}^{2}$-actions, the topological full group can
be non-amenable: examples which contain non-abelian free groups are
constructed in \cite{ElekMonod}. Examples of non-amenable simple
groups with non-atomic IRSs are considered in \cite{Juschenko-Golan}.

The notion of a \emph{uniformly recurrent subgroup} (URS) is introduced
in \cite{GlasnerWeiss} as a topological analogue of invariant random
subgroups. A URS of a countable group $G$ is a minimal, conjugation
invariant, closed subset of ${\rm Sub}(G)$. Theorem \ref{inclu-rigid}
can be compared to results on URS in \cite[Theorem 3.10]{LBMB} and
\cite[Theorem 6.1]{MatteBon}. Note that we don't impose additional
assumptions on the action $G\curvearrowright X$ in part (i) of Theorem
\ref{inclu-rigid}. For URS, it is shown in \cite{MatteBon} that
for any minimal groupoid of germs $\mathcal{G}$, $\mathsf{A}(\mathcal{G})$
admit a unique URS, namely the stabilizer URS of $\mathsf{A}(\mathcal{G})\curvearrowright X=\mathcal{G}^{(0)}$.
In contrast, if $\mathcal{G}$ admits ergodic invariant measures on
$X$, then we have an infinite collection of IRSs, including the stabilizer
IRSs of diagonal actions on $X^{k}$. It seems to be an interesting
question whether conclusion of Corollary \ref{classificationZd} is
true for alternating full groups of general minimal groupoids of germs. 

\subsection*{Applications to weakly branch groups\label{subsec:branch-intro}}

We now turn to groups acting on spherically symmetric rooted trees.
While the alternating full groups discussed in the previous paragraphs
are infinite simple groups, groups acting faithfully on a rooted tree
$\mathsf{T}$ by automorphisms are residually finite. 

To state the consequences of Theorem \ref{inclu-rigid} for groups
acting on rooted trees, we introduce necessary terminology and notations
following references \cite{Gri-justinfinite,handbook}. Let $\mathbf{d}=(d_{j})_{j\in\mathbb{N}}$
be a sequence of integers, $d_{j}\ge2$ for all $j\in\mathbb{N}$.
The spherically symmetric rooted tree $\mathsf{T}=\mathsf{T}_{\mathbf{d}}$
with valency sequence ${\bf d}$ is the tree with vertices $v=v_{1}\ldots v_{n}$
where each $v_{j}\in\{0,1,\ldots,d_{j}-1\}$. The root of the tree
is denoted by the empty sequence $\emptyset$. Edge set of the tree
is $\left\{ \left(v_{1}\ldots v_{n},v_{1}\ldots v_{n}v_{n+1}\right)\right\} $.
The index $n$ is called the depth or level of $v$, denoted $\left|v\right|=n$.
Denote by $\mathsf{T}_{\mathbf{d}}^{n}$ the finite subtree of vertices
up to depth $n$ and $\mathsf{L}_{n}$ the vertices of level $n$.
The boundary $\partial\mathsf{T}_{\mathbf{d}}$ of the tree $\mathsf{T}_{\mathbf{d}}$
is the set of infinite rays $x=v_{1}v_{2}\ldots$ with $v_{j}\in\{0,1,\ldots,d_{j}-1\}$
for each $j\in\mathbb{N}$. The action of $G$ on the tree $\mathsf{T}$
extends to the boundary $\partial\mathsf{T}$. 

For each vertex $x\in\mathsf{T}$, denote by $\mathsf{T}_{x}$ the
subtree rooted at $x$ and $C_{x}$ the cylinder set in $\partial\mathsf{T}$
which consists of infinite rays with prefix $x$. We follow the terminology
in the theory of branch groups and write 
\[
{\rm Rist}_{\Gamma}(u):=R_{\Gamma}(C_{u}).
\]
That is, for $u\in\mathsf{T}$, the rigid stabilizer $R_{\Gamma}(C_{u})$
of the cylinder set $C_{u}$ is called the \emph{rigid vertex stabilizer}
of $u$ in $\Gamma$. 

Given a subtree $\mathsf{T}_{x}$ rooted at $x$ and number $m$,
denote by ${\rm Rist}_{m}(\mathsf{T}_{x})$ the level $m$ rigid stabilizer
of this subtree, that is 
\[
{\rm Rist}_{m}^{\Gamma}(\mathsf{T}_{x}):=\prod_{1\le j\le m,\ u_{j}\in\{0,\ldots,d_{|x|+j}-1\}}{\rm Rist}_{\Gamma}(xu_{1}\ldots u_{m}).
\]

A group $G$ acting on the rooted tree $\mathsf{T}$ is said to be
\emph{weakly branching} if it acts level transitively and the rigid
stabilizer ${\rm Rist}_{\Gamma}(u)$ is non-trivial for any vertex
$u\in\mathsf{T}$. It is said to be a \emph{branch group} if in addition
all the level rigid stabilizers ${\rm Rist}_{m}(\mathsf{T})$ have
finite index in $\Gamma$, $m\in\mathbb{N}$. These notions are introduced
by Grigorchuk in \cite{Gri-justinfinite}. 

Given a closed subset $K$ of $\partial\mathsf{T}$, we associate
to it the following index set $I_{K}\subseteq\mathsf{T}$. The complement
$\partial\mathsf{T}\setminus K$ can be written uniquely as a disjoint
union of cylinder sets $\cup_{x\in I_{K}}C_{x}$, where the cylinder
sets are maximal in the sense that if $x=x_{1}\ldots x_{\ell}$ is
in the collection $I_{K}$, then there exists a sibling $x'=x_{1}\ldots x_{\ell-1}x_{\ell}'$,
$x_{\ell}'\neq x_{\ell}$, such that $x'\notin I_{K}$. For example,
if $\mathsf{T}$ is the rooted binary tree and $K=\{1^{\infty}\}$,
where $1^{\infty}$ is the right most ray, then $I_{\{1^{\infty}\}}=\{0,10,\ldots,1^{n-1}0,\ldots\}$.
The following corollary is a direct consequence of Theorem \ref{inclu-rigid}:

\begin{corollary}\label{branch}

Let $\Gamma$ be a countable weakly branch group acting faithfully
on a rooted spherically symmetric tree $\mathsf{T}$. Let $\mu$ be
an ergodic IRS of $\Gamma$. Then there exist numbers $(m_{i})$,
each $m_{i}\in\mathbb{N}$, such that $\mu$-a.e. subgroup $H$ satisfies
\[
H\ge\bigoplus{}_{x\in I_{{\rm Fix}(H)}}\left[{\rm Rist}_{m_{|x|}}^{\Gamma}(\mathsf{T}_{x}),{\rm Rist}_{m_{|x|}}^{\Gamma}(\mathsf{T}_{x})\right],
\]
where $I_{{\rm Fix}(H)}$ is the set of vertices in $\mathsf{T}$
associated with ${\rm Fix}(H)={\rm Fix}_{\partial\mathsf{T}}(H)$
as described above.

\end{corollary}

The special case of Corollary \ref{branch} where $\Gamma$ is a finitary
regular branch group was obtained in \cite{Bencs-Toth}, where finitary
means that $G$ is a subgroup of ${\rm Aut}_{f}(\mathsf{T})$. Note
that ${\rm Aut}_{f}(\mathsf{T})$ is locally finite. Our approach
is independent of \cite{Bencs-Toth} and applies to more general settings.
For branch groups one can draw the following immediate consequence
(Corollary \ref{branch-clopen}): if $\mu$ is an ergodic IRS of a
branch group $\Gamma$ such that $\mu$-a.e. ${\rm Fix}(H)$ is clopen
(possibly empty), then $\mu$ is atomic. The same conclusion holds
for the Basilica group, which is weakly branching but not branching,
see Corollary \ref{basilica}. 

In particular, an ergodic fixed point free IRS of a just-infinite
branch group $\Gamma$ is supported on finite index subgroups of $\Gamma$.
The most celebrated example of a just infinite branch group is the
first Grigorchuk group $\mathfrak{G}$ which is defined in \cite{Grigorchuk1980},
see \cite{Gri-justinfinite} for more information on just-infiniteness.
The group $\mathfrak{G}$ has many remarkable properties, in particular,
it is torsion (\cite{Grigorchuk1980}) and shown in \cite{Girgorchuk1984}
to be the first example of groups of intermediate volume growth. 

\textcolor{black}{The induced distribution of ${\rm Fix}(H)$ is described
in \cite{Bencs-Toth}. Namely, by \cite[Lemma 2.2, 2.3]{Bencs-Toth},
any $\Gamma$-invariant ergodic probability measure on $F(\partial\mathsf{T})$
arises as translation of a fixed closed set $K$ by a random element
in the profinite completion $\bar{\Gamma}$, with respect to the Haar
measure on $\bar{\Gamma}$. We have seen that for branch groups, ${\rm Fix}(H)$
being clopen corresponds to atomic IRSs. The situation where ${\rm Fix}(H)$
is a closed but not clopen subset of $\partial\mathsf{T}$ is more
complicated. Atomless IRSs of a weakly branch group with ${\rm Fix}(H)$
being a finite subset of $\partial\mathsf{T}$ are considered in \cite{Dudko-Grigorchuk2018}.
Corollary \ref{branch} provides information that conditioned on the
fixed point set ${\rm Fix}(H)$ on $\partial\mathsf{T}$, the subgroup
$H$ must contain derived subgroups of certain level rigid stabilizers.
In other words, the conditional distribution of $H$ given ${\rm Fix}(H)=K$
is pulled back from an IRS of the quotient 
\[
\bar{\Gamma}_{K,{\bf m}}:={\rm Fix}_{G}(K)/\oplus_{x\in I_{K}}\left[{\rm Rist}_{m_{|x|}}^{\Gamma}(\mathsf{T}_{x}),{\rm Rist}_{m_{|x|}}^{\Gamma}(\mathsf{T}_{x})\right],
\]
where ${\rm Fix}_{\Gamma}(K)=\{g\in G:\ x\cdot g=x\mbox{ for any }x\in K\}$
is the pointwise stabilizer of the set $K$ in $\Gamma$. It is natural
to ask in some specific examples, for instance the first Grigorchuk
group $\mathfrak{G}$, whether one can completely classify the IRSs.
For the group $\mathfrak{G}$, when $K$ is closed but not clopen,
such a quotient $\bar{\Gamma}_{K,{\bf m}}$ is a locally finite infinite
group which admits a continuum of IRSs, see more discussion in Section
\ref{sec:rootedtree}. }We summarize the discussion above for the
first Grigorchuk group $\mathfrak{G}$ as follows, answering positively
\cite[Problem 8]{Bencs-Toth}:

\begin{example}\label{firstgri}

Let $\mu$ be an ergodic invariant random subgroups of the first Grigorchuk
group $\mathfrak{\mathfrak{G}}$. Then it falls into one of the following
types:
\begin{description}
\item [{(i)}] Fixed point free IRSs. In this case $\mu$ is supported on
finite index subgroups of $\mathfrak{\mathfrak{\mathfrak{G}}}$; equivalently
there exists a constant $m\in\mathbb{N}$ such that $\mu$-a.e. $H$
contains the level stabilizers ${\rm St}_{\mathfrak{G}}(m)$. 
\item [{(ii)}] IRSs with non-empty clopen fixed point sets. In this case
$\mu$ is atomic and it arises in the following way. There exists
a non-empty clopen subset $C\subseteq\partial\mathsf{T}$ and a finite
index subgroup $\Gamma$ of ${\rm Fix}_{\mathfrak{G}}(C)$ such that
$\mu$ is the uniform measure on $\mathfrak{G}$-conjugates of $\Gamma$. 
\item [{(iii)}] Non-atomic IRSs. In this case there exists a closed but
not clopen subset $C\subseteq\partial\mathsf{T}$ and a sequence of
integers ${\bf m}$ such that the fixed point set ${\rm Fix}(H)$
is a random Haar translate of $C$ and $\mu$-a.e. $H$ contains the
infinite direct sum $\bigoplus_{x\in I_{{\rm Fix}(H)}}\left[{\rm Rist}_{m_{|x|}}(\mathsf{T}_{x}),{\rm Rist}_{m_{|x|}}(\mathsf{T}_{x})\right]$. 
\end{description}
\end{example}

Based on Corollary \ref{branch}, in Theorem \ref{cosofic-1} we show
that under certain additional assumptions (contracting and the bounded
activity assumption), all IRSs of a just-infinite branch group are
co-sofic. The conditions in Theorem \ref{cosofic-1} are verified
by just-infinite branch groups generated by finitely many \emph{bounded
automatic automorphisms} of a regular rooted tree (for a description
of such groups see Subsection \ref{subsec:C-examples}). In particular
we show:

\begin{theorem}\label{P-stable}

All IRSs of the Grigorchuk group $\mathfrak{G}$ are co-sofic; the
same conclusion holds for the Gupta-Sidki $p$-groups.

\end{theorem}

Combined with \cite[Theorem 1.3]{BLT19}, we deduce that the Grigorchuk
group and the Gupta-Sidki $p$-groups are P-stable. These are first
examples of non-elementary amenable P-stable groups. The question
whether the Grigorchuk group $\mathfrak{G}$ is P-stable is asked
in \cite{BLT19}. Theorem \ref{cosofic-1} applies to groups beyond
the class of just-infinite branch groups generated by finitely many
bounded automatic automorphisms. For example, the Grigorchuk groups
$G_{\omega}$, where $\omega$ is in a certain uncountable subset
of $\{0,1,2\}^{\infty}$, satisfies the conditions of Theorem \ref{cosofic-1},
see Corollary \ref{omega}. In a similar way, Theorem \ref{cosofic-1}
can be applied to groups in an uncountable sub-collection of amenable
groups of non-uniform exponential growth constructed in \cite{Bri09},
see Corollary \ref{non-uni}. Uncountably many P-stable amenable groups
are first given in \cite{LL2}. 

\subsection*{Organization of the paper}

The rest of the paper is organized as follows. In Section \ref{sec:conditional}
we consider distributions of certain random collections of partial
homeomorphisms induced by the invariant random subgroup. In Section
\ref{sec:rigidstabilizer} we prove the double commutator lemma for
IRSs as stated in Theorem \ref{inclu-rigid}. Section \ref{sec:Preliminaries}
collects necessary definitions and basic properties of topological
full groups. In Section \ref{sec:full groups} we show Theorem \ref{full-1}
on IRSs of topological full groups and consider minimal $\mathbb{Z}^{d}$-Cantor
systems as main examples. Section \ref{sec:rootedtree} on IRSs of
weakly branch groups is independent of Section \ref{sec:Preliminaries}
and \ref{sec:full groups} and can be understood right after Section
\ref{sec:rigidstabilizer}. 

\subsection*{Acknowledgment}

We thank Rostislav Grigorchuk and Kate Juschenko for their valuable
comments on topological full groups. 

\section{Restrictions and conditional distributions\label{sec:conditional}}

\subsection{Regular conditional distributions\label{subsec:reg}}

We follow notations of regular conditional distributions in the book
\cite[Chapter V.8]{Parthasarathy}. Let $(X,\mathcal{B})$, $(Y,\mathcal{C})$
be two Borel spaces, $\mathbb{P}$ a probability measure on $\mathcal{B}$
and $\pi:X\to Y$ a measurable map. Let $\mathbb{Q}=\mathbb{P}\circ\pi^{-1}$
be probability measure on $\mathcal{C}$ which is the pushforward
of $\mathbb{P}$. A \emph{regular conditional distribution} given
$\pi$ is a mapping $y\mapsto\mathbb{P}(y,\cdot)$ such that 

(i) for each $y\in Y$, $\mathbb{P}(y,\cdot)$ is a probability measure
on $\mathcal{B}$;

(ii) there exists a set $N\in\mathcal{C}$ such that $\mathbb{Q}(N)=0$
and for each $y\in Y\setminus N$, $\mathbb{P}(y,X\setminus\pi^{-1}(\{y\}))=0$;

(iii) for any $A\in\mathcal{B}$, the map $y\mapsto\mathbb{P}(y,A)$
is $\mathcal{C}$-measurable and 
\[
\mathbb{P}(A)=\int_{Y}\mathbb{P}(y,A)d\mathbb{Q}(y).
\]
We will refer to these three items as properties (i),(ii),(iii) of
a regular conditional distribution. 

Recall that a measure space $(X,\mathcal{B})$ is called a standard
Borel space if it is isomorphic to some Polish space equipped with
the Borel $\sigma$-field. It is classical that if $(X,\mathcal{B})$
and $(Y,\mathcal{C})$ are standard Borel spaces and $\pi:X\to Y$
is measurable, then there exists such a regular conditional distribution
$y\mapsto\mathbb{P}(y,\cdot)$ with properties (i),(ii),(iii); and
moreover it is unique: if $\mathbb{P}'(y,\cdot)$ is another such
mapping, then $\{y:\ \mathbb{P}'(y,\cdot)\neq\mathbb{P}(y,\cdot)\}$
is a set of $\mathbb{Q}$-measure $0$, see \cite[Theorem 8.1]{Parthasarathy}.

In our setting $G$ is a countable group, the Chabauty topology on
${\rm Sub}(G)$ is restriction of the product topology on $\{0,1\}^{G}$
to the closed subset ${\rm Sub}(G)$. The space $\left({\rm Sub}(G),\mathcal{B}\right)$,
where $\mathcal{B}$ is the Borel $\sigma$-field on ${\rm Sub}(G)$,
is a standard Borel space.

\subsection{Restrictions to open subsets}

In this and the next section, we impose:

\begin{assumption}[Standing assumption]\label{standing}

Suppose $G$ is a countable group acting faithfully on a second countable
Hausdorff space $X$ by homeomorphisms. Denote by $\mathcal{U}$ a
countable base of topology of $X$.

\end{assumption}

Let $G\curvearrowright X$ as in Assumption \ref{standing}. Let $U$,
$V$ be two open subsets of $X$ such that there exists some $g\in G$
with $V=U\cdot g$. Given such open sets and a subgroup $H\in{\rm Sub}(G)$,
define the following:
\begin{equation}
H_{U\to V}:=\left\{ h\in H:\ V=U\cdot h\right\} .\label{eq:SHU}
\end{equation}
And define the restrictions 
\begin{equation}
\bar{H}_{U\to V}:=\left\{ h|_{U}:\ h\in H_{U\to V}\right\} .\label{eq:BSHU}
\end{equation}
Elements of $\bar{H}_{U\to V}$ are viewed as partial homeomorphisms
with domain $U$ and range $V$, denoted by $h|_{U}:U\to V$. By definition
$H_{U\to U}$ is the subgroup of $H$ which consists of elements that
leaves $U$ invariant, in other words the setwise stabilizer of $U$
in $H$. The group $\bar{H}_{U\to U}$ is a quotient of $H_{U\to U}$
which acts on $U$ by homeomorphisms. The following fact will be used
repeatedly:

\begin{fact}\label{coset}

The set $H_{U\to V}$ is either empty or a right coset of $H_{U\to U}$.
That is, if $H_{U\to V}$ is non-empty, then for any element $h\in H_{U\to V}$,
\[
H_{U\to V}=H_{U\to U}h.
\]
Similarly, if $\bar{H}_{U\to V}$ is non-empty, then $\bar{H}_{U\to V}=\bar{H}_{U\to U}h|_{U}$
for any $h|_{U}\in\bar{H}_{U\to V}$. 

\end{fact}

\begin{proof}

From definitions it is clear that $H_{U\to U}h\subseteq H_{U\to V}$
for any $h\in H_{U\to V}$. In the other direction, given any element
$h'\in H_{U\to V}$, we have $U\cdot h'h^{-1}=V\cdot h^{-1}=\left(U\cdot h\right)h^{-1}=U$.
Therefore $h'h^{-1}\in H_{U\to U}$, which implies $H_{U\to V}=H_{U\to U}h$.
The same argument shows that $\bar{H}_{U\to V}=\bar{H}_{U\to U}h|_{U}$. 

\end{proof}

Let $U,V\in\mathcal{U}$ be such that $U\cap V=\emptyset$ and $G_{U\to V}\neq\emptyset$.
Due to Fact \ref{coset}, given $\bar{H}_{U\to U}$, the set $\bar{H}_{U\to V}$
can only take value in a countable collection, namely $\emptyset$
and right cosets of the form $\bar{H}_{U\to U}\gamma|_{U}$. Let $\Omega_{U,V}=\{H\in{\rm Sub}(G):\ H\cap G_{U\to V}\neq\emptyset\}$
be the set of subgroups of $G$ which contains some element that sends
$U$ to $V$. Note that $\Omega_{U,V}$ is an open set in ${\rm Sub}(G)$. 

Let $\mu$ be an IRS of $G$. Consider a pair $U,V$ such that $\mu(\Omega_{U,V})>0$.
Define $\mu_{U,V}$ as the probability measure on Borel subsets of
$\Omega_{U,V}$
\begin{align*}
\mu_{U,V} & :\mathcal{B}(\Omega_{U,V})\to[0,1]\\
\mu_{U,V}(A) & =\frac{\mu(A)}{\mu(\Omega_{U,V})},\ A\subseteq\Omega_{U,V}.
\end{align*}
Let ${\rm C}_{U,V}:=\cup_{\gamma\in G_{U\to V}}{\rm Sub}\left(\bar{G}_{U\to U}\right)\gamma|_{U}$
be the union of cosets of subgroups of $\bar{G}_{U\to U}$, represented
by restrictions to $U$ of elements $\gamma\in G_{U\to V}$. Let $\mathcal{U}_{U}$
be a base for the Chabauty topology on ${\rm Sub}(\bar{G}_{U\to U})$.
Equip ${\rm C}_{U,V}$ with the topology generated by the base $\mathcal{U}_{U,V}=\cup_{\gamma\in G_{U\to V}}\left\{ O\gamma|_{U}:\ O\in\mathcal{U}_{U}\right\} $.
Define $\tilde{{\rm C}}_{U,V}$ as the subspace of ${\rm C}_{U,V}\times{\rm Sub}\left(\bar{G}_{U\to U}\right)$
which consists of pairs
\[
\tilde{{\rm C}}_{U,V}:=\left\{ (C,A)\in{\rm C}_{U,V}\times{\rm Sub}\left(\bar{G}_{U\to U}\right):\ C=A\gamma|_{U}\mbox{ for some }\gamma\in G_{U\to V}\right\} .
\]
It's clear by definition that ${\rm \tilde{C}}_{U,V}$ equipped with
the Borel $\sigma$-field is a standard Borel space. By Fact \ref{coset},
the map $H\mapsto\left(\bar{H}_{U\to V},\bar{H}_{U\to U}\right)$
is a map from $\Omega_{U,V}$ to $\tilde{{\rm C}}_{U,V}$. It is routine
to check that this map is measurable. 

Let $\pi:\tilde{{\rm C}}_{U,V}\to{\rm Sub}\left(\bar{G}_{U\to U}\right)$
be the projection to the second coordinate, that is $\pi(C,A)=A$.
Denote by $\mathbb{P}_{U,V}^{\mu}:{\rm Sub}\left(\bar{G}_{U\to U}\right)\times\mathcal{B}\left({\rm \tilde{C}}_{U,V}\right)\to[0,1]$
the regular conditional distribution of $\left(\bar{H}_{U\to V},\bar{H}_{U\to U}\right)$
given $\pi$, where the distribution of $\left(\bar{H}_{U\to V},\bar{H}_{U\to U}\right)$
is the pushforward of $\mu_{U,V}$ under this map $H\mapsto\left(\bar{H}_{U\to V},\bar{H}_{U\to U}\right)$.
Recall that such regular conditional distribution exists since $\pi$
is a measurable map between standard Borel spaces. 

We introduce one more piece of notation that will be also be used
in the next section. For $U,V$ open sets of $X$, let $W_{V}^{U}$
be the subgroup of $G$ which consists of elements that fix $U$ pointwise
and leaves $V$ invariant, that is 
\begin{equation}
W_{V}^{U}:=\{g\in G:\ x\cdot g=x\mbox{ for any }x\in U\mbox{ and }V=V\cdot g\}.\label{eq:WUV}
\end{equation}
Note that the set $\Omega_{U,V}$ is invariant under conjugation by
$W_{V}^{U}$. Since $\mu$ is invariant under conjugation by $G$,
it follows that $\mu_{U,V}$ is invariant under conjugation by $W_{V}^{U}$.

Let $\gamma\in W_{V}^{U}$. Then for such $\gamma$, it is clear from
definitions that under conjugation by such an element $\gamma$, $\overline{\left(\gamma^{-1}H\gamma\right)}_{U\to U}$
remains the same as $\bar{H}_{U\to U}$ since $\gamma$ acts trivially
on $U$; while $\overline{\left(\gamma^{-1}H\gamma\right)}_{U\to V}$
is the right translate of $\bar{H}_{U\to V}$ by the restriction of
$\gamma$ to $V$. That is, for $\gamma\in W_{V}^{U}$,
\begin{align*}
\overline{\left(\gamma^{-1}H\gamma\right)}_{U\to U} & =\bar{H}_{U\to U},\\
\overline{\left(\gamma^{-1}H\gamma\right)}_{U\to V} & =\bar{H}_{U\to V}\gamma|_{V}
\end{align*}
The assumption that $\mu$ is invariant under conjugation implies
translation invariance properties of $\mathbb{P}_{U,V}^{\mu}(\bar{H}_{U\to U},\cdot)$
as stated in the following lemma.

\begin{lemma}\label{invariant}

Let $\mu$ be an IRS of $G$. Let $U,V$ be open subsets of $X$ such
that $\mu(\Omega_{U,V})>0$. Then for $\mu$-a.e. $H\in\Omega_{U,V}$,
$\mathbb{P}_{U,V}^{\mu}(\bar{H}_{U\to U},\cdot)$ is a probability
measure supported on $\left\{ \left(\bar{H}_{U\to U}g|_{U},\bar{H}_{U\to U}\right):g\in G_{U\to V}\right\} $
and $\mathbb{P}_{U,V}^{\mu}(\bar{H}_{U\to U},\cdot)$ is invariant
under right multiplication of restrictions $\gamma|_{V}$ for $\gamma\in W_{V}^{U}$:
for any $g\in G_{U\to V}$ and $\gamma\in W_{V}^{U}$,
\[
\mathbb{P}_{U,V}^{\mu}\left(\bar{H}_{U\to U},\left\{ \left(\bar{H}_{U\to U}g|_{U},\bar{H}_{U\to U}\right)\right\} \right)=\mathbb{P}_{U,V}^{\mu}\left(\bar{H}_{U\to U},\left\{ \left(\bar{H}_{U\to U}g|_{U}\gamma|_{V},\bar{H}_{U\to U}\right)\right\} \right).
\]

\end{lemma}

\begin{proof}

The first half of the claim follows directly from property (i) and
(ii) of the regular conditional distribution $\mathbb{P}_{U,V}^{\mu}$
as reviewed in Subsection \ref{subsec:reg}.

To show the second half, note that as a consequence of $\mu_{U,V}$
being invariant under conjugation by $\gamma$, $\gamma\in W_{V}^{U}$,
we have that $\left(\bar{H}_{U\to V},\bar{H}_{U\to U}\right)$ and
$\left(\bar{H}_{U\to V}\gamma|_{V},\bar{H}_{U\to U}\right)$ have
the same distribution. Indeed, for any measurable bounded function
$f:\tilde{{\rm C}}_{U,V}\to\mathbb{R}$, we have
\begin{align*}
 & \mathbb{E}_{\mu_{U,V}}\left[f\left(\bar{H}_{U\to V},\bar{H}_{U\to U}\right)\right]\\
= & \mathbb{E}_{\mu_{U,V}}\left[f\left(\overline{\left(\gamma^{-1}H\gamma\right)}_{U\to U},\overline{\left(\gamma^{-1}H\gamma\right)}_{U\to U}\right)\right]\\
= & \mathbb{E}_{\mu_{U,V}}\left[\left(f\left(\bar{H}_{U\to V}\gamma|_{V},\bar{H}_{U\to U}\right)\right)\right].
\end{align*}
In the last line the identities $\overline{\left(\gamma^{-1}H\gamma\right)}_{U\to U}=\bar{H}_{U\to U}$
and $\overline{\left(\gamma^{-1}H\gamma\right)}_{U\to V}=\bar{H}_{U\to V}\gamma|_{V}$
explained above are used. The claim follows from uniqueness of regular
conditional distribution. 

\end{proof}

\section{Properties of IRS in connection to rigid stabilizers\label{sec:rigidstabilizer}}

In this section we prove Theorem \ref{inclu-rigid}. As explained
in Lemma \ref{invariant}, conjugation invariance of the IRS distribution
$\mu$ results in translation invariance properties of certain conditional
distributions. We will repeatedly use the fact that a countable orbit
supporting an invariant probability measure must be finite. 

\subsection{From IRS to rigid stabilizers}

Recall the notation $H_{U\to V}$ and $\bar{H}_{U\to V}$ as defined
in (\ref{eq:SHU}) and (\ref{eq:BSHU}). Recall that $W_{V}^{U}$,
as defined in (\ref{eq:WUV}), is the subgroup of $G$ which consists
of elements that fix $U$ pointwise and leave $V$ invariant. For
a subgroup $H<G$, denote by $N_{G}(H)$ its normalizer in $G$, that
is $N_{G}(H)=\{g\in G:g^{-1}Hg=H\}$. A group is called an FC-group
if all of its conjugacy classes are finite. 

\begin{proposition}\label{conjugate1}

Let $G\curvearrowright X$ be as in Assumption \ref{standing}. Let
$\mu$ be an IRS of $G$. Then $\mu$-a.e. $H$ satisfies the following
property: if $U,V\in\mathcal{U}$ are such that $U\cap V=\emptyset$
and $H\cap G_{U\to V}\neq\emptyset$, then:
\begin{description}
\item [{(i)}] There exists an element $\sigma\in G_{U\to V}$ and a finite
index subgroup $\Gamma$ of $W_{U}^{U\cdot\sigma^{-1}}$ such that
$\pi_{U}(\Gamma)\le\bar{H}_{U\to U}$, where $\pi_{U}$ is the projection
$G_{U\to U}\to\bar{G}_{U\to U}$. 
\item [{(ii)}] The subgroup $K=\left\{ g\in N_{R_{G}(U)}(R_{H}(U)):\ \pi_{U}(g)\in\bar{H}_{U\to U}\right\} $
is of finite index in $R_{G}(U)$. Moreover, the quotient group $K/K\cap H$
is an FC-group.
\end{description}
\end{proposition}

To prove part (i) we consider regular conditional distribution of
$(\bar{H}_{U\to V},\bar{H}_{U\to U})$ given $\bar{H}_{U\to U}$.
The invariance property stated in Lemma \ref{invariant} forces the
number of cosets under consideration to be finite, in order to support
an invariant probability measure. 

\begin{proof}[Proof of Proposition \ref{conjugate1}(i)]

If $\mu$ is $\delta$-mass at $\{id\}$, then the claim is trivially
true. We may assume $\mu$ is not $\delta$-mass at $\{id\}$. In
what follows we use notations introduced in Section \ref{sec:conditional}.

Take a pair of $U,V$ such that $\mu(\Omega_{U,V})>0$ and consider
the random variables $H_{U\to V}$, $\bar{H}_{U\to V}$ and $\bar{H}_{U\to U}$
as defined in (\ref{eq:SHU}), (\ref{eq:BSHU}). Denote by $\mathbb{P}_{U,V}^{\mu}(\bar{H}_{U\to U},\cdot)$
the regular conditional distribution of $\left(\bar{H}_{U\to V},\bar{H}_{U\to U}\right)$
given $\bar{H}_{U\to U}$, where $H$ has distribution $\mu_{U,V}$
on $\Omega_{U,V}$. 

Recall that by Lemma \ref{invariant}, $\mathbb{P}_{U,V}^{\mu}(\bar{H}_{U\to U},\cdot)$
is a probability measure on a countable set. For $\mu$-a.e. $H\in\Omega_{U,V}$,
we can find one coset $\bar{H}_{U\to U}\sigma|_{U}$, $\sigma\in G_{U\to V}$
depending on $\bar{H}_{U\to U}$, such that $\mathbb{P}_{U,V}^{\mu}\left(\bar{H}_{U\to U},\left\{ \left(\bar{H}_{U\to U}\sigma|_{U},\bar{H}_{U\to U}\right)\right\} \right)>0$.
If the number of right cosets $\bar{H}_{U\to U}\sigma|_{U}\gamma|_{V}$,
where $\gamma$ is taken over elements of $W_{V}^{U}$, is infinite,
then the probability measure $\mathbb{P}_{U,V}^{\mu}(\bar{H}_{U\to U},\cdot)$
cannot be invariant under right multiplication as stated in Lemma
\ref{invariant}. Therefore there are only finitely many cosets of
$\bar{H}_{U\to U}\sigma|_{U}$ in this collection. In other words,
there are finitely many representatives $\gamma_{1},\ldots,\gamma_{\ell}$
in $W_{V}^{U}$ such that for any $\gamma\in W_{V}^{U}$, we have
$\bar{H}_{U\to U}\sigma|_{U}\gamma|_{V}=\bar{H}_{U\to U}\sigma|_{U}\gamma_{k}|_{V}$
for some $k\in\{1,\ldots,\ell\}$. It follows that for any $\gamma\in W_{V}^{U}$,
there is a representative $\gamma_{k}$, $k\in\{1,\ldots,\ell\}$,
such that $\bar{H}_{U\to U}$ contains $\sigma|_{U}(\gamma\gamma_{k}^{-1})|_{V}\sigma|_{U}^{-1}$.
Consider the subgroup $W_{1}$ of $W_{V}^{U}$ generated by the collection
$\gamma\gamma_{k}^{-1}$, where $\gamma\in W_{V}^{U}$ and $\gamma_{k}$
is its corresponding representative. It's clear by definition of $W_{1}$
that $\cup_{j=1}^{\ell}W_{1}\gamma_{j}=W_{V}^{U}$, therefore $W_{1}$
is a finite index subgroup of $W_{V}^{U}$. Recall that $\sigma$
maps $U$ to $V$, therefore $\sigma W_{V}^{U}\sigma^{-1}=W_{U}^{U\cdot\sigma^{-1}}$.
Let $\Gamma=\sigma W_{1}\sigma^{-1}$, it is a finite index subgroup
of $W_{U}^{U\cdot\sigma^{-1}}$. Elements of $\Gamma$ satisfy the
property that $\bar{H}_{U\to U}=\bar{H}_{U\to U}\gamma|_{U}$, in
other words, $\pi_{U}(\Gamma)\le\bar{H}_{U\to U}$. Let $\Omega_{U,V}'$
be the subset of $\Omega_{U,V}$ which consists of subgroups $H$
of $G$ such that statement (i) is satisfied for some $\sigma\in G_{U\to V}$
and some finite index subgroup $\Gamma\le_{f.i.}W_{U}^{U\cdot\sigma^{-1}}$.
We have proved that $\mu(\Omega_{U,V}')=\mu(\Omega_{U,V})$. 

Finally, take the union of the measure $0$ sets we want to discard.
Let $\Lambda_{0}=\left\{ (U,V)\in\mathcal{U}^{2}:\ \mu\left(\Omega_{U,V}\right)=0\right\} $
and $E=\left(\cup_{(U,V)\in\Lambda_{0}}\Omega_{U,V}\right)\cup\left(\cup_{(U,V)\notin\Lambda_{0}}\Omega_{U,V}\setminus\Omega_{U,V}'\right)$.
This is a countable union of $\mu$-measure $0$ sets, and ${\rm Sub}(G)-E$
gives a full measure set in the statement of part (i).

\end{proof}

We now turn to the proof of Proposition \ref{conjugate1} (ii). Let
$U$ be an open subset of $X$ and consider the subgroup $H_{U\to U}$
of $H$ as in (\ref{eq:SHU}). Elements of the subgroup $H_{U\to U}$
preserves the partition $U\sqcup U^{c}$ of $X$, thus $h\in H_{U\to U}$
can be recorded as a pair $(f_{1},f_{2})$, where $f_{1}=\pi_{U}\left(h\right)$
and $f_{2}=\pi_{U^{c}}\left(h\right)$. View $\bar{H}_{U\to U}=\pi_{U}(H_{U\to U})$
as a group of homeomorphisms on $U$, then the rigid stabilizer $R_{H}(U)$
is a normal subgroup of $\bar{H}_{U\to U}$. Note the following elementary
fact about the corresponding quotient group:

\begin{fact}\label{isomorphism}

There is an isomorphism $\phi:\pi_{U^{c}}(H_{U\to U})/R_{H}(U^{c})\to\pi_{U}(H_{U\to U})/R_{H}(U)$. 

\end{fact}

\begin{proof}

Write $L=H_{U\to U}$. Record elements of $L$ as $(f_{1},f_{2})$,
where $f_{1}\in\pi_{U}\left(L\right)$ and $f_{2}\in\pi_{U^{c}}\left(L\right)$.
Let $B_{L}(f_{2})=\{f_{1}\in\pi_{U}(L):\ (f_{1},f_{2})\in L\}$. Note
that $B_{L}(f_{2})$ is a right coset of $L\cap R_{G}(U)$. Moreover,
if $f_{2}$ and $f_{2}'$ are in the same coset of $L\cap R_{G}(U^{c})$
in $\pi_{U^{c}}\left(L\right)$, then $B_{L}(f_{2})=B_{L}(f_{2}')$.
Therefore we have a map $\phi:\pi_{U^{c}}(L)/L\cap G_{U^{c}}\to\pi_{U}(L)/L\cap G_{U}$
induced by $f_{2}\mapsto B_{L}(f_{2})$.

The map $\phi$ is a homomorphism because $L\cap R_{G}(U)$ is normal
in $\pi_{U}(L)$. If $B_{L}(f_{2})=B_{L}(f_{2}')$, then there exists
$f_{1}\in\pi_{U}(L)$ such that $(f_{1},f_{2})\in L$ and $(f_{1},f_{2}')\in L$.
It follows that $(id,f_{2}^{-1}f_{2}')=(f_{1},f_{2})^{-1}(f_{1},f_{2}')\in L$,
that is $f_{2}$ and $f_{2}'$ are in the same coset of $L\cap R_{G}(U^{c})$.
In other words $\phi$ is injective. For any $f_{1}\in\pi_{U}(L)$,
the preimage of $\bar{f}_{1}$ under $\phi$ is the coset $\{f_{2}\in\pi_{U^{c}}(L):\ (f_{1},f_{2})\in L\}$.
We conclude $\phi$ is an isomorphism. 

\end{proof}

By part (i) of the proposition we have that $\pi_{U}(\Gamma)$ is
a subgroup of $\bar{H}_{U\to U}$ for some $\Gamma\le_{f.i.}W_{U}^{U\cdot\sigma^{-1}}$
where $\sigma$ is some element in $G_{U\to V}$. Note that $U\cap U\cdot\sigma^{-1}=\emptyset$.
Our next step is to gain some information on the intersection $H\cap R_{G}(U)=R_{H}(U)$.
Note that $R_{H}(U)$ is a normal subgroup of $H_{U\to U}$, it can
also be viewed as a normal subgroup of $\bar{H}_{U\to U}$, where
$R_{H}(U)$ and $\bar{H}_{U\to U}$ are regarded as groups acting
by homomorphisms on $U$. We look for properties of the quotient group
$\bar{H}_{U\to U}/R_{H}(U)$ that can be derived from conjugation
invariance of the distribution $\mu$.

\begin{proof}[Proof of Proposition \ref{conjugate1} (ii)]

Let $U,V\in\mathcal{U}$ be such that $U\cap V=\emptyset$ and $\mu(\Omega_{U,V})>0$. 

We first show that for $\mu$-a.e. $H$ in $\Omega_{U,V}$, the subgroup
$K$ defined in the statement is of finite index in $R_{G}(U)$. For
$\mu$-a.e. $H$ in $\Omega_{U,V}$, let $\Gamma<_{f.i.}W_{U}^{U\cdot\sigma^{-1}}$
be a subgroup provided by part (i) such that $\pi_{U}(\Gamma)\le\bar{H}_{U\to U}$.
Note that this property implies that the rigid stabilizer $R_{H}(U)$
is invariant under conjugation by $\Gamma$, that is $\Gamma\le N_{G}(R_{H}(U))$.
Indeed, for $\gamma\in\Gamma$, let $h\in H$ be an element with $h|_{U}=\gamma|_{U}$.
Since $\Gamma<W_{U}^{U\cdot\sigma^{-1}}$, we have that such $h$
leaves $U$ invariant, that is $h\in H_{U\to U}$. Let $g\in R_{H}(U)$.
Since $g$ acts trivially on $U^{c}$, we have $\gamma^{-1}g\gamma=h^{-1}gh$.
Since $h^{-1}gh\in R_{H}(U)$, it follows that $R_{H}(U)$ is invariant
under conjugation by $\Gamma$. Then $\Gamma\cap R_{G}(U)$ is contained
in the normalizer $N_{R_{G}(U)}(R_{H}(U))=N_{G}(R_{H}(U))\cap R_{G}(U)$.
Since $\Gamma\le_{f.i.}W_{U}^{U\cdot\sigma^{-1}}$, $R_{G}(U)\le W_{U}^{U\cdot\sigma^{-1}}$,
and $K$ contains $\Gamma\cap R_{G}(U)$, we conclude that $K$ is
of finite index in $R_{G}(U)$. 

Next we show the claim on finite conjugacy classes. Take an element
$\sigma\in\pi_{U^{c}}(G_{U\to U})$, as in the proof of Fact \ref{isomorphism},
let $B_{H_{U\to U}}(\sigma)=\{g\in\pi_{U}(G_{U\to U}):(g,\sigma)\in H_{U\to U}\}$.
Recall that $B_{H_{U\to U}}(\sigma)$ is either empty or a coset of
$R_{H}(U)$ in $\pi_{U}(H_{U\to U})$. Write
\[
A_{\sigma}=\left\{ H\in{\rm Sub}(G):\ \sigma\in\pi_{U^{c}}\left(H_{U\to U}\right)\right\} .
\]
Note that the set $A_{\sigma}$ is open in ${\rm Sub}(G)$ and invariant
under conjugation by the rigid stabilizer $R_{G}(U)$. 

Take an element $\sigma\in\pi_{U^{c}}(G_{U\to U})$ such that $\mu(A_{\sigma})>0$.
Define a map $\alpha_{\sigma}$ by 
\begin{align*}
\alpha_{\sigma}:A_{\sigma} & \to\{0,1\}^{\bar{G}_{U\to U}},\\
\alpha_{\sigma} & (H)=B_{H_{U\to U}}(\sigma).
\end{align*}
Similar as in Subsection \ref{sec:conditional}, define the probability
$\mu_{A_{\sigma}}:\mathcal{B}(A_{\sigma})\to[0,1]$ by setting $\mu_{A_{\sigma}}(C)=\mu(C)/\mu(A_{\sigma})$
for measurable set $C\subseteq A_{\sigma}$. Since $\mu$ is invariant
under conjugation by $G$ and $A_{\sigma}$ is set-wise invariant
under conjugation by $R_{G}(U)$, it follows that $\mu_{A_{\sigma}}$
is invariant under conjugation by $R_{G}(U)$.

Consider the map 
\begin{align*}
A_{\sigma} & \to\{0,1\}^{\bar{G}_{U\to U}}\times{\rm Sub}(R_{G}(U))\\
H & \mapsto\left(\alpha_{\sigma}(H),R_{H}(U)\right).
\end{align*}
The pair $\left(\alpha_{\sigma}(H),R_{H}(U)\right)$ takes value in
the subspace ${\rm D}_{U}$ of $\{0,1\}^{\bar{G}_{U\to U}}\times{\rm Sub}(R_{G}(U))$,
\[
{\rm D}_{U}=\left\{ \left(C,A\right)\in\{0,1\}^{\bar{G}_{U\to U}}\times{\rm Sub}(G_{U}):\ C=Ag|_{U},\ g\in G_{U\to U}\right\} .
\]
It is clear that ${\rm D}_{U}$ is a standard Borel space. Denote
by $\mathbb{P}_{\sigma,U}^{\mu}:{\rm Sub}(R_{G}(U))\times\mathcal{B}({\rm D}_{U})\to[0,1]$
the regular conditional distribution of $\left(\alpha_{\sigma}(H),R_{H}(U)\right)$
given the random variable $R_{H}(U)$, where the distribution of the
pair $\left(\alpha_{\sigma}(H),R_{H}(U)\right)$ is the pushforward
of $\mu_{A_{\sigma}}$ under the map $H\mapsto\left(\alpha_{\sigma}(H),R_{H}(U)\right)$. 

Denote by $N=N_{R_{G}(U)}(R_{H}(U))$ the normalizer of $R_{H}(U)$
in $R_{G}(U)$. Then for $\gamma\in N$, we have 
\begin{equation}
R_{\gamma^{-1}H\gamma}\left(U\right)=R_{H}(U).\label{eq:R-con}
\end{equation}
Since $N$ acts trivially on $U^{c}$, for an element $h$ of $H_{U\to U}$
recorded as $(f_{1},f_{2})$, where $f_{1}=\pi_{U}(h)$, $f_{2}=\pi_{U^{c}}(h)$,
we have that $\gamma^{-1}(f_{1},f_{2})\gamma=(\gamma^{-1}f_{1}\gamma,f_{2})$.
Therefore for $\gamma\in N$,
\begin{equation}
\alpha_{\sigma}(\gamma^{-1}H\gamma)=\gamma^{-1}\alpha_{\sigma}(H)\gamma.\label{eq:alpha-con}
\end{equation}
Similar to Lemma \ref{invariant}, we have the following invariance
property of $\mathbb{P}_{\sigma,U}^{\mu}$ under conjugation by $N$. 

\begin{lemma}\label{map-inv}

For $\mu$-a.e. $H\in A_{\sigma}$, the regular conditional distribution
$\mathbb{P}_{\sigma,U}^{\mu}\left(R_{H}(U),\cdot\right)$ of $\left(\alpha_{\sigma}(H),R_{H}(U)\right)$
given $R_{H}(U)$ is supported on the countable set $\{(R_{H}(U)g|_{U},R_{H}(U)):\ g\in G_{U\to U}\}$;
moreover, for any $\gamma\in N=N_{R_{G}(U)}(R_{H}(U))$ and $g\in G_{U\to U}$,
we have 

\[
\mathbb{P}_{\sigma,U}^{\mu}\left(R_{H}(U),\left\{ \left(R_{H}(U)g|_{U},R_{H}(U)\right)\right\} \right)=\mathbb{P}_{\sigma,U}^{\mu}\left(R_{H}(U),\left\{ \left(R_{H}(U)\left(\gamma^{-1}g\gamma\right)|_{U},R_{H}(U)\right)\right\} \right).
\]

\end{lemma}

\begin{proof}[Proof of Lemma \ref{map-inv}]

The first claim follows from property (i),(ii) of regular conditional
distribution. For any $\gamma\in N$ and bounded measurable function
$f$ on ${\rm D}_{U}$, we have 
\begin{align*}
 & \mathbb{E}_{\mu_{A_{\sigma}}}\left[f(\alpha_{\sigma}(H),R_{H}(U))\right]\\
= & \mathbb{E}_{\mu_{A_{\sigma}}}\left[f\left(\alpha_{\sigma}(\gamma^{-1}H\gamma),R_{\gamma H\gamma^{-1}}(U)\right)\right]\mbox{ (by invariance of \ensuremath{\mu_{A_{\sigma}}} under conjugation by }N)\\
= & \mathbb{E}_{\mu_{A_{\sigma}}}\left[f(\gamma^{-1}\alpha_{\sigma}(H)\gamma,R_{H}(U))\right],
\end{align*}
where in the last line we plugged in (\ref{eq:R-con}) and (\ref{eq:alpha-con}).
To rewrite into the form as stated, note that $\alpha_{\sigma}(H)$
is a coset of the form $\alpha_{\sigma}(H)=R_{H}(U)g|_{U}$, and $\gamma^{-1}\alpha_{\sigma}(H)\gamma=\gamma^{-1}R_{H}(U)\gamma(\gamma^{-1}g\gamma)|_{U}=R_{H}(U)(\gamma^{-1}g\gamma)|_{U}$.
The claim follows from uniqueness of regular conditional distribution. 

\end{proof}

With the lemma we return to the proof of the second claim of Proposition
\ref{conjugate1} part (ii). Given $\sigma\in\pi_{U^{c}}(G_{U\to U})$
with $\mu(A_{\sigma})>0$, by Lemma \ref{map-inv}, for $\mu$-a.e.
$H\in A_{\sigma}$, the conditional distribution $\mathbb{P}_{\sigma,U}^{\mu}\left(R_{H}(U),\cdot\right)$
is invariant under conjugation by $N=N_{R_{G}(U)}\left(R_{H}(U)\right)$.
Since the support of $\mathbb{P}_{\sigma,U}^{\mu}\left(R_{H}(U),\cdot\right)$
is countable, we have that the collection of cosets $\left\{ R_{H}(U)\left(\gamma^{-1}g\gamma\right)|_{U},\ \gamma\in N\right\} $
under conjugation by $N$ must be finite if $\mathbb{P}_{\sigma,U}^{\mu}\left(R_{H}(U),\left\{ \left(R_{H}(U)g|_{U},R_{H}(U)\right)\right\} \right)>0$.
In other words, write 
\begin{align*}
{\rm IC}\left(R_{H}(U)\right) & :=\left\{ \left(R_{H}(U)g,R_{H}(U)\right):\ g\in G_{U\to U},\ \left|{\rm Cl}_{N}\left(R_{H}(U)g\right)\right|=\infty\right\} ,\\
\mbox{where } & {\rm Cl}_{N}\left(R_{H}(U)g\right)=\left\{ R_{H}(U)\gamma^{-1}g\gamma:\ \gamma\in N\right\} .
\end{align*}
Then the invariance property in Lemma \ref{map-inv} implies that
$\mathbb{P}_{\sigma,U}^{\mu}\left(R_{H}(U),{\rm IC}\left(R_{H}(U)\right)\right)=0$
for $\mu_{A_{\sigma}}$-a.e. $H$. 

Let 
\[
E_{\sigma}=\{H\in A_{\sigma}:\mbox{the }N_{R_{G}(U)}\left(R_{H}(U)\right)\mbox{-conjugacy class of }\alpha_{\sigma}(H)\mbox{ is infinite}\}.
\]
That is, $E_{\sigma}$ is the event that in the quotient group $\bar{H}_{U\to U}/R_{H}(U)$,
the element $\alpha_{\sigma}(H)$ has infinite $N$-conjugacy class
$\left\{ \gamma^{-1}\alpha_{\sigma}(H)\gamma,\ \gamma\in N\right\} $.
The reasoning above implies that $E_{\sigma}$ is a $\mu$-null set.
Indeed, for $\sigma$ such that $\mu(A_{\sigma})>0$, by property
(iii) of regular conditional probability, we have
\[
\mu_{A_{\sigma}}(E_{\sigma})=\mathbb{E}_{\mu_{A_{\sigma}}}\left[\mathbb{P}_{\sigma,U}^{\mu}\left(R_{H}(U),{\rm IC}\left(R_{H}(U)\right)\right)\right]=0.
\]
For $\sigma$ such that $\mu(A_{\sigma})=0$, we have $\mu(E_{\sigma})=\mu(A_{\sigma})=0$. 

Thus if $H\in{\rm Sub}(G)-\cup_{\sigma\in\pi_{U^{c}}(G_{U\to U})}E_{\sigma}$,
we have that for any $\sigma\in\pi_{U^{c}}(G_{U\to U})$, either $\sigma\notin\pi_{U^{c}}(H_{U\to U})$
or the $N$-conjugacy orbit $\left\{ \gamma^{-1}\alpha_{\sigma}(H)\gamma\right\} _{\gamma\in N}$
is finite, where $N=N_{R_{G}(U)}\left(R_{H}(U)\right)$. Recall that
by Fact \ref{isomorphism}, there is an isomorphism $\phi:\pi_{U^{c}}(H_{U\to U})/H_{U\to U}\cap R_{G}(U^{c})\to\pi_{U}(H_{U\to U})/R_{H}(U)$
where $\phi(\bar{\sigma})=\alpha_{\sigma}(H)$ for $\sigma\in\pi_{U^{c}}(H_{U\to U})$.
In particular, since $\phi$ is onto, as sets we have
\[
\cup_{\sigma\in\pi_{U^{c}}(G_{U\to U})}\alpha_{\sigma}(H)=\bar{H}_{U\to U}/R_{H}(U).
\]
It follows that for $H\in{\rm Sub}(G)-\cup_{\sigma\in\pi_{U^{c}}(G_{U\to U})}E_{\sigma}$,
in the quotient group $\bar{H}_{U\to U}/R_{H}(U)$, every element
has finite $N$-conjugacy orbit. Recall that by definition of $K$
we have that $K=N\cap\bar{H}_{U\to U}$ and $R_{H}(U)=H\cap K$. Since
the $N$-conjugacy orbits in $\bar{H}_{U\to U}/R_{H}(U)$ are finite,
it follows that the quotient group $K/H\cap K$ is FC. We have proved
that the statement of part (ii) holds for any $H$ not in the $\mu$-null
set $\cup_{\sigma\in\pi_{U^{c}}(G_{U\to U})}E_{\sigma}$.

\end{proof}

\subsection{Containment of derived subgroups of rigid stabilizers}

Proposition \ref{conjugate1} allows us to draw conclusions on subgroups
that $H$ contains, provided some additional information on the rigid
stabilizers $R_{G}(U)$. 

We first deduce from Proposition \ref{conjugate1} the following immediate
corollary when every finite index subgroup of $G_{U}$ contains an
infinite simple group. This statement will be used in applications
to classify IRS of certain topological full groups. 

\begin{corollary}\label{simple}

Suppose $G\curvearrowright X$ as in Assumption \ref{standing}. Suppose
an open set $U\subset X$ has the property that there is an infinite
simple group $A_{U}$ such $A_{U}\le R_{G}(U)$ and every finite index
subgroup of $R_{G}(U)$ contains $A_{U}$. Then $\mu$-a.e. $H$ has
the following property: if $H$ contains an element $g$ such that
$U\cap U\cdot g=\emptyset$, then $H$ contains $A_{U}$ as a subgroup.

\end{corollary}

\begin{proof}

By part (ii) of Proposition \ref{conjugate1}, for $\mu$-a.e. $H$
in $\Omega_{g}=\{H:H\ni g\}$ where $g$ is an element such that $U\cap U\cdot g=\emptyset,$
there is a finite index subgroup $K\le_{f.i.}R_{G}(U)$ such that
$K/K\cap H$ is an FC group. By the assumption on $G_{U}$ we have
that $A_{U}\le K$. Then $A_{U}/A_{U}\cap H$ is isomorphic to $A_{U}(K\cap H)/K\cap H$,
in particular it is an FC group. Since $A_{U}$ is assumed to be an
infinite simple group, we conclude that $A_{U}\cap H=A_{U}$, in other
words, $A_{U}$ is contained in $H$. 

\end{proof}

More generally we have the following lemma, which is a variation of
the standard double commutator lemma mentioned in the Introduction.

\begin{lemma}\label{doublecom}

Let $G\curvearrowright X$ be as in Assumption \ref{standing}. Let
$U\in\mathcal{U}$, $K\le_{f.i.}R_{G}(U)$ and $N\vartriangleleft K$
be such that the quotient $K/N$ is an FC group. Then if $x\in U$
is a point such that the orbit $x\cdot R_{G}(V)$ is infinite for
any $V\in\mathcal{U}$ with $x\in V\subseteq U$, then there exists
a neighborhood $W$ of $x$ such that $\left[R_{G}(W),R_{G}(W)\right]\le N$. 

\end{lemma}

\begin{proof}

We will use a few times the double commutator lemma (\cite[Lemma 4.1]{Nek-finitepresent}).
We recall the argument here for the convenience of the reader: suppose
$\Gamma\curvearrowright X$ and let $N$ be a non-trivial normal subgroup
of $\Gamma$. If there is an element $\gamma\in N$ and a set $U$
such that $U\cap U\cdot\gamma=\emptyset$, then $[R_{\Gamma}(U),R_{\Gamma}(U)]\le N$.
Indeed, given such $\gamma,U$, take any elements $\alpha,\beta\in\Gamma_{U}$,
then $\left[\left[\gamma,\alpha\right],\beta\right]=[\alpha,\beta]$
and since $N$ is normal, we have $\left[\left[\gamma,\alpha\right],\beta\right]\in N$.
It follows that $[R_{\Gamma}(U),R_{\Gamma}(U)]\le N$. 

Replacing $K$ by a its normal core in $R_{G}(U)$ if necessary, we
may assume $K$ is a normal subgroup of finite index in $R_{G}(U)$.
We first verify that $K/N$ being an FC-group implies that $x$ is
not a fixed point of $N$. Suppose on the contrary that $x$ is fixed
by $N$. Since the orbit $x\cdot R_{G}(U)$ is assumed to be infinite,
$x$ cannot be fixed by $K.$ Choose a $\gamma\in K$ such that $x\cdot\gamma\neq x$.
Let $V$ be a neighborhood of $x$ such that $V\cap V\cdot\gamma=\emptyset$.
Then for any $g\in R_{G}(V)$, the conjugation $g\gamma g^{-1}$ coincides
with $g\gamma$ on $V$. In particular the set $\{x\cdot g\gamma g^{-1}:\ g\in R_{K}(V)\}$
has the same cardinality as $\{x\cdot g:\ g\in R_{K}(V)\}$, where
the latter is infinite by assumption. It follows that in $K/N$ the
element $N\gamma$ has infinite conjugacy class: \textcolor{black}{since
$x$ is fixed by $N$, $g_{1}\gamma g_{1}^{-1}$ and $g_{2}\gamma g_{2}^{-1}$
are in different right cosets of $N$ if $x\cdot g_{1}\neq x\cdot g_{2}$.
Thus it contradicts with $K/N$ being FC. }

Since $x$ is not fixed by $N$, by the double commutator lemma for
normal subgroups, we conclude that there exists a neighborhood $W$
of $x$ such that $N\ge[R_{K}(W),R_{K}(W)]$. Finally, note that since
$K$ is normal in $R_{G}(U)$, we have that $[R_{K}(W),R_{K}(W)]$
is normal in $R_{G}(W)$. Since $R_{G}(W)/[R_{K}(W),R_{K}(W)]$ is
virtually abelian, thus FC, the argument in the paragraph above implies
that there exists a neighborhood $W'$ of $x$ such that $[R_{G}(W'),R_{G}(W')]\le[R_{K}(W),R_{K}(W)]\le N$.

\end{proof}

We are now ready to prove Theorem \ref{inclu-rigid}. 

\begin{proof}[Proof of Theorem \ref{inclu-rigid}]

Let $\Omega'\subseteq{\rm Sub}(G)$ be a subset of full $\mu$-measure
such that for any $H\in\Omega'$, the two statements of Proposition
\ref{conjugate1} are satisfied. Let $H\in\Omega'$ and $x\in X$
be a point such that $x$ is not a fixed point of $H$. Find a neighborhood
$x\in U\in\mathcal{U}$ and $h\in H$ such that $U\cap U\cdot h=\emptyset$.
Then as in (ii) of Proposition \ref{conjugate1}, we have that $K=N_{R_{G}(U)}(R_{H}(U))$
is of finite index in $R_{G}(U)$ and $K/K\cap H$ is FC. 

(i) If there exists an open set $U$ such that $R_{G}(U)=\{id\}$,
then the statement is trivially true. We may assume that for any $U\in\mathcal{U}$,
$R_{G}(U)$ is non-trivial. Then in this case the rigid stabilizer
$R_{G}(U)$ is ICC for any $U$. To see this, we follow the argument
in \cite[Theorem 9.17]{Grigorchuk2011} which shows that a weakly
branch group is ICC. Take $\gamma\in R_{G}(U)$, $\gamma\neq id$,
and choose an open subset $V\subseteq U$ such that $V\cap V\cdot\gamma=\emptyset$.
Next choose an infinite sequence of mutually disjoint open subsets
$V_{1},V_{2},\ldots$ of $V$ (this is possible because under our
assumptions $X$ is Hausdorff and has no isolated points). Take $g_{i}\in R_{G}(V_{i})\setminus\{id\}$
for each $i$ and consider the collection of conjugates $\left\{ g_{i}\gamma g_{i}^{-1}\right\} _{i\in\mathbb{N}}$.
The restriction of $g_{i}\gamma g_{i}^{-1}\gamma^{-1}$ to $V$ acts
as $g_{i}$ in $V_{i}$ and as $id$ on $V\setminus V_{i}$. Thus
the set of elements $\left\{ g_{i}\gamma g_{i}^{-1}\right\} _{i\in\mathbb{N}}$
are all distinct. We conclude that $R_{G}(U)$ is ICC. 

Since $K$ is a finite index subgroup of $R_{G}(U)$, we have that
it is ICC as well, thus $K\cap H$ must be a non-trivial normal subgroup
of $K$. Thus the double commutator lemma implies that there exists
a non-empty open subset $V\subseteq U$ such that $\left[R_{K}(V),R_{K}(V)\right]\le K\cap H$.
Since $K$ is of finite index in $R_{G}(U)$, we have $\left[R_{K}(V),R_{K}(V)\right]$
contains a non-trivial normal subgroup of $R_{G}(V)$. Therefore by
the double commutator lemma again, there exists a non-empty open set
$W\subseteq V$ such that 
\[
[R_{G}(W),R_{G}(W)]\le\left[R_{K}(V),R_{K}(V)\right]\le H\cap K\le H.
\]

(ii) Under the additional assumption that $R_{G}(V)$ has no fixed
point in $V$ for any $V\in\mathcal{U}$, we have that the orbit $x\cdot R_{G}(V)$
is infinite for any $V\in\mathcal{U}$ and $x\in V$. Indeed, first
this assumption implies that $X$ has no isolated point. For any $x\in V$,
first find $g_{1}\in R_{G}(V)$ such that $x\neq x\cdot g_{1}$. Since
$X$ is Hausdorff, we can find a neighborhood $V_{1}$ of $x$ such
that $V_{1}\cap V_{1}\cdot g=\emptyset$ . Continue this with choosing
a $g_{2}\in R_{G}(V_{1})$ such that $x\neq x\cdot g_{2}$ and $V_{2}$
a neighborhood of $x$ such that $V_{2}\cap V_{2}\cdot g_{2}=\emptyset$.
In this way we find a sequence of group elements $g_{n}\in R_{G}(V)$
such that $x\cdot g_{n}$ are all distinct. 

Thus we can apply Lemma \ref{doublecom} to $N=K\cap H$, which implies
that these exists an open cover $V_{i}$ of $U$ such that $[R_{G}(V_{i}),R_{G}(V_{i})]\le H$
for all $V_{i}$. Take $W$ to be a $V_{i}$ such that $x\in V_{i}$,
then we obtain the statement. 

\end{proof} 

\begin{remark}\label{nofixed}

The condition that $R_{G}(U)$ has no fixed point in $U$ for any
$U\in\mathcal{U}$ is verified in many situations. Note that if $x\in U$
is a fixed point of $R_{G}(U)$, then for any $g\in R_{G}(U)$, $\gamma\in N_{G}(R_{G}(U))$,
we have $x\cdot\gamma g=x\cdot\gamma$. That is, the orbit of $x$
under $N_{G}(R_{G}(U))$ is contained in the fixed point set of $R_{G}(U)$.
Therefore if the orbit $x\cdot G_{U\to U}$ is dense in $U$ and $R_{G}(U)\neq\{id\}$,
then $x$ cannot be a fixed point of $R_{G}(U)$. 

\end{remark}

\section{Preliminaries on topological full groups\label{sec:Preliminaries}}

This section serves as preparation for the next section where we derive
consequences of Theorem \ref{inclu-rigid} for topological full groups.
Most of the materials are drawn from \cite{Nek-simple}. The reader
may also consult the survey \cite{MatuiSurvey} and references therein
to have a more complete view of recent development in the study of
topological full groups. 

\subsection{Definitions \label{subsec:full-def}}

A groupoid consists of a unit space $\mathcal{\mathcal{G}}^{(0)}$,
a set of morphisms $\mathcal{G}$ and maps $\mathsf{s,r}:\mathcal{G}\to\mathcal{G}^{(0)}$
called the \emph{source} and \emph{range} that specify the initial
(source) and terminal (range) of a morphism. Multiplication in $\mathcal{G}$
is partially defined: product $\gamma\delta$ of $\gamma,\delta\in\mathcal{G}$
is defined if and only if $\mathsf{r}(\gamma)=\mathsf{s}(\delta)$.
A topological groupoid is a groupoid equipped with a topology on it
such that the operations of multiplication and taking inverse are
continuous.

An important class of examples of groupoid is the \emph{transformation
groupoid of an action} $G\curvearrowright X$:

\begin{example}\label{germs}

Let $G$ be a countable group acting by homeomorphisms on a topological
space $X$. Define an equivalence relation on $G\times X$ by $(g_{1},x_{1})\sim(g_{2},x_{2})$
if $x_{1}=x_{2}$ and there exists a neighborhood $U$ of $x_{1}$
such that $g_{1}|_{U}=g_{2}|_{U}$. On the quotient space $G\times X/\sim$
(the space of germs), multiplication of equivalence classes $(g_{1},x_{1})$,
$(g_{2},x_{2})$ is defined if and only if $x_{2}=x_{1}\cdot g_{1}$
and the rule of multiplication is $(g_{1},x_{1})(g_{2},x_{2})=(g_{1}g_{2},x_{1})$.
The groupoid $\mathcal{G}=G\times X/\sim$ is called \emph{the groupoid
of germs of the action} $G\curvearrowright X$, it is also called
the\emph{ transformation groupoid }of the action. For an open set
$U\subseteq X$, the set of germs $\{(g,y),y\in U\}$ is an open set
of $\mathcal{G}.$ The collection of all such open sets forms a basis
of topology of $\mathcal{G}$.

\end{example}

More generally one can consider groupoid of germs (also called effective
\'etale groupoids), for which we now recall the definition. 

\begin{definition}

Let $\mathcal{G}$ be a topological groupoid. 
\begin{itemize}
\item A \emph{$\mathcal{G}$-bisection} is a compact open subset $F\subseteq\mathcal{G}$
such that $\mathsf{s}:F\to\mathsf{s}(F)$ and $\mathsf{r}:F\to\mathsf{r}(F)$
are homeomorphisms. 
\item The groupoid $\mathcal{G}$ is said to be \emph{\'etale} if it has
a basis of topology consisting of $\mathcal{G}$-bisections.
\item An \'etale groupoid $\mathcal{G}$ is called a \emph{groupoid of
germs }(or \emph{effective}) if for any non-unit $\gamma\in\mathcal{G}$
and any bisection $F$ that contains $\gamma$, there exists $\delta\in F$
such that $\mathsf{s}(\delta)\neq\mathsf{r}(\delta)$. 
\end{itemize}
\end{definition}

A bisection $F$ defines a homeomorphism from $\mathsf{s}(F)$ to
$\mathsf{r}(F)$, denote by $\tau_{F}$ this homeomorphism. Multiplication
of bisections corresponds to composition of the associated homeomorphisms.
For an \'etale groupoid $\mathcal{G}$, its topological full group
is defined as:

\begin{definition}

Let $\mathcal{G}$ be an \'etale groupoid. Its \emph{topological
full group} $\mathsf{F}(\mathcal{G})$ is the set of bisections $F\subset\mathcal{G}$
such that $\mathsf{s}(F)=\mathsf{r}(F)=\mathcal{G}^{(0)}$ with respect
to multiplication of bisections. 

\end{definition}

\begin{example}\label{transform}

Let $\mathcal{G}$ be the transformation groupoid of the action $G\curvearrowright X$,
where $G$ is countable and $X$ is homeomorphic to the Cantor set.
Then the topological full group $\mathsf{F}(\mathcal{G})$ can be
described as consisting of all homeomorphisms $\gamma:X\to X$ such
that for every point $x\in X$, there is a neighborhood $U$ and $g\in G$
such that $\gamma|_{U}=g|_{U}$. In other words, $\mathsf{F}(\mathcal{G})$
consists of homeomorphisms $\gamma:X\to X$ such that there exists
a continuous map $c:X\to G$ such that $x\cdot\gamma=x\cdot c(\gamma)$
for all $x$.

More generally, when $\mathcal{G}$ is a groupoid of germs, $\mathsf{F}(\mathcal{G})$
agrees with the definition of topological full group in \cite{Matui2012,Matui2015}:
it consists of all homeomorphisms $\gamma:\mathcal{G}^{(0)}\to\mathcal{G}^{(0)}$
such that there exists a bisection $F\subset\mathcal{G}$ such that
$\gamma=\tau_{F}$. 

\end{example}

The notion of multisections is introduced in \cite{Nek-simple}. A
\emph{multisection of degree} $d$ is a collection of $d^{2}$ bisections
$F=\left\{ F_{i,j}\right\} _{i,j=1}^{d}$ such that 
\begin{description}
\item [{(i)}] $F_{i_{1},i_{2}}F_{i_{2},i_{3}}=F_{i_{1},i_{3}}$ for all
$1\le i_{1},i_{2},i_{3}\le d$,
\item [{(ii)}] the bisections $F_{i.i}$ are disjoint subsets of the unit
space $X=\mathcal{G}^{(0)}$.
\end{description}
The union $\cup_{i=1}^{d}F_{i,i}$ is called the \emph{domain} of
the multisection and the sets $F_{i,i}$ the \emph{components} of
the domain. It follows from properties (i), (ii) that $\mathsf{s}(F_{i,j})=F_{i,i}$
and $\mathsf{r}(F_{i,j})=F_{j,j}$ for $1\le i,j\le d$. 

Given a multisection $F=\left\{ F_{i,j}\right\} _{i,j=1}^{d}$ and
a permutation $\pi\in{\rm Sym}(d)$, let 
\begin{equation}
F_{\pi}:=\cup_{i=1}^{d}F_{i,\pi(i)}\cup(X\setminus U),\label{eq:defF}
\end{equation}
where $U$ is the domain of $F$. Then $\pi\mapsto F_{\pi}$ gives
an embedding of ${\rm Sym}(d)$ into $\mathsf{F}(\mathcal{G})$. Denote
by $\mathsf{S}(F)$ the image of ${\rm Sym}(d)$ under this embedding,
and $\mathsf{A}(F)$ the image of the subgroup ${\rm Alt}(d)$. 

Now we have introduced enough notations to state the definitions of
the groups $\mathsf{S}(\mathcal{G})$ and $\mathsf{A}(\mathcal{G})$. 

\begin{definition}

Let $\mathcal{G}$ be an \'etale groupoid. The group $\mathsf{S}_{d}(\mathcal{G})$
($\mathsf{A}_{d}(\mathcal{G})$ resp.) is defined as the subgroup
of $\mathsf{F}(\mathcal{G})$ which is generated by the union of all
subgroups $\mathsf{S}(F)$, all multisections of degree $d$. Denote
by $\mathsf{S}(\mathcal{G})=\mathsf{S}_{2}(\mathcal{G})$ and $\mathsf{A}(\mathcal{G})=\mathsf{A}_{3}(\mathcal{G})$.

\end{definition}

From the definitions it is clear that if $\mathcal{H}$ is an open
sub-groupoid of an \'etale groupoid $\mathcal{G}$ then $\mathsf{S}(\mathcal{H})\le\mathsf{S}(\mathcal{G})$
and $\mathsf{A}(\mathcal{H})\le\mathsf{A}(\mathcal{G})$. 

A Borel probability measure $\mu$ on the unit space $\mathcal{G}^{(0)}$
is said to be $\mathcal{G}$-invariant if $\mu(\mathsf{r}(F))=\mu(\mathsf{s}(F))$
for any bisection $F$. 

We will only consider the special case where $\mathcal{G}$ is a minimal
groupoid of germs. Recall that a groupoid $\mathcal{G}$ is \emph{minimal}
if all orbits are dense in $\mathcal{G}^{(0)}$. By \cite[Theorem 1.1]{Nek-simple},
if $\mathcal{G}$ is a minimal groupoid of germs with $\mathcal{G}^{(0)}$
homeomorphic to the Canter set, then $\mathsf{A}(\mathcal{G})$ is
simple and contained in every non-trivial normal subgroup of $\mathsf{F}(\mathcal{G})$. 

When $\mathcal{G}$ is an AF-groupoid (AF is abbreviation for approximately
finite), the topological full group $\mathsf{F}(\mathcal{G})$ and
the alternating full group $\mathsf{A}(\mathcal{G})$ are well studied,
see \cite{Matui2012}. AF-groupoids are associated with Bratteli diagrams.
In the following example we review necessary terminologies and introduce
notations for later use. 

\begin{example}\label{AF}

A \emph{Bratteli diagram }is an infinite graph $B=(V,E)$ such that
the vertex set $V=\cup_{i=0}^{\infty}V_{i}$ and the edge set $E=\cup_{i=1}^{\infty}E_{i}$
are partitions such that $V_{0}=\{v_{0}\}$ is a single point, $V_{i}$
and $E_{i}$ are finite sets, and moreover there are range map $r$
and source map $s$ from $E$ to $V$ such that $r(E_{i})=V_{i}$
and $s(E_{i})=V_{i-1}$. That is an edge $e\in E$ connects a vertex
in $V_{i-1}$ to a vertex in $V_{i}$ for some $i$. The set $V_{i}$
is referred to as the $i$-th level vertices of the diagram $B$. 

The space of infinite paths of $B$, denoted by $X_{B}$, consists
of sequences $(e_{1},e_{2},\ldots)$ of edges such that $r(e_{i})=s(e_{i+1})$
for all $i$. A finite path is a finite sequence $(e_{1},\ldots,e_{n})$,
such that $r(e_{i})=s(e_{i+1})$ for all $1\le i\le n$. The vertex
$s(e_{1})$ is called the beginning of the path and $r(e_{n})$ the
end of the path. Endow $X_{B}$ with the product topology generated
by cylinder sets $U(e_{1},\ldots,e_{n})=\{x\in X_{B}:x_{i}=e_{i},\ 1\le i\le n\}$. 

If two finite paths $(e_{1},\ldots,e_{n})$ and $(f_{1},\ldots,f_{n})$
have the same end, $r(e_{n})=r(f_{n})$, then this pair defines a
homeomorphism between clopen subsets of $X_{B}$ mapping a path of
the form $\left(e_{1},\ldots,e_{n},x_{n+1},\ldots\right)$ to $(f_{1},\ldots,f_{n},x_{n+1},\ldots)$.
Let $\mathcal{G}$ be the groupoid of germs of the semigroup generated
by such local homeomorphisms. Such a groupoid associated with a Bratteli
diagram is called an \emph{AF groupoid}. By slight abuse of notation,
we write $\mathsf{F}(B)$ for the topological full group of the groupoid
associated with the Bratteli diagram $B$, similarly for $\mathsf{S}(B)$
and $\mathsf{A}(B)$. 

For a vertex $v\in V$, let $E(v_{0},v)$ be the space of finite paths
that start at $v_{0}$ and end at $v$. Let $d(v_{0},v)$ be the cardinality
of $E(v_{0},v)$. From definitions, the topological full group of
the Bratteli diagram $B$ is the direct limit of $G_{n}=\prod_{v\in V_{n}}{\rm Sym}(E(v_{0},v)))\simeq\prod_{v\in V_{n}}{\rm Sym}_{d(v_{0},v)}$.
The alternating full group $\mathsf{A}(B)$ is the direct limit of
$\prod_{v\in V_{n}}{\rm Alt}(E(v_{0},v)))\simeq\prod_{v\in V_{n}}{\rm Alt}_{d(v_{0},v)}$. 

For an infinite path $x\in X_{B}$, the $\mathcal{G}$-orbit of $x$
consists of all infinite paths cofinal to $x$. The groupoid $\mathcal{G}$
associated with the Bratteli diagram $B$ is minimal on $X_{B}$ if
and only if the diagram $B$ is \emph{simple}, which means that for
any level $V_{n}$, there exists a level $m>n$ such that every pair
of vertices $(v,w)\in V_{n}\times V_{m}$ is connected by a finite
path, see for example \cite[Theorem 3.11]{Putnam}.

\end{example}

\subsection{Rigid stabilizers }

For an open set $U\subseteq\mathcal{G}^{(0)}$, denote by $\mathcal{G}|_{U}$
the \emph{restriction }of $\mathcal{G}$ to $U$, which is the sub-groupoid
\[
\mathcal{G}|_{U}:=\{\gamma\in\mathcal{G}:\ \mathsf{s}(\gamma),\mathsf{r}(\gamma)\in U\}.
\]
Similarly one can consider the alternating full group $\mathsf{A}(\mathcal{G}|_{U})$,
which by definition is generated by the union of $\mathsf{A}(F)$,
where $F$ ranges over multisections of degree $d\ge3$ in the sub-groupoid
$\mathcal{G}|_{U}$. The group $\mathsf{A}(\mathcal{G}|_{U})$ is
naturally identified as a subgroup of $\mathsf{A}(\mathcal{G})$ by
extending elements of $\mathsf{A}(\mathcal{G}|_{U})$ identically
on the complement of $U$. We have the following identification:

\begin{fact}\label{rigidstab}

Let $\mathcal{G}$ be a groupoid of germs and $U$ be an open subset
of $X=\mathcal{G}^{(0)}$, then 
\[
\mathsf{A}(\mathcal{G}|_{U})=R_{\mathsf{A}(\mathcal{G})}(U).
\]

\end{fact}

\begin{proof}

The inclusion $\mathsf{A}(\mathcal{G}|_{U})\le R_{\mathsf{A}(\mathcal{G})}(U)$
is clear by definitions. We need to verify the other direction. Let
$g=F_{\pi}$ be an element in $R_{\mathsf{A}(\mathcal{G})}(U)$, where
$F=\{F_{i,j}\}_{i,j=1}^{d}$ is a multisection and $\pi\in{\rm Alt}_{d}$.

Given the multisection $F=\left\{ F_{i,j}\right\} _{i,j=1}^{d}$,
one can define a sub-multisection $\{F_{k_{i},k_{j}}\}_{i,j=1}^{\ell}$
by choosing a subset $1\le k_{1}<\ldots<k_{\ell}\le d$ and restrict
to these indices. It is clear by definition that $\{F_{k_{i},k_{j}}\}_{i,j=1}^{\ell}$
is a multisection of degree $\ell$. Let $\pi$ be a permutation of
$\{1,\ldots,d\}$ and denote by ${\rm Fix}(\pi)=\{i:1\le i\le d,\ i=\pi(i)\}$
the set of fixed points of $\pi$. List elements of $Q_{\pi}=\{1,\ldots,d\}\setminus{\rm Fix}(\pi)$
in increasing order as $k_{1},\ldots,k_{\ell}$, $\ell=d-|{\rm Fix}(\pi)|$.
Denote by $F'$ the sub-multisection $\{F_{k_{i},k_{j}}\}_{i,j=1}^{\ell}$.
Let $\pi'$ be the restriction of $\pi$ on $Q_{\pi}$. Then by definition
in (\ref{eq:defF}) we have 
\[
F_{\pi}=F'_{\pi'}.
\]
That is, for the group element $F_{\pi}\in\mathsf{F}(\mathcal{G})$,
one can remove the components $F_{i,i}$ in $F$ with $\pi(i)=\pi$
and consider it as $F'_{\pi'}$ where every component in the domain
is mapped to some other component. 

Restrict to the sub-multisection $F'=\{F_{k_{i},k_{j}}\}_{i,j=1}^{\ell}$,
where $\ell=d-|{\rm Fix}(\pi)|$, then $g=F'_{\pi'}$, $\pi'(k_{i})=\pi(k_{i})\neq k_{i}$.
Since the components of the domain are by definition disjoint, $\pi'(k_{i})\neq k_{i}$
implies that for any point $x\in F_{k_{i},k_{i}}$, $x\cdot F'_{\pi'}\neq x$.
Since $g\in R_{\mathsf{A}(\mathcal{G})}(U)$ fixes $U^{c}$ pointwise,
it follows that $\left(\cup_{i=1}^{\ell}F_{k_{i},k_{i}}\right)\cap U^{c}=\emptyset$.
In other words, the domain of the multisection $F'$ is contained
in $U$. By definition of the restriction $\mathcal{G}|_{U}$ we have
that $\mathsf{A}(F')\le\mathsf{A}(\mathcal{G}|_{U})$, in particular,
$g\in\mathsf{A}(\mathcal{G}|_{U})$. 

\end{proof}

Similar to the alternating group on a finite set, we have the following
property. When $\mathcal{G}$ is a minimal AF-groupoid, this is shown
in \cite[Lemma 3.4]{Dudko-Medynets3}. 

\begin{lemma}\label{condition1}

Suppose $\mathcal{G}$ is a minimal groupoid of germs with unit space
$X$ homeomorphic to the Cantor set. Then if $U$ and $V$ are clopen
sets such that $U\cap V\neq\emptyset$, we have
\[
R_{\mathsf{A}(\mathcal{G})}(U\cup V)=\left\langle R_{\mathsf{A}(\mathcal{G})}(U),R_{\mathsf{A}(\mathcal{G})}(V)\right\rangle .
\]

\end{lemma}

The proof of Lemma \ref{condition1} relies on the following fact,
which is a direct consequence of \cite[Lemma 3.2, Proposition 3.3]{Nek-simple}.
Let $F=\left\{ F_{i,j}\right\} _{i,j=1}^{d}$ be a multisection and
$U$ be a clopen subset of $F_{k,k}$ for some $k\in\{1,\ldots,d\}$.
Write $U_{i}=\mathsf{r}(UF_{k,i})$ and $F_{i,j}'=U_{i}F_{i,j}$.
Then $F'=\{F_{i,j}'\}_{i,j=1}^{d}$ is a multisection, which is referred
to as the \emph{restriction} of $F$ to $U$. 

\begin{fact}\label{multi-restrict}

Let $\mathcal{G}$ be a minimal groupoid of germs with unit space
$X=\mathcal{G}^{(0)}$ homeomorphic to the Cantor set. Let $F=\left\{ F_{i,j}\right\} _{i,j=1}^{d}$
be a multisection and $\mathcal{P}=\{U_{j}\}_{j=1}^{q}$ an open cover
of the domain of $F$. Then there is a collection of multisections
$F^{(k)}=\left\{ F_{i,j}^{(k)}\right\} _{i,j=1}^{d_{k}}$, $1\le k\le\ell$
such that:
\begin{description}
\item [{(i)}] each component $F_{i,i}^{(k)}$ is contained in $U_{j}$,
for some $j\in\{1,\ldots,q\}$;
\item [{(ii)}] for each $k$ and each set $U_{j}$, $\left|\{i:\ F_{i,i}^{(k)}\subseteq U_{j}\}\right|\ge3$;
\item [{(iii)}] $\mathsf{A}(F)$ is contained in the group generated by
$\cup_{k=1}^{\ell}\mathsf{A}(F^{(k)})$. 
\end{description}
\end{fact}

\begin{proof}[Proof of Lemma \ref{condition1}]

Let $F=\{F_{i,j}\}_{i,j=1}^{d}$ be a multisection of degree $d$
with domain contained $U\cup V$. If the domain of $F$ is contained
in $U$ (resp. $V$), then by definitions we have $\mathsf{A}(F)\le R_{\mathsf{A}(\mathcal{G})}(U)$
(resp. $R_{\mathsf{A}(\mathcal{G})}(U)$). We need to consider the
case where domain of $F$ intersects both $U,V$. Partition $U\cup V$
into $U\cap V^{c}$, $U\cap V$ and $V\cap U^{c}$. Let $F^{(k)}$,
$1\le k\le\ell$, be a collection of multisections satisfying the
statement of Fact \ref{multi-restrict}. We show that $\mathsf{A}\left(F^{(k)}\right)$
is contained in $\left\langle R_{\mathsf{A}(\mathcal{G})}(U),R_{\mathsf{A}(\mathcal{G})}(V)\right\rangle $.
To see this, let $I_{1}=\{i:\ F_{i,i}^{(k)}\subseteq U\}$ and let
$F_{U}^{(k)}$ be the sub-multisection $\left\{ F_{i,j}^{(k)}\right\} _{i,j\in I_{1}}$.
Since by definition the domain of $F_{U}^{(k)}$ is contained in $U$,
we have that $\mathsf{A}\left(F_{U}^{(k)}\right)\le R_{\mathsf{A}(\mathcal{G})}(U)$.
Similarly we take the sub-multisection $F_{V}^{(k)}$ whose domain
is contained in $V$. By property (ii) in Lemma \ref{multi-restrict},
there are at least $3$ components of $F^{(k)}$ contained in the
intersection $U\cap V$. Recall the elementary fact that if $X$ and
$Y$ are two finite sets such that $|X|,|Y|\ge3$ and $X\cap Y\neq\emptyset$,
then the alternating group ${\rm Alt}(X\cup Y)$ is generated by ${\rm Alt}(X)$
and ${\rm Alt}(Y)$. It follows that 
\[
\mathsf{A}\left(F^{(k)}\right)=\left\langle \mathsf{A}\left(F_{U}^{(k)}\right),\mathsf{A}\left(F_{V}^{(k)}\right)\right\rangle \le\left\langle R_{\mathsf{A}(\mathcal{G})}(U),R_{\mathsf{A}(\mathcal{G})}(V)\right\rangle .
\]
We conclude that $\mathsf{A}\left(F\right)\le\left\langle R_{\mathsf{A}(\mathcal{G})}(U),R_{\mathsf{A}(\mathcal{G})}(V)\right\rangle $. 

\end{proof}

\section{Applications to classification of IRSs of topological full groups\label{sec:full groups}}

In this section we consider invariant random subgroups of topological
and alternating full groups of a minimal groupoid of germs $\mathcal{G}$. 

\subsection{Consequences of Corollary \ref{simple} \label{subsec:full}}

With the alternating full group $A(\mathcal{G})$ in mind, we derive
the following statement from Corollary \ref{simple}. Recall that
$F(X)$ denotes the space of closed subsets of $X$, equipped with
the Vietoris topology. 

\begin{proposition}\label{corfull}

Let $G$ be a countable group acting faithfully on the Cantor set
$X$ by homeomorphisms. Let $\mathcal{U}$ be the collection of non-empty
clopen subsets of $X$. Suppose there is a collection of infinite
simple groups $\left\{ A_{U}\right\} _{U\in\mathcal{U}}$ such that
each $A_{U}\le R_{G}(U)$ and $A_{U}$ has no fixed point in $U$.

Suppose in addition:
\begin{description}
\item [{(i)}] all finite index subgroups of $R_{G}(U)$ contain $A_{U}$; 
\item [{(ii)}] the collection of groups $A_{U}$ satisfies that for any
$U,V\in\mathcal{U}$ such that $U\cap V\neq\emptyset$, $A_{U\cup V}=\left\langle A_{U},A_{V}\right\rangle $
;
\item [{(iii)}] any ergodic $G$-invariant probability measure on the space
$F(X)$ other than $\delta_{X}$ is supported on the subspace of closed
subsets with empty interior.
\end{description}
Let $\mu$ be an ergodic IRS of $G$, then $\mu$-a.e. $H$ satisfies
that 
\[
H\ge\left\langle \cup\left\{ A_{U}:\ U\in\mathcal{U}\mbox{ and }U\subseteq X\setminus{\rm Fix}(H)\right\} \right\rangle .
\]

\end{proposition}

\begin{proof}

Define $\mathcal{V}_{H}$ to be the collection
\[
\mathcal{V}_{H}:=\{V\subseteq X:\ V\mbox{ is open and any clopen subset }W\mbox{ of }V\mbox{ satisfies }A_{W}\le H\}.
\]
In other words, an open set $V$ is in $\mathcal{V}_{H}$ if and only
if it can be written as an increasing union of clopen sets $\cup_{n}W_{n}$
where $A_{W_{n}}\le H$ for each $n$. Equip $\mathcal{V}_{H}$ with
the partial order of set inclusions. By definition it is clear that
if $\{V_{i}\}_{i\in I}$ is a chain in $\mathcal{V}_{H}$, then the
union $\cup_{i\in I}V_{i}$ is in $\mathcal{V}_{H}$. Assumption (ii)
implies that if $V$ and $V'$ are in $\mathcal{V}_{H}$ and $V\cap V'\neq\emptyset$,
then $V\cup V'\in\mathcal{V}_{H}$ as well. 

Let $\mathcal{M}_{H}=\{M_{j}\}_{j\in J}$ be the collection of maximal
elements in $\mathcal{V}_{H}$. Note that any two distinct maximal
subsets $M_{j}$ and $M_{j'}$ must be disjoint, otherwise $M_{j}\cup M_{j'}$
would be in $\mathcal{V}_{H}$ which contradicts the maximality of
$M_{j},M_{j'}$. Since $\mathcal{U}$ is countable, as a consequence
of disjointness, we have that the index set $J$ is countable. For
convenience of notation we index this collection by $J=\mathbb{N}$
(if it is a finite collection formally add empty sets to the sequence).
Thus we have a map 
\[
H\mapsto\left({\rm Fix}(H),\left\{ M_{j}\right\} _{j\in\mathbb{N}}\right).
\]
 Given a countable collection of open sets $\mathcal{M}=\left(M_{j}\right)$,
let $\tau_{\mathcal{M}}:\{0,1\}^{\mathbb{N}}\to F(X)$ be the map
defined by $(x_{j})\mapsto X\setminus\cup\{M_{j}:\ x_{j}=1,j\in\mathbb{N}\}$.
Recall that $F(X)$ is equipped with the Vietoris topology. It is
easy to verify that the map $\tau_{\mathcal{M}}$ is Borel measurable.
Let $p=1/2$, $\nu$ be the Bernoulli distribution on $\{0,1\}$,
$\nu(\{0\})=\nu(\{1\})=1/2$, and $\nu^{\otimes\mathbb{N}}$ the product
measure on $\{0,1\}^{\mathbb{N}}$. Denote by $\eta_{\mathcal{M}}$
the pushforward of $\nu^{\otimes\mathbb{N}}$ under $\tau_{\mathcal{M}}$.
Note that the measure $\eta_{\mathcal{M}}$ does not depend on ordering
of $M_{j}$. It is routine to check measurability:

\begin{lemma}\label{eta}

Let $\mathcal{P}(F(X))$ be the space of Borel probability measures
on $F(X)$, equipped with the topology of weak convergence. The map
${\rm Sub}(G)\to\mathcal{P}(F(X))$ given by $H\mapsto\eta_{\mathcal{M}_{H}}$
is Borel measurable.

\end{lemma}

\begin{proof}[Proof of Lemma \ref{eta}]

Note that the Borel $\sigma$-field on $F(X)$ is generated by sets
$\{C\in F(X):\ C\subseteq U\}$, where $U$ goes over clopen subsets
of $X$. Thus it suffices to check for any clopen set $U$ and $x\in[0,1]$,
the set $\{H:\ \eta_{\mathcal{M}_{H}}(\{C\in F(X):\ C\subseteq U\})\ge x\}$
is a measurable subset of ${\rm Sub}(G)$. 

Consider all finite cover of $U^{c}$ by disjoint clopen sets. This
is a countable list, denote it by $\mathcal{O}_{1},\mathcal{O}_{2},\ldots$.
By definition of $\eta_{\mathcal{M}_{H}}$, we have 
\begin{align*}
\{H & :\ \eta_{\mathcal{M}_{H}}(\{C\in F(X):\ C\subseteq U\}\ge x)\}\\
= & \{H:\ \exists\mbox{ a sub-collection }\{M_{j_{k}}\}_{k=1}^{\ell}\mbox{ of }\mathcal{M}_{H}\mbox{ that covers }U^{c}\mbox{ and }2^{-\ell}\ge x\}\\
= & \cup_{i:\ |\mathcal{O}_{i}|\le-\log_{2}x}\left\{ H:\ A_{V}\le H\mbox{ for every }V\in\mathcal{O}_{i}\right\} ,
\end{align*}
where $\left|\mathcal{O}_{i}\right|$ denotes the number of sets in
the cover $\mathcal{O}_{i}$. Since the set $\{H:\ A_{V}\le H\}$
is measurable, we conclude that $\{H:\ \eta_{\mathcal{M}_{H}}(\{C\in F(X):\ C\subseteq U\}\ge x)\}$
is measurable.

\end{proof}

We continue the proof of the proposition. Take the probability measure
on $F(X)$ given by 
\[
\nu=\int_{{\rm Sub}(G)}\eta_{\mathcal{M}_{H}}d\mu(H).
\]
Note that for $g\in G$, the decomposition associated with the conjugation
$g^{-1}Hg$ is $\left({\rm Fix}(H)\cdot g,\left\{ M_{j}\cdot g\right\} _{j\in\mathbb{N}}\right)$.
Since $\mu$ is invariant under conjugation, it follows that the measure
$\nu$ is a $G$-invariant probability measure on $F(X)$: 
\begin{align*}
\nu\cdot g & =\int_{{\rm Sub}(G)}\left(\eta_{\mathcal{M}_{H}}\cdot g\right)d\mu(H)\\
 & =\int_{{\rm Sub}(G)}\left(\eta_{\mathcal{M}_{H}}\cdot g\right)d\mu(g^{-1}Hg)\\
 & =\int_{{\rm Sub}(G)}\left(\eta_{\mathcal{M}_{g^{-1}Hg}}\right)d\mu(g^{-1}Hg)=\nu.
\end{align*}

By assumption (i) on the infinite simple groups $A_{U}$ and Corollary
\ref{simple}, we have that for $\mu$-a.e. $H$, there exists a cover
$\{U_{i}\}_{i\in I}$ of $X-{\rm Fix}_{H}$ by clopen sets such that
$A_{U_{i}}<H$ for every $i\in I$. In particular each $U_{i}\in\mathcal{V}_{H}$.
The assumption that $A_{U}$ has no fixed point in $U$ implies that
a set in $\mathcal{V}_{H}$ must be disjoint from ${\rm Fix}(H)$.
It follows that for $\mu$-a.e. $H$, $\mathcal{M}_{H}$ forms a partition
of $X\setminus{\rm Fix}(H)$ into disjoint open sets. Next we show
that under assumption (iii), for $\mu$-a.e. $H$, the partition $\mathcal{M}_{H}$
consists of one nonempty set, namely $X\setminus{\rm Fix}(H)$. 

Consider the event $D=\{H:\ H\ge A_{U}\,\mbox{for all clopen }U\subseteq X\setminus{\rm Fix}(H)\}$.
Since $D$ can be expressed as 
\[
D=\cap_{U\in\mathcal{U}}\left(\{H:{\rm Fix}(H)\cap U\neq\emptyset\}\cup\{H:\ A_{U}\le H\}\right),
\]
it is indeed a measurable subset of ${\rm Sub}(G)$. By definition,
it is clear that the set $D$ is invariant under conjugation by $G$.
Ergodicity of $\mu$ implies $\mu(D)\in\{0,1\}$. Note that $\mu(D)=1$
is exactly the statement of the proposition. We need to rule out the
possibility that $\mu(D)=0$. Suppose $\mu(D)=0$, then it means $\mu$-a.e.
the partition $\mathcal{M}_{H}$ has at least two non-empty open sets.
It follows that 
\[
\nu\left(\left\{ C\in F(X):\ C\neq X,\ C\mbox{ contains a nonempty open set}\right\} \right)\ge\frac{1}{4}.
\]
By the ergodic decomposition, this implies that there exist $G$-invariant
ergodic probability measures supported on proper closed subsets of
$X$ with non-empty interior, which contradicts assumption (iii). 

\end{proof}

\subsection{Applications to IRSs of topological full groups and proof of Theorem
\ref{full-1}}

In this subsection we apply Proposition \ref{corfull} to prove Theorem
\ref{full-1}. Let $\mathcal{G}$ be a minimal groupoid of germs with
unit space $X=\mathcal{G}^{(0)}$ homeomorphic to the Cantor set.
Let $\mathcal{U}$ be the collection of clopen subsets of $X$. For
$U\in\mathcal{U}$, take $A_{U}=R_{\mathsf{A}(\mathcal{G})}(U)=\mathsf{A}(\mathcal{G}|_{U})$.
Then by \cite[Theorem 1.1]{Nek-simple}, $A_{U}$ is a simple group
and is contained in every non-trivial normal subgroup of $\mathsf{F}(\mathcal{G}|_{U})$.
Thus assumption (i) in Proposition \ref{corfull} is verified. Assumption
(ii) is satisfied because of Lemma \ref{condition1}. It remain to
verify (iii):

\begin{lemma}\label{condition2}

Suppose $\mathcal{G}$ is a minimal groupoid of germs with unit space
$X$ homeomorphic to the Cantor set. An ergodic $\mathsf{A}(\mathcal{G})$-invariant
probability measure on $F(X)$ that is not $\delta_{X}$ is supported
on the subspace of closed sets with empty interior.

\end{lemma}

\begin{proof}

Let $\nu$ be an ergodic $\mathsf{A}(\mathcal{G})$-invariant probability
measure on $F(X)$ such that $\nu\neq\delta_{X}$. 

Let $U,V$ be two disjoint clopen subsets of $X$. We first show that
there exist an infinite sequence of group elements $\{g_{n}\}$ in
$\mathsf{A}(\mathcal{G})$ and pairwise disjoint clopen subsets $V_{1},V_{2},\ldots$
of $V$ such that 
\begin{align*}
U & \supsetneqq V_{1}\cdot g_{1}\supsetneqq V_{2}\cdot g_{2}\supsetneqq\ldots\\
\mbox{and } & {\rm supp}\ g_{n}\subsetneqq\left(V_{n-1}\cdot g_{n-1}\right)\cup V_{n}.
\end{align*}
Such elements can be chosen recursively by applying \cite[Lemma 3.2]{Nek-simple}.
Indeed, suppose we have chosen $g_{1},\ldots,g_{n-1}$, take a point
$x_{1}\in V\setminus\cup_{j=1}^{n-1}V_{j}$ and two points $x_{2},x_{3}\in V_{n-1}\cdot g_{n-1}$
which are in the orbit of $x_{1}$. Then there is a $\mathcal{G}$-bisection
$H_{2}$ (resp. $H_{3}$) such that $H_{2}$ contains an element $\gamma_{2}$
(resp. $\gamma_{3}$) such that $x_{2}=x_{1}\cdot\gamma_{2}$ (resp.
$x_{3}=x_{1}\cdot\gamma_{3}$). Take a small clopen neighborhood $V'_{n}$
of $x_{1}$ in $V\setminus\cup_{j=1}^{n-1}V_{j}$ and $W_{n}$ of
$\{x_{2},x_{3}\}$ in $V_{n-1}\cdot g_{n-1}$ such that $V\neq V_{1}\cup\ldots\cup V_{n}$
and $W_{n}\neq V_{n-1}\cdot g_{n-1}$. Then one can find a smaller
neighborhood $V_{n}$ of $x_{1}$ such that $V_{n}\subseteq V_{n}'\cap\mathsf{s}(H_{2})\cap\mathsf{s}(H_{3})$,
moreover $\mathsf{r}(V_{n}H_{2})\cup\mathsf{r}(V_{n}H_{3})\subseteq W_{n}$.
Take the multisection 
\[
F_{n}=\left(\begin{array}{ccc}
V_{n} & V_{n}H_{2} & V_{n}H_{3}\\
(V_{n}H_{2})^{-1} & \mathsf{r}(V_{n}H_{2}) & (V_{n}H_{2})^{-1}V_{n}H_{3}\\
(V_{n}H_{3})^{-1} & (V_{n}H_{3})^{-1}(V_{n}H_{2}) & \mathsf{r}(V_{n}H_{3})
\end{array}\right).
\]
Then the $g_{n}\in\mathsf{A}(F_{n})$ which acts as a $3$-cycle permuting
$x_{1},x_{2},x_{3}$ cyclically satisfies the requirements. 

Write
\[
B_{U}^{V}:=\{C\in\mathcal{C}(X):\ U\subseteq C\subseteq X\setminus V\}.
\]
It follows that the translates $\left\{ B_{U}^{V}\cdot g_{n}\right\} _{n\in\mathbb{N}}$
are pairwise disjoint. Indeed, if $C\in B_{U}^{V}\cdot g_{i}$ then
$C\cap(V\cdot g_{i})=\emptyset$. For $j>i$, $U\cdot g_{j}\supseteq U\cap\left({\rm supp}g_{j}\right)^{c}$,
where by choice of $g_{j}$, $U\cap\left({\rm supp}g_{j}\right)^{c}$
contains a non-empty subset of $V_{j-1}\cdot g_{j-1}$. It follows
that 
\[
\left(U\cdot g_{j}\right)\cap\left(V\cdot g_{i}\right)\supseteq\left(U\cdot g_{j}\right)\cap\left(V_{i}\cdot g_{i}\right)\supseteq\left(U\cdot g_{j}\right)\cap\left(V_{j-1}\cdot g_{j-1}\right)\neq\emptyset.
\]
Therefore a set $C\in B_{U}^{V}\cdot g_{j}$ has non-empty intersection
with $V\cdot g_{i}$, thus it is not in $B_{U}^{V}\cdot g_{i}$. 

By invariance of $\nu$, we have $1\ge\sum_{n\in\mathbb{N}}\nu(B_{U}^{V}\cdot g_{n})=\sum_{n\in\mathbb{N}}\nu(B_{U}^{V})$.
Therefore 
\[
\nu(B_{U}^{V})=0.
\]

Let $F^{'}(X)=\{C\in F(X):\ C\neq X,\ C\mbox{ contains a nonempty open subset}\}$.
For each $C\in F'(X)$, one can find clopen sets $U\subseteq C$ and
$V\subseteq X\setminus C$, thus $C\in B_{U}^{V}$. Therefore $\left\{ B_{U}^{V}\right\} $,
where $U,V$ are taken over disjoint clopen subsets of $X$, form
a countable cover of $F'(X)$. Thus $\nu(B_{U}^{V})=0$ for any pair
of disjoint clopen sets $U,V$ implies $\nu\left(F'(X)\right)=0$.
Since $\nu\neq\delta_{X}$, we conclude that $\nu$ is supported on
subsets with empty interior. 

\end{proof}

With these two lemmas we can apply Proposition \ref{corfull} to prove
Theorem \ref{full-1}.

\begin{proof}[Proof of Theorem \ref{full-1}]

Take $A_{U}=R_{\mathsf{A}(\mathcal{G})}(U)$. As explained at the
beginning of this subsection, \cite[Theorem 1.1]{Nek-simple} implies
that $A_{U}$ is simple; moreover $A_{U}$ has no fixed points in
$U$ since $\mathcal{G}$ is minimal. Lemma \ref{condition1} and
\ref{condition2} verify condition (ii) and (iii) respectively. It
follows from Proposition \ref{corfull} that $\mu$-a.e. $H$ contains
the subgroup $T$ that is generated by all $R_{\mathsf{A}(\mathcal{G})}(U)$,
where $U$ ranges over all clopen subsets that are contained in $X\setminus{\rm Fix}(H)$.
It's clear that $T\le R_{\mathsf{A}(\mathcal{G})}\left(X\setminus{\rm Fix}(H)\right)$,
while the other direction $R_{\mathsf{A}(\mathcal{G})}\left(X\setminus{\rm Fix}(H)\right)\le T$
follows from definition of $\mathsf{A}(\mathcal{G})$ (c.f. Fact \ref{rigidstab}). 

\end{proof}

The following corollary of Theorem \ref{full-1} is previously known
as a consequence of \cite[Theorem 2.9]{Dudko-Medynets2} on indecomposable
characters.

\begin{corollary}\label{properlyinfinite}

Let $\mathcal{G}$ be a minimal groupoid of germs whose unit space
$X$ is homeomorphic to the Canter set. Assume that the action of
$\mathsf{A}(\mathcal{G})$ on $X$ is compressible in the sense that
for any clopen sets $U$ such that $U\neq X$, there exists $g_{1},g_{2}\in\mathsf{A}(\mathcal{G})$
such that $\left(U\cdot g_{1}\right)\cup\left(U\cdot g_{2}\right)\subseteq U$
and $\left(U\cdot g_{1}\right)\cap\left(U\cdot g_{2}\right)=\emptyset$.
Then $\mathsf{A}(\mathcal{G})$ does not admit any IRS other than
$\delta_{\{id\}}$ and $\delta_{\mathsf{A}(\mathcal{G})}$. 

\end{corollary}

\begin{proof}

By Theorem \ref{full-1}, it suffices to show that the only $\mathsf{A}(\mathcal{G})$-invariant
probability measures on $F(X)$ are $\delta$-masses at $\emptyset$
or $X$. Let $\nu$ be an $\mathsf{A}(\mathcal{G})$-invariant probability
measures on $F(X)$. For a clopen set $U$ such that $U\neq X$, by
the compressibility assumption, we can find $g_{1},g_{2}\in\mathsf{A}(\mathcal{G})$
such that $U\cdot g_{1}\subseteq U$, $U\cdot g_{2}\subseteq U$ and
$\left(U\cdot g_{1}\right)\cap\left(U\cdot g_{2}\right)=\emptyset$.
Denote by $C_{U}=\{C\in F(X):\ C\neq\emptyset,\ C\subseteq U\}$.
Since $\nu$ is an invariant measure, we have that $\nu(C_{U})=\nu(C_{U\cdot g_{1}})=\nu(C_{U\cdot g_{2}})$.
On the other hand since $C_{U\cdot g_{1}}$ and $C_{U\cdot g_{2}}$
are disjoint, we have $\nu(C_{U})\ge\nu(C_{U\cdot g_{1}})+\nu(C_{U\cdot g_{2}})$.
Therefore $\nu(C_{U})=0$ for any clopen set $U\neq X$. The conclusion
follows. 

\end{proof}

\begin{example}

Following \cite[Definition 4.9]{Matui2015}, we say that a clopen
set $A\subseteq\mathcal{G}^{(0)}$ is \emph{properly infinite} if
there exists bisections $F_{1},F_{2}$ such that $\mathsf{s}(F_{1})=\mathsf{s}(F_{2})=A$,
$\mathsf{r}(F_{1})\cup\mathsf{r}(F_{2})\subseteq A$ and $\mathsf{r}(F_{1})\cap\mathsf{r}(F_{2})=\emptyset$.
If every clopen subset of $X$ is properly infinite, then the groupoid
$\mathcal{G}$ is said to be \emph{purely infinite}. Examples of purely
infinite minimal groupoids include the groupoid of germs of $\Gamma\curvearrowright X$
where the action of the countable group $\Gamma$ on the Cantor set
$X$ is $n$-filling in the sense of \cite{Jolissaint-Robertson},
groupoids arising from a one-sided irreducible shift of finite type,
see \cite{Matui2015} for more details. By \cite[Lemma 4.13]{Matui2015}
a minimal purely infinite groupoid $\mathcal{G}$ satisfies the assumption
of Corollary \ref{properlyinfinite}. Note that by \cite[Theorem 4.16]{Matui2015},
in this case $\mathsf{A}(\mathcal{G})=[\mathsf{F}(\mathcal{G}),\mathsf{F}(\mathcal{G})]$.
It follows that if $\mu$ is an ergodic IRS of $\mathsf{F}(\mathcal{G})$,
where $\mathcal{G}$ is a minimal purely infinite groupoid of germs,
then either $\mu=\delta_{\{id\}}$ or $\mu$-a.e. $H$ satisfies $H\ge\left[\mathsf{F}(\mathcal{G}),\mathsf{F}(\mathcal{G})\right]$. 

\end{example}

The rest of the section is devoted to cases opposite to Corollary
\ref{properlyinfinite}, more precisely, examples where invariant
measures on $F(X)$ can be studied through subgroups that are locally
finite.

\subsection{Examples of classifications of IRSs of alternating full groups\label{subsec:fullexamples}}

Combine Theorem \ref{full-1}, Lemma \ref{kset} and Lemma \ref{invariant measure},
we have the following:

\begin{corollary}\label{full}

Let $\mathcal{G}$ be a minimal groupoid of germs with unit space
$X=\mathcal{G}^{(0)}$ homeomorphic to the Cantor set. Assume that
there is an open sub-groupoid $\mathcal{H}\subseteq\mathcal{G}$ such
that $\mathcal{H}$ is a minimal AF-groupoid with unit space $\mathcal{H}^{(0)}=\mathcal{G}^{(0)}$.
Let $\mu$ be an ergodic IRS of $\mathsf{A}(\mathcal{G})$, $\mu\neq\delta_{\{id\}},\delta_{\mathsf{A}(\mathcal{G})}$.
Then there exists an integer $k\in\mathbb{N}$ and an $\mathsf{A}(\mathcal{G})$-invariant
ergodic probability measure $\nu$ on $X^{(k)}$, such that $\mu$
is the pushforward of $\nu$ under the map
\begin{align*}
X^{(k)} & \to{\rm Sub}(\mathsf{A}(\mathcal{G}))\\
\{x_{1},\ldots,x_{k}\} & \mapsto\cap_{i=1}^{k}{\rm St}_{\mathsf{A}(\mathcal{G})}(x_{i}).
\end{align*}

\end{corollary}

\begin{proof}

Since $\mathcal{H}$ is an open sub-groupoid of $\mathcal{G}$ with
$\mathcal{H}^{(0)}=\mathcal{G}^{(0)}$, we have $\mathsf{A}(\mathcal{H})\le\mathsf{A}(\mathcal{G})$.
Let $\nu$ be an $\mathsf{A}(\mathcal{G})$-invariant ergodic probability
measure on $F(X)$, $\nu\neq\delta_{\emptyset},\delta_{X}$. Then
it follows from Lemma \ref{kset} that there exists a constant $k\in\mathbb{N}$
such that $\nu$ is supported on the subspace $X^{(k)}$ of finite
sets of cardinality $k$. Indeed, since $\nu$ is also $\mathsf{A}(\mathcal{H})$-invariant,
by the ergodic decomposition theorem, we have $\nu=\int\eta_{\theta}d\lambda(\theta)$,
where $\eta_{\theta}$'s are ergodic $\mathsf{A}(\mathcal{H})$-invariant
probability measures. Since the cardinality function $|C|$ is $\mathsf{A}(\mathcal{G})$-invariant,
it follows there exists $k\in\mathbb{N}\cup\{0,\infty\}$ such that
$\nu$ and $\eta_{\theta}$'s are concentrated on the subspace $X^{(k)}$.
Then Lemma \ref{kset} implies that $k\in\mathbb{N}$. 

Let $\nu$ be the distribution of ${\rm Fix}(H)$, where $H$ is an
ergodic IRS of $\mathsf{A}(\mathcal{G})$ with distribution $\mu$.
Since the action of $\mathsf{A}(\mathcal{G})$ on $X$ is faithful,
$\nu=\delta_{X}$ implies $\mu=\delta_{\{id\}}$. By Theorem \ref{full-1},
$\mu$-a.e. $H\ge R_{\mathsf{A}\left(\mathcal{G}\right)}(X\setminus{\rm Fix}(H))$,
thus $\nu=\delta_{\emptyset}$ implies $\mu=\delta_{\mathsf{A}(\mathcal{G})}$.
When $\mu\neq\delta_{\{id\}},\delta_{\mathsf{A}(\mathcal{G})}$, as
explained above, $\nu$ is supported on finite sets of cardinality
$k$, for some $k\in\mathbb{N}$. In this case $R_{\mathsf{A}\left(\mathcal{G}\right)}(X\setminus\{x_{1},\ldots,x_{k}\})=\{g\in\mathsf{A}(\mathcal{G}):\ x_{i}\cdot g=x_{i}\mbox{ for all }1\le i\le k\}$,
the statement follows. 

\end{proof}

Under the assumption of Corollary \ref{full}, IRSs of $\mathsf{A}(\mathcal{G})$
other than $\delta_{\{id\}},\delta_{\mathsf{A}(\mathcal{G})}$ are
in one-to-one correspondence with $\mathsf{A}(\mathcal{G})$-invariant
ergodic probability measures on the space of non-empty finite subsets
of $X$. It seems an interesting question whether the conclusion of
Corollary \ref{full} is true for general minimal groupoid of germs,
that is, without the assumption of containing an open minimal AF sub-groupoid. 

In the rest of this subsection we consider the topological full group
of a minimal $\mathbb{Z}^{d}$ action on the Cantor set, where Corollary
\ref{full} applies.

For minimal Cantor $\mathbb{Z}$-actions, a systematic study of topological
orbit equivalence was completed in the works \cite{HPS,GPS}. In \cite{Forrest}
it is shown for $d\ge2$ that the orbit equivalence relation of a
free minimal $\mathbb{Z}^{d}$-action on the Cantor set contains a
\textquotedbl large\textquotedbl{} AF sub-equivalence relation. In
\cite{GMPS1,GMPS2} it is shown (for $d=2$ and $d\ge3$ resp.) that
every minimal action of $\mathbb{Z}^{d}$ action on the Cantor set
is topologically orbit equivalent to an AF relation. For our purpose,
it suffices to apply results from \cite{Forrest} so that Corollary
\ref{full}, Lemma \ref{kset} and Lemma \ref{invariant measure}
can be used. 

We are now ready to prove the classification result stated in Corollary
\ref{classificationZd}.

\begin{proof}[Proof of Corollary \ref{classificationZd}]

Let $\mathcal{G}$ be the transformation groupoid of the action $\varphi:\mathbb{Z}^{d}\to{\rm Homeo}(X)$.
First assume that the action $\varphi$ is minimal and free. For $d=1$,
by the model theorem (see \cite{HPS}), up to topological conjugacy,
every minimal action of $\mathbb{Z}$ on a Cantor set arises as the
Bratteli-Vershik map of some properly ordered, simple Bratteli diagram.
For $d\ge2$, by \cite[Theorem 1.2]{Forrest} there exists an open
AF-subgroupoid $\mathcal{H}$ of $\mathcal{G}$ with the same unit
space such that $\mathcal{H}$ is associated with a simple Bratteli
diagram $B$, and the set $Y$, which consists of points whose $\mathcal{H}$-orbit
is not equal to its $\mathcal{G}$-orbit, is a set of zero measure
with respect to any $\mathcal{H}$-invariant measure on $X$. The
set $Y$ can be chosen to be $\mathbb{Z}^{d}$-invariant and $\mathcal{H}$-invariant. 

Therefore the conditions in Corollary \ref{full} are satisfied by
$\mathcal{G}$, we have that for $\mathsf{A}(\mathcal{G})$, ergodic
IRSs other than $\delta_{\{id\}},\delta_{\mathsf{A}(\mathcal{G})}$
are in one-to-one correspondence to ergodic invariant probability
measures on $X^{(k)}$, $k\in\mathbb{N}$. It remains to show that
an ergodic invariant measure on $X^{(k)}$ must be of the form in
the statement. Observe that $M\left(X^{(k)},\mathsf{A}(\mathcal{G})\right)=M\left(X^{(k)},\mathsf{A}(\mathcal{H})\right)$
for any $k\in\mathbb{N}$. Indeed, let $\nu$ be an $\mathsf{A}(\mathcal{H})$-invariant
measure on $X^{(k)}$. Since the set $Y$ has measure zero with respect
to any invariant measure, we have that $\nu\left(\left\{ C:C\cap Y=\emptyset\right\} \right)=1$.
Let $W$ be an open set in $X^{(k)}$ and $W'=\{C\in W:C\cap Y=\emptyset\}$,
let $g\in\mathsf{A}(\mathcal{G})$ be a group element. Then there
exists a countable partition of $W'$ into measurable subsets $W'=\cup_{h\in\mathsf{A}(\mathcal{H})}B_{h}$
such that for $C\in B_{h}$, $C\cdot g=C\cdot h$. Then $\nu(W\cdot g)=\nu(W'\cdot g)=\sum_{h}\nu(B_{h}\cdot h)=\sum\nu(B_{h})=\nu(W').$
Thus $\nu$ is $\mathsf{A}(\mathcal{G})$-invariant. Apply Lemma \ref{invariant measure}
to $\mathsf{A}(\mathcal{H})$-invariant measures, we obtain the statement. 

As pointed out in \cite{GMPS1,GMPS2} the freeness assumption on the
$\mathbb{Z}^{d}$ action can be dropped for the following reason.
Given a minimal $\mathbb{Z}^{d}$-action on $X$, it can be seen as
a free action of a quotient group $\mathbb{Z}^{d}/H$. As shown in
\cite[Theorem A.1]{GMPS2}, if the free abelian rank of $\mathbb{Z}^{d}/H$
is $d'$, $d'\ge1$, then there exists a free action of $\mathbb{Z}^{d'}$
on $X$ with the same topological full group. Thus a minimal action
of $\mathbb{Z}^{d}$ can be reduced to the case of a free minimal
action. 

\end{proof}

\begin{example}\label{IET}

An interval exchange map $f$ on the unit interval $I=[0,1)$ is obtained
by cutting $I$ into subintervals and reordering them. The map $f$
is a translation on each sub-interval. Following notations of \cite{JMMS},
we identify the end points $a,b$ and regard $f$ as a right-continuous
permutation of $\mathbb{R}/\mathbb{Z}.$ The set of translation length
(or angle), $\{x\cdot f-x:x\in\mathbb{R}/\mathbb{Z}\}$ is finite.
Dynamic properties of interval exchange maps have attracted a lot
of research in the past decades, see for instance the survey \cite{Viana}
and references therein.

The connection between interval exchange groups and topological full
groups is explained in \cite{JMMS}. Let $\Lambda$ be a finitely
generated infinite subgroup of $\mathbb{R}/\mathbb{Z}$ and $\Sigma=\{x_{1},\ldots,x_{r}\}$
be a finite subset of $\mathbb{R}/\mathbb{Z}$. The group ${\rm IET}(\Lambda,\Sigma)$
is defined as the collection of interval exchange transforms $f$
such that the extrema of the defining intervals of $f$ lie in the
cosets $x_{i}+\Lambda$, $x_{i}\in\Sigma$ and $x\cdot f-x\in\Lambda$
for all $x\in\mathbb{R}/\mathbb{Z}$. In \cite[Proposition 5.11]{JMMS}
it is shown that there is a minimal Cantor $\Lambda$-system $\varphi:\Lambda\to{\rm Homeo}(X)$
such that its topological full group is isomorphic to ${\rm IET}(\Lambda,\Sigma)$.
Thus we can apply Corollary \ref{classificationZd} to ${\rm IET}(\Lambda,\Sigma)$,
regarded as the topological full group of a minimal Cantor $\Lambda$-system.

Assume that $\Sigma+\Lambda$ contains some irrational numbers and
$0\in\Sigma+\Lambda$. Then ${\rm IET}(\Lambda,\Sigma)$ contains
irrational rotations and the only invariant measure on $\mathbb{R}/\mathbb{Z}$
is the Lebesgue measure. Corollary \ref{classificationZd} implies
that when $\Sigma+\Lambda$ contains an irrational number and $0\in\Sigma+\Lambda$,
the list of IRSs of the derived subgroup of ${\rm IET}(\Lambda,\Sigma)$
is
\begin{itemize}
\item (atomic) $\delta_{\{id\}}$, $\delta_{\Gamma'}$;
\item (non-atomic) stabilizer IRSs of diagonal actions on $\left((\mathbb{R}/\mathbb{Z})^{k},m^{k}\right)$,
where $m$ is the Lebesgue measure on $\mathbb{R}/\mathbb{Z}$, $k\in\mathbb{N}$. 
\end{itemize}
\end{example}

In \cite{Nek-palindromic} Nekrashevych constructs first examples
of simple groups of intermediate growth. These groups are obtained
as alternating full groups of fragmentations of certain non-free expansive
actions of dihedral group on the Cantor set. Corollary \ref{full}
can be applied to classify IRSs of some examples of such groups. For
instance, consider the explicit example of a group $F$ which is obtained
by fragmentation of the golden mean dihedral action in \cite[Section 8]{Nek-palindromic}.
It contains an LDA subgroup which acts minimally on the Cantor set
(denoted by $\mathsf{A}_{\omega}$ in \cite[Subsection 8.2]{Nek-palindromic}).
Similar to Corollary \ref{classificationZd} one can classify IRSs
of $F'$, we don't spell out the details here.

\section{Applications to weakly branch groups\label{sec:rootedtree}}

Let $\mathsf{T}$ be a rooted spherically symmetric tree and $\Gamma$
is a countable subgroup of ${\rm Aut}(\mathsf{T})$. In this case,
$\Gamma$ is residually finite and let $\bar{\Gamma}$ be its profinite
completion. Equip $X=\partial\mathsf{T}$ with distance $d(x,y)=2^{-\left|x\wedge y\right|}$,
where $x\wedge y$ denotes the longest common prefix of $x$ and $y$.
Then $\Gamma$ acts faithfully on $\left(\partial\mathsf{T},d\right)$
by isometry. 

In this section we use the notations for groups acting on trees introduced
before the statement of Corollary \ref{branch} in the Introduction.
Recall that given a closed subset $K$ of $\partial\mathsf{T}$, we
associate to it an index set $I_{K}\subseteq\mathsf{T}$. One can
visualize the set $I_{K}$ on the tree as follows, the same kind of
coloring is considered earlier in \cite{Bencs-Toth}. 

\begin{definition}[Coloring of the tree according to $K$]\label{color}

Let $K\subseteq\partial\mathsf{T}$ be a closed subset. Color the
vertices of the tree $\mathsf{T}$ according to $K$ as follows:
\begin{itemize}
\item a vertex $v$ is colored red if $H$ has no fixed point in the cylinder
$C_{v}$, that is, $C_{v}\cap{\rm Fix}_{\partial\mathsf{T}}(H)=\emptyset$;
\item $v$ is colored green if the cylinder $C_{v}$ is contained in ${\rm Fix}_{\partial\mathsf{T}}(H)$;
\item otherwise $v$ is colored blue. 
\end{itemize}
It is clear that children of a red (green resp.) vertex must be red
(green resp.). We say $v=x_{1}\ldots x_{k}$ is a \emph{maximal red
}vertex if its parent $x_{1}\ldots x_{k-1}$ is not red. With the
coloring interpretation of $K$ as above, the index set $I_{K}$ is
exactly the set of maximal red vertices in $\mathsf{T}$ colored according
to $K$. 

\end{definition}

We now deduce Corollary \ref{branch} stated in the Introduction from
Theorem \ref{inclu-rigid}. 

\begin{proof}[Proof of Corollary \ref{branch}]

Let $\Gamma$ be a weakly branch group acting on the rooted tree $\mathsf{T}.$
Recall that by definition of weakly branching, it means $\Gamma$
acts level transitively on $\mathsf{T}$ and the rigid stabilizer
${\rm Rist}_{\Gamma}(u)=R_{\Gamma}(C_{u})$ is nontrivial for every
vertex $u\in\mathsf{T}.$ 

First note that ${\rm Rist}_{\Gamma}(u)$ does not have any fixed
point in the cylinder set $C_{u}$. Indeed the stabilizer of the vertex
$u$, ${\rm St}_{\Gamma}(u)$, is in the normalizer of ${\rm Rist}_{\Gamma}(u)$
in $\Gamma$. By level transitivity of action of $\Gamma$, we have
that ${\rm St}_{\Gamma}(u)$ acts level transitively on the subtree
$\mathsf{T}_{u}$ rooted at $u$. Thus the orbit $x\cdot{\rm St}_{\Gamma}(u)$
is dense in $C_{u}$ for any $x\in C_{u}$. By Remark \ref{nofixed},
${\rm Rist}_{\Gamma}(u)$ cannot have any fixed point in $C_{u}$. 

Given a subgroup $H$, write $K={\rm Fix}(H)$ and $\partial\mathsf{T}-K=\cup_{x\in I_{K}}C_{x}$
as the decomposition into disjoint union of maximal cylinders. Given
a length $\ell\in\{|x|:\ x\in I_{K}\}$, consider all vertices $x$
in the index set $I_{K}$ such that $|x|=\ell$. Define $m_{\ell}(H)$
as the smallest integer $m$ such that 
\[
H\ge\prod_{x\in I_{K},|x|=\ell}\left[{\rm Rist}_{m}^{\Gamma}(\mathsf{T}_{x}),{\rm Rist}_{m}^{\Gamma}(\mathsf{T}_{x})\right].
\]
Note that $m_{\ell}(H)$ is invariant under conjugation by $\Gamma$:
$m_{\ell}(H)=m_{\ell}(g^{-1}Hg)$. Indeed, ${\rm Fix}(g^{-1}Hg)={\rm Fix}(H)\cdot g$
and the set $I_{K\cdot g}=I_{K}\cdot g$. Thus we have that $g^{-1}Hg\ge\prod_{x\in I_{K\cdot g},|x|=\ell}\left[{\rm Rist}_{m}(\mathsf{T}_{x}),{\rm Rist}_{m}(\mathsf{T}_{x})\right]$,
with $m=m_{\ell}(H)$. 

Since the set $\{|x|,x\in I_{K}\},$ $K={\rm Fix}(H)$ is invariant
under conjugation by $\Gamma$, it follows from ergodicity of $\mu$
that it is a.e. a fixed set $P$. By Theorem \ref{inclu-rigid}, we
have that for $\mu$-a.e. $H$, there is a countable open cover $\{V_{i}\}_{i\in I}$
of $\partial\mathsf{T}-{\rm Fix}(H)$ such that $\left[R_{\Gamma}(V_{i}),R_{\Gamma}(V_{i})\right]\le H$
for each $i\in I$. For each cylinder $C_{x}$ in the decomposition
of $\partial\mathsf{T}-{\rm Fix}(H)$, since $C_{x}$ is compact we
can take a finite sub-cover of $C_{x}$ from $\{V_{i}\}$ and for
each $V_{i}$ in the sub-cover, shrink it to $V_{i}\cap C_{x}$. It
follows that for each $C_{x}$, there exists a finite level $m_{x}$
such that $H\ge\left[{\rm Rist}_{m_{x}}(\mathsf{T}_{x}),{\rm Rist}_{m_{x}}(\mathsf{T}_{x})\right]$.
It is clear by definition of $m_{\ell}(H)$ that for each $\ell\in P$,
$m_{\ell}(H)\le\max_{x\in I_{K},|x|=\ell}m_{x}$, thus $m_{\ell}(H)$
is a finite integer. Finally since $m_{\ell}(H)$ is a conjugation
invariant function, ergodicity of $\mu$ implies that $m_{\ell}(H)$
is $\mu$-a.e. a finite constant. We have proved the statement. 

\end{proof}

Based on Corollary \ref{branch}, we introduce the following notation. 

\begin{notation}\label{branch-notation}

For a subgroup $H\in{\rm Sub}(\Gamma)$, write ${\rm Fix}(H)=\{x\in\partial\mathsf{T}:\ x\cdot h=x\mbox{ for all }h\in H\}$
for its fixed point set on $\partial\mathsf{T}$. Given a closed set
$K\subseteq\partial\mathsf{T}$, write ${\rm Fix}_{\Gamma}(K):=\{g\in\Gamma:\ x\cdot g=x\mbox{ for all }x\in K\}$
for its pointwise stabilizer in $\Gamma$. Give a sequence of positive
integers ${\bf m}=(m_{i})$, write 
\begin{align*}
R_{K,{\bf m}} & :=\bigoplus_{x\in I_{K}}{\rm Rist}_{m_{|x|}}^{\Gamma}(\mathsf{T}_{x}),\\
\bar{\Gamma}_{K,{\bf m}} & :={\rm Fix}_{\Gamma}(K)/\left[R_{K,{\bf m}},R_{K,{\bf m}}\right],
\end{align*}
where the index set $I_{K}$ is the collection of maximal red vertices
in $\mathsf{T}$ colored according to $K$ as in Definition \ref{color}. 

\end{notation}

The distribution of ${\rm Fix}(H)$ in the space $F(\partial\mathsf{T})$
of closed subsets of $\partial\mathsf{T}$ is known. In our setting
$\Gamma$ is a group acting faithfully and level transitively on a
rooted tree $\mathsf{T}$. Then $\Gamma$ is residually finite, denote
by $\bar{\Gamma}$ its profinite completion and $\eta$ the Haar measure
on the profinite group $\bar{\Gamma}$. Then by \cite[Lemma 2.4]{Bencs-Toth},
for an ergodic IRS $\mu$ of $\Gamma$, there exists a closed set
$K\subseteq\partial\mathsf{T}$ such that 
\[
\mu({\rm Fix}(H)\in B)=\eta\left(\left\{ g\in\bar{\Gamma}:\ K\cdot g\in B\right\} \right)\mbox{ for any }B\in\mathcal{B}(F(\partial\mathsf{T})).
\]
That is, the distribution of ${\rm Fix}(H)$ is the pushforward of
the Haar measure on $\bar{\Gamma}$ under $\bar{g}\mapsto K\cdot\bar{g}$,
where $K$ is a deterministic closed set in $\partial\mathsf{T}$.
When $K$ is not clopen, the orbit of $K$ under $\bar{\Gamma}$ is
infinite. In this case as considered in \cite{Bencs-Toth}, one can
take the IRS of $\Gamma$ defined by $\left\{ {\rm Fix}_{\Gamma}(K\cdot\bar{g})\right\} $
where the distribution of $\bar{g}$ is the Haar measure $\eta$.
When $\Gamma$ is weakly branch, this construction gives a continuum
of atomless IRSs when $K$ varies over closed but not clopen subsets
of $\partial\mathsf{T}$. 

Denote by $\mathbb{P}_{\mu}:F(\partial\mathsf{T})\times\mathcal{B}({\rm Sub}(\Gamma))\to[0,1]$
the regular conditional distribution of $H$ given its fixed point
set ${\rm Fix}(H)$. Then Theorem \ref{branch} states that there
exists a sequence of positive integers ${\bf m}=(m_{i})$ such that
for $\mu$-a.e. $H$, $\mathbb{P}_{\mu}({\rm Fix}(H),\cdot)$ is pulled
back from an IRS of the quotient group $\bar{\Gamma}_{{\rm Fix}(H),{\bf m}}$.

Recall that a group $\Gamma$ acting on the rooted tree $\mathsf{T}$
is said to be a branch group if it acts level transitively and the
level rigid stabilizers ${\rm Rist}_{m}(\mathsf{T})$ have finite
index in $\Gamma$ for all $m\in\mathbb{N}$. A group is just infinite
if it is infinite and all of its proper quotient groups are finite.
By the characterization in \cite{Gri-justinfinite}, $\Gamma$ is
just infinite if and only if ${\rm Rist}_{m}^{\Gamma}(\mathsf{T}_{x})$
has finite abelianization for any $x\in\mathsf{T}$ and $m\in\mathbb{N}$. 

\begin{corollary}\label{branch-clopen}

Let $\mu$ be an ergodic IRS of a branch group $\Gamma$ such that
$\mu$-a.e. ${\rm Fix}(H)$ is clopen (possibly empty). Then $\mu$
is atomic. Suppose in addition $\Gamma$ is just infinite, then $\mu$-a.e.
$H$ is a finite index subgroup of ${\rm Fix}_{\Gamma}(C)$, where
$C={\rm Fix}(H)$ (which is assumed to be clopen a.e.). 

\end{corollary}

\begin{proof}

Since $\Gamma$ is assumed to be a branch group, we have that ${\rm Rist}_{\Gamma}(x)$
has no fixed point in the cylinder $C_{x}$, indeed the orbit $z\cdot{\rm Rist}_{\Gamma}(x)$
for infinite for any $z\in C_{x}$. Then Corollary \ref{branch} applies.
Since $\mu$-a.e. ${\rm Fix}(H)$ is clopen, it follows that the distribution
of ${\rm Fix}(H)$ is uniform over translates of a clopen set $C$
and the index set $I_{C}$ is finite. Let ${\bf m}$ be the sequence
of integers provided by Corollary \ref{branch}. Since $I_{C}$ is
finite and $\Gamma$ is branch, we have that the quotient group $\bar{\Gamma}_{C,{\bf m}}$
is virtually abelian. Since a virtually abelian group has only countably
many subgroups, by Corollary \ref{branch} we conclude that $\mu$
is atomic.

Suppose in addition $\Gamma$ is just infinite, then by the characterization
in \cite{Gri-justinfinite} we have that in this case $\bar{\Gamma}_{C,{\bf m}}$
is finite. Since $\mu$-a.e. $H$ contains $\bigoplus_{x\in I_{{\rm Fix}(H)}}{\rm Rist}_{m_{|x|}}(\mathsf{T}_{x})'$,
the statement follows. 

\end{proof}

Recall that according to \cite{Gri-justinfinite}, a group acting
on a rooted tree $\mathsf{T}$ is said to have \emph{the congruence
property} if any finite index subgroup contains a level stabilizer
${\rm St}_{\Gamma}(n)$ for some $n\in\mathbb{N}$. This definition
is in analogy to the classical congruence property for arithmetic
groups, where level stabilizers replace congruence modulo ideals.
By Corollary \ref{branch-clopen} we \textcolor{black}{have that if
$G$ is }a just infinite\textcolor{black}{{} branch group which satisfies
the congruence property, then an ergodic fixed point free IRS of $G$
is atomic and contains a level stabilizer (which is a congruence subgroup
in this context) almost surely.}

The most well-known example of just infinite branch groups is the
first Grigorchuk group $\mathfrak{G}$, which is constructed in \cite{Grigorchuk1980}.
The group $\mathfrak{G}$ is generated by (for notation wreath recursion
see \cite[Chapter 1]{handbook})
\[
a=\varepsilon,\ b=(a,c),\ c=(a,d),\ d=(1,b),
\]
where $\varepsilon$ is the root permutation which permutes the two
subtrees of the root. The group $\mathfrak{G}$ is branch, just infinite
and has the congruence property (see \cite{Gri-justinfinite}). Therefore
by Corollary \ref{branch-clopen}, for an ergodic fixed point free
IRS $\mu$ of $\mathfrak{G}$, there is a level $n$ such that $\mu$-a.e.
$H$ contains the level stabilizer ${\rm St}_{\mathfrak{G}}(n)$. 

The phenomenon that ergodic IRSs with clopen fixed point sets must
be atomic occurs in some examples of weakly branch groups as well.
We illustrate it on the Basilica group, which is an example of weakly
branch groups that are not branch. The Basilica group $\mathfrak{B}$
is introduced in \cite{GrigorchukZuk}. It is generated by two automorphisms

\[
a=(1,b),\ b=(1,a)\varepsilon.
\]
It is shown in \cite{GrigorchukZuk} that $\mathfrak{B}$ weakly branches
over its commutator subgroup $[\mathfrak{B},\mathfrak{B}]$, and $\mathfrak{B}$
is of exponential growth, non-elementary amenable. 

\begin{corollary}\label{basilica}

Suppose $\mu$ is an ergodic IRS of the Basilica group $\mathfrak{B}$
such that $\mu$-a.e. ${\rm Fix}(H)$ is clopen (possibly empty).
Then $\mu$ is atomic.

\end{corollary}

\begin{proof}

We will use the following algebraic fact about $\mathfrak{B}$: $\mathfrak{B}''=\gamma_{3}(\mathfrak{B})\times\gamma_{3}(\mathfrak{B})$,
see \cite{GrigorchukZuk}. It means that the one step recursion $g=(g_{0},g_{1})$
of an element $g\in\mathfrak{B}''$ satisfies that $g_{0},g_{1}\in\gamma_{3}(\mathfrak{B})$.
It follows that $\mathfrak{B}/\mathfrak{B}''$ is virtually step $2$-nilpotent. 

We first verify that for $K$ clopen, the group $\bar{\Gamma}_{K,{\bf m}}={\rm Fix}_{\mathfrak{B}}(K)/\oplus_{x\in I_{K}}{\rm Rist}_{m_{|x|}}(\mathsf{T}_{x})'$
is a sub-quotient of a finitely generated virtually nilpotent group.
Recall the notation of $I_{K}$, $I_{K}$ is finite in this case.
Let $n=\max\{|x|+m_{|x|}:x\in I_{K}\}$ and $Q_{n}$ be the subgroup
$\prod_{u\in\mathsf{L}_{n}}([\mathfrak{B},\mathfrak{B}])_{u}$. Then
the quotient group $\mathfrak{B}/Q_{n}'$ is virtually step-$2$ nilpotent.
It is clear by the choice of $n$ that $\bar{\Gamma}_{K,{\bf m}}$
is a quotient group of ${\rm Fix}_{\mathfrak{B}}(K)/\left({\rm Fix}_{\mathfrak{B}}(K)\cap Q'_{n}\right)$,
thus a sub-quotient of a virtually step-$2$ nilpotent group $\mathfrak{B}/Q_{n}'$.

It follows then $\bar{\Gamma}_{K,{\bf m}}$ is a finitely generated
virtually nilpotent group. Thus it only has countably many subgroups
and all its IRSs are atomic. Corollary \ref{branch} then implies
the statement.

\end{proof}

We now discuss the case where ${\rm Fix}(H)$ is a general closed
subset of $\partial\mathsf{T}$. Let $\Gamma$ be a branch group acting
faithfully on $\mathsf{T}$ and $\mu$ an ergodic IRS of $\Gamma$.
The quotient group $\bar{\Gamma}_{K,{\bf m}}$ may admit a continuum
of non-atomic IRSs. For example, consider a branch group acting on
the rooted binary tree and an IRS where ${\rm Fix}(H)$ is a single
point. In this case the distribution of ${\rm Fix}(H)$ is a uniform
point on $\partial\mathsf{T}$ with respect to Hausdorff measure.
According to Theorem \ref{branch}, there is a sequence of integers
${\bf m}=(m_{i})$ such that the conditional distribution of $H$
given ${\rm Fix}(H)=\{x\}$ is pulled back from an IRS of the quotient
$\bar{\Gamma}_{x,{\bf m}}$. Note that by definitions, the group $\bar{\Gamma}_{x,{\bf m}}$
contains the infinite direct sum of finite groups
\[
\bigoplus_{i=1}^{\infty}{\rm Rist}_{\mathfrak{G}}(x_{1}\ldots x_{i-1}\check{x}_{i})/\left[{\rm Rist}_{m_{i}}(\mathsf{T}_{x_{1}\ldots x_{i-1}\check{x}_{i}}),{\rm Rist}_{m}(\mathsf{T}_{x_{1}\ldots x_{i-1}\check{x}_{i}})\right]
\]
as a normal subgroup. For the first Grigorchuk group $\mathfrak{G}$,
the parabolic subgroup ${\rm St}_{\mathfrak{G}}(1^{\infty})$ is described
explicitly in \cite[Theorem 4.4]{BG02}, it has the structure of an
iterated semi-direct product. To further classify IRSs of such quotients
one would need more algebraic information on the groups under consideration,
which is beyond the scope of this work. 

\subsection{A sufficient condition for co-soficity\protect\footnote{In the previous version of this paper, the proof of Theorem \ref{P-stable}
had a gap. In the current version the proof is fixed with the additional
bounded activity assumption, see Theorem \ref{cosofic-1}.}}

In this subsection we formulate a sufficient condition IRSs of a branch
group to be co-sofic. To state the assumptions, we introduce necessary
notations for the wreath recursion in ${\rm Aut}(\mathsf{T})$. A
permutation wreath product $A\wr_{X}G$, where $G\curvearrowright X$,
is the semi-direct product $\oplus_{X}A\rtimes G$, where $G$ acts
on $\oplus_{X}L$ by permuting coordinates. Elements in $A\wr_{X}G$
are recorded as pairs $(f,g)$ where $f:X\to A$ is a function of
finite support and $g\in G$. Multiplication in $A\wr_{X}G$ is given
by 
\[
(f_{1},g_{1})(f_{2},g_{2})=\left(f_{1}\left(g_{1}\cdot f_{2}\right),g_{1}g_{2}\right),\ \mbox{where }g\cdot f(x)=f(x\cdot g).
\]
Let $\mathsf{T}_{{\bf d}}$ be a spherically symmetric rooted tree
with valency sequence ${\bf d}=(d_{i})_{i=0}^{\infty}$. Denote by
$s$ the shift, $s{\bf d}=(d_{i+1})_{i=0}^{\infty}$. Denote by $\varphi_{n}$
the canonical wreath recursion in ${\rm Aut}(\mathsf{T}),$ 
\[
\varphi_{n}:{\rm Aut}(\mathsf{T}_{{\bf d}})\to{\rm Aut}\left(\mathsf{T}_{s^{n}{\bf d}}\right)\wr_{\mathsf{L}_{n}}{\rm Aut}(\mathsf{T}_{{\bf d}}^{n}).
\]
We identify $g$ with its image under the wreath recursion $\varphi_{n}$
and write 
\[
g=\varphi_{n}(g)=\left(\left(g_{v}\right)_{v\in\mathsf{L}_{n}},\sigma_{g}\right),
\]
where $g_{v}$ is called the \emph{section} of $g$ at vertex $v$
and $\sigma_{g}$ indicates how $g$ permutes in the finite tree $\mathsf{T}_{{\bf d}}^{n}$.
For more detailed exposition on the wreath recursion see e.g. \cite{handbook}. 

An important family of groups acting on rooted trees consists of \emph{contracting
self-similar groups}, see the monograph \cite{Nek-book}. The action
of $\Gamma$ on a regular rooted tree $\mathsf{T}$ is called \emph{self-similar
}if for any $g\in\Gamma$ and $v\in\mathsf{T},$ the section $g_{v}$
is also in $\Gamma$. A self-similar group action $\Gamma\curvearrowright\mathsf{T}$
is called \emph{contracting} if there exists a finite set $\mathcal{N}\subseteq\Gamma$
such that for every $g\in\Gamma$, there exists $n$ such that $g_{v}\in\mathcal{N}$
for all vertices $v$ of length at least $n$. The smallest set $\mathcal{N}$
satisfying this condition is called the \emph{nucleus} of the contracting
group, see \cite[Section 2.11]{Nek-book}. We need a condition that
is a generalized version of a contracting self-similar action.

\begin{definition}[Contracting in a generalized sense]\label{C}

We say the action of $\Gamma$ on a rooted spherically symmetric tree
$\mathsf{T}_{{\bf d}}$ satisfies Assumption (C) if there exists a
sequence of subsets $\mathcal{N}_{i}\subseteq{\rm Aut}(\mathsf{T}_{s^{i}{\bf d}})$
such that:
\begin{enumerate}
\item for every $g\in\Gamma$, there exists $n_{g}\in\mathbb{N}$ such that
$g_{v}\in\mathcal{N}_{|v|}$ for all vertices $v$ of length at least
$n_{g}$,
\item there exists constants $c_{0},i_{0}\in\mathbb{N}$ such that for any
$i\ge i_{0}$ and any $\gamma\in\mathcal{N}_{i}\setminus\{id\}$,
we have $\pi_{i,c_{0}}(\gamma)\neq id_{{\rm Aut}\left(\mathsf{T}_{s^{i}{\bf d}}^{c_{0}}\right)}$,
where $\pi_{i,c_{0}}$ is the projection map ${\rm Aut}(\mathsf{T}_{s^{i}{\bf d}})\to{\rm Aut}\left(\mathsf{T}_{s^{i}{\bf d}}^{c_{0}}\right)$. 
\end{enumerate}
\end{definition}

Clearly a contracting self-similar action satisfies Assumption (C).
There are many interesting examples satisfying (C) that are not contracting
self-similar actions, see Subsection \ref{subsec:C-examples}.

For an element $g\in{\rm Aut}(\mathsf{T})$, we say a vertex $v\in\mathsf{L}_{n}$
is \emph{active} if the section $g_{v}\neq id$. Denote by $A_{g}(n)$
the set of active vertices for $g$ on level $n$, that is, $A_{g}(n)=\left\{ v\in\mathsf{L}_{n}:g_{v}\neq id\right\} $.
We say $g$ is an automorphism of \emph{bounded activity} if $\sup_{n\in\mathbb{N}}\left|A_{g}(n)\right|<\infty$. 

\begin{theorem}\label{cosofic-1}

Let $\Gamma$ be a group acting faithfully and level transitively
on a spherically symmetric rooted tree $\mathsf{T}_{{\bf d}}$ of
bounded valency. Suppose:
\begin{description}
\item [{(i)}] $\Gamma$ is a just-infinite branch group (equivalently,
$\left[{\rm Rist}_{m}^{\Gamma}(\mathsf{T}_{{\bf d}}),{\rm Rist}_{m}^{\Gamma}(\mathsf{T}_{{\bf d}})\right]$
is of finite index in $\Gamma$ for all $m\in\mathbb{N}$),
\item [{(ii)}] $\Gamma\curvearrowright\mathsf{T}_{{\bf d}}$ satisfies
Assumption (C) in Definition \ref{C}, 
\item [{(iii)}] any element $g\in\Gamma$ is of bounded activity.
\end{description}
Then all IRSs of $\Gamma$ are co-sofic.

\end{theorem}

\begin{remark}

Although amenability of $\Gamma$ is not explicitly assumed in Theorem
\ref{cosofic-1}, we expect that groups satisfying Assumption (C)
and the bounded activity assumption (iii) are amenable, for the following
reasons. If assumption (ii) is strengthened to that ${\bf d}$ is
constant and $\Gamma\curvearrowright\mathsf{T}$ is a contracting
self-similar action, then the strengthened assumption (ii) together
with (iii) implies that $\Gamma$ is a subgroup of bounded automatic
automorphisms. In this case \cite[Theorem 1.2]{BKN} implies that
$\Gamma$ is amenable.

More generally, by \cite[Theorem 4.8]{JNS}, if $\Gamma<{\rm Aut}(\mathsf{T}_{{\bf d}})$
is such that any element $g\in\Gamma$ is of bounded activity, then
$\Gamma$ is amenable if and only if the \emph{isotropy groups} on
the tree boundary $\partial\mathsf{T}_{{\bf d}}$ (also called the
groups of germs) are amenable. For the definition of isotropy groups
see e.g., \cite[Section 3.1]{JNS}. 

\end{remark}

\begin{example}

Let $\mathsf{T}_{{\bf d}}$ be a rooted spherically symmetric tree
with bounded valency sequence. Denote by ${\rm Alt}_{f}\left(\mathsf{T}_{{\bf d}}\right)$
the group generated of finitary automorphisms whose vertex permutations
are all even permutations. The group ${\rm Alt}_{f}\left(\mathsf{T}_{d}\right)$
is just-infinite and branching. It is locally finite, thus amenable.
Let $\mathcal{N}_{i}={\rm Alt}\left(\left\{ 0,1,\ldots,d_{i}-1\right\} \right)$,
then ${\rm Alt}_{f}\left(\mathsf{T}_{{\bf d}}\right)$ satisfies Assumption
(C) with $c_{0}=1$, $i_{0}=1$ and the sequence $(\mathcal{N}_{i})$.
For any finitary automorphism $\gamma$, there exists a finite level
$n$ such that $A_{i}(\gamma)=\emptyset$ for all $i\ge n$. Therefore
elements of ${\rm Alt}_{f}\left(\mathsf{T}_{{\bf d}}\right)$ are
of bounded activity. Theorem \ref{cosofic-1} implies that all IRSs
of ${\rm Alt}_{f}\left(\mathsf{T}_{{\bf d}}\right)$ are co-sofic. 

\end{example}

Finitely generated examples of groups that satisfy the assumptions
of Theorem \ref{cosofic-1} are discussed in the next subsection.
By the definition of activity, if elements of a generating set of
$\Gamma$ are of bounded activity, then all elements of $\Gamma$
are of bounded activity. 

Throughout the rest of this subsection, let $\Gamma$ be a group satisfying
the assumptions of Theorem \ref{cosofic-1} and $\mu$ be an IRS of
$\Gamma$. We take the following steps to prove Theorem \ref{cosofic-1}.

\textbf{Step I }\textcolor{black}{(Choice of approximations)}\textbf{.}
Given the fixed point set ${\rm Fix}(H)$, color the tree $\mathsf{T}$
according to ${\rm Fix}(H)$ as in Definition \ref{color}. On level
$i$, denote by $\mathsf{R}_{i}(H)$ ($\mathsf{G}_{i}(H)$, $\mathsf{B}_{i}(H)$
resp.) the set of vertices in $\mathsf{L}_{i}$ that are colored red
(green, blue resp.). Note the following equivariance property which
can be verified directly by the definition of coloring:

\begin{fact}

Let $\gamma\in\Gamma$, then $\mathsf{R}_{i}(\gamma^{-1}H\gamma)=\mathsf{R}_{i}(H)\cdot\gamma$,
$\mathsf{G}_{i}(\gamma^{-1}H\gamma)=\mathsf{G}_{i}(H)\cdot\gamma$,
$\mathsf{B}_{i}(\gamma^{-1}H\gamma)=\mathsf{B}_{i}(H)\cdot\gamma$. 

\end{fact}

Given a subset $\mathsf{B}\subseteq\mathsf{L}_{i}$, consider the
subgroup $A(\mathsf{B})$ of $\Gamma$ defined as 
\[
A\left(\mathsf{B}\right):=\{g\in\Gamma:\ v.g=v\mbox{ for }v\in\mathsf{B}\mbox{ and }g_{v}=id\mbox{ for all }v\in\mathsf{B}\}.
\]
In words, $A\left(\mathsf{B}\right)$ consists of elements of $\Gamma$
which fix the subset $\mathsf{B}$ pointwise, and moreover, the sections
at vertices in $\mathsf{B}$ are all trivial. 

Now we choose a sequence of subgroups depending on $H$. Recall the
constant $c_{0}$ as in the second item of Assumption (C) and that
${\rm Rist}_{c_{0}}^{\Gamma}(\mathsf{T}_{v})$ is the level $c_{0}$
rigid stabilizer of the subtree rooted at $v$. Define $K_{i}(H)$
to be
\begin{equation}
K_{i}(H):=\left(H\cap A\left(\mathsf{B}_{i}(H)\right)\right)\prod_{v\in\mathsf{G}_{i}(H)\cup\mathsf{B}_{i}(H)}{\rm Rist}_{c_{0}}^{\Gamma}(\mathsf{T}_{v}).\label{eq:K_i}
\end{equation}
It is routine to check that the map $H\mapsto K_{i}(H)$ is measurable. 

\textbf{Step II }\textcolor{black}{(IRSs supported on finite index
subgroups)}\textbf{.} In what follows the group $\Gamma$ is fixed
and we drop reference to $\Gamma$ in rigid stabilizers. Apply Corollary
\ref{branch} to $\mu$, we have that for $\mu$-a.e. subgroup $H$,
there exist a sequence of positive integers ${\bf m}=(m_{i})$, $m_{i}\in\mathbb{N}$,
which only depends on the ergodic component that $H$ is in, such
that 
\begin{equation}
H\ge\left[R_{{\rm Fix}(H),{\bf m}},R_{{\rm Fix}(H),{\bf m}}\right],\label{eq:containK}
\end{equation}
where ${\rm Fix}(H)$ is the set of fixed points of $H$ in $\partial\mathsf{T}$,
and the subgroup $R_{K,{\bf m}}$ is defined in Notation \ref{branch-notation}.
The containment (\ref{eq:containK}) guarantees that $K_{i}(H)$ is
a finite index subgroup of $\Gamma$:

\begin{lemma}\label{transversal}

Denote by $\mu_{i}$ the pushforward of $\mu$ under the map $H\mapsto K_{i}(H)$.
Then $\mu_{i}$ is an IRS supported on finite index subgroups of $\Gamma$. 

\end{lemma}

\begin{proof}[Proof of Lemma \ref{transversal}]

We first show that $\mu$-a.e. $K_{i}(H)$ is of finite index in $\Gamma$.
Since by its definition, $K_{i}(H)$ contains the level rigid stabilizers
of the subtrees rooted at green and blue vertices, it suffices to
show that if $v$ is a red vertex, then ${\rm Rist}_{n_{i}}\left(\mathsf{T}_{v}\right)'<H\cap A_{i}$,
for some finite $n_{i}$. Suppose $H$ satisfies (\ref{eq:containK}).
Let $u$ be the shortest prefix of $v$ that is colored red. Then
${\rm Rist}_{m_{|u|}}(\mathsf{T}_{u})<R_{{\rm Fix}_{\partial\mathsf{T}}(H),{\bf m}}$,
therefore (\ref{eq:containK}) implies that 
\[
H>{\rm Rist}_{m_{|u|}}(\mathsf{T}_{u})'.
\]
Elements of ${\rm Rist}_{m_{|u|}}(\mathsf{T}_{u})$ only permutes
red vertices, thus ${\rm Rist}_{m_{|u|}}(\mathsf{T}_{u})<A_{i}$.
It follows that 
\[
{\rm Rist}_{n_{i}}\left(\mathsf{T}_{v}\right)'<{\rm Rist}_{m_{|u|}}(\mathsf{T}_{u})'<H\cap A_{i},
\]
where $n_{i}=\max\left\{ 0,m_{|u|}-|v|+|u|\right\} $. Therefore Corollary
\ref{branch} implies that for $\mu$-a.e. $H$, $K_{i}(H)$ is of
finite index in $\Gamma$. 

Next we verify that $\mu_{i}$ is invariant under conjugation by $\Gamma$.
Let $f:{\rm Sub}(\Gamma)\to\mathbb{R}$ be a bounded measurable function
and $\gamma\in\Gamma$. Then by the definition of the subgroup $A\left(\mathsf{B}\right)$,
we have 
\[
A\left(\mathsf{B}_{i}\left(\gamma^{-1}H\gamma\right)\right)=\gamma^{-1}A\left(\mathsf{B}_{i}\left(H\right)\right)\gamma.
\]
It follows that 
\begin{align*}
\gamma^{-1}K_{i}(H)\gamma & =\left(\gamma^{-1}H\gamma\cap\gamma^{-1}A\left(\mathsf{B}_{i}(H)\right)\gamma\right)\prod_{v\in\mathsf{G}_{i}(H)\cup\mathsf{B}_{i}(H)}\gamma^{-1}{\rm Rist}_{c_{0}}(\mathsf{T}_{v})\gamma.\\
 & =\left(\gamma^{-1}H\gamma\cap A\left(\mathsf{B}_{i}(\gamma^{-1}H\gamma)\right)\right)\prod_{v\in\mathsf{G}_{i}(H)\cup\mathsf{B}_{i}(H)}{\rm Rist}_{c_{0}}(\mathsf{T}_{v\cdot\gamma})\\
 & =\left(\gamma^{-1}H\gamma\cap A\left(\mathsf{B}_{i}(\gamma^{-1}H\gamma)\right)\right)\prod_{u\in\mathsf{G}_{i}(\gamma^{-1}H\gamma)\cup\mathsf{B}_{i}(\gamma^{-1}H\gamma)}{\rm Rist}_{c_{0}}(\mathsf{T}_{u})\\
 & =K_{i}(\gamma^{-1}H\gamma).
\end{align*}
We conclude that 
\begin{align*}
\mathbb{E}_{\mu_{i}}\left[f\left(\gamma^{-1}K\gamma\right)\right] & =\mathbb{E}_{\mu}\left[f\left(\gamma^{-1}K_{i}(H)\gamma\right)\right]=\mathbb{E}_{\mu}\left[f\left(K_{i}(\gamma^{-1}H\gamma)\right)\right]\\
 & =\mathbb{E}_{\mu}\left[f\left(K_{i}(H)\right)\right]=\mathbb{E}_{\mu_{i}}\left[f\left(K\right)\right].
\end{align*}

\end{proof}

\textbf{\textcolor{black}{Step III}} (Reduction to counting bad blue
vertices). In this step we reduce the problem of showing convergence
of $(\mu_{i})$ to $\mu$ in the weak$^{\ast}$-topology to counting
what we call bad blue vertices. The map $H\mapsto K_{i}(H)$ provides
a natural coupling between $\mu$ and $\mu_{i}$. By the definition
of the weak$^{\ast}$-topology, showing convergence $\mu_{i}\to\mu$
is equivalent to showing that for any pair of finite sets $E\subseteq F$
in $\Gamma$, 
\[
\left|\mu\left(H\cap F=E\right)-\mu_{i}\left(K\cap F=E\right)\right|\to0\mbox{ when }i\to\infty.
\]
Using the coupling $(H,K_{i}(H))$, the difference is bounded by 
\begin{align*}
\left|\mu\left(H\cap F=E\right)-\mu_{i}\left(K\cap F=E\right)\right| & \le\mu\left(H:\left(H\bigtriangleup K_{i}(H)\right)\cap F\neq\emptyset\right)\\
 & \le\sum_{g\in F}\mu\left(g\in H\bigtriangleup K_{i}(H)\right),
\end{align*}
where in the second line we take the union bound over elements of
$F$. Therefore the task is reduced to estimate $\mu\left(g\in H\bigtriangleup K_{i}(H)\right)$
for a group element $g$. 

\begin{definition}[Bad blue vertices]

Let $g\in\Gamma$ be a given element. We say that a vertex $v\in\mathsf{B}_{i}(H)$
is a \emph{bad blue} vertex for $g$, if there exists $\gamma\in\Gamma$
such that the conjugate $g^{\gamma}=\gamma^{-1}g\gamma$ is in $K_{i}(H)\bigtriangleup H$
and moreover the section $(g^{\gamma})_{v}$ is not $id$. Denote
by $\mathsf{BB}_{i}(H,g)$ the set of bad blue vertices for $g$ on
level $i$ with respect to $H$. 

\end{definition}

Note that if $g^{\gamma}\in K_{i}(H)\bigtriangleup H$, then a vertex
$v\in\mathsf{B}_{i}(H)\cup\mathsf{G}_{i}(H)$ is necessarily fixed
by $g^{\gamma}$. The reason for marking the bad blue vertices is
stated in the following lemma. Recall that by Assumption (C) item
1, there exists an integer $n_{g}$ such that $g_{v}\in\mathcal{N}_{|v|}$
for all vertices $v$ of length at least $n_{g}$.

\begin{lemma}\label{BB}

Let $g\in\Gamma$ and $i\ge\max\{n_{g},i_{0}\}$. If $\gamma\in\Gamma$
is such that $g^{\gamma}\in K_{i}\bigtriangleup H$, then $g^{\gamma}\in H$
and there exists a vertex $v\in\mathsf{BB}_{i}(H,g)$ such that $v\cdot\gamma^{-1}\in A_{g}(i)$. 

\end{lemma}

\begin{proof}[Proof of Lemma \ref{BB}]

Throughout the proof $H$ is fixed and we omit reference to $H$ in
the notations. Let $\gamma\in\Gamma$ be an element such that $g^{\gamma}\in K_{i}\bigtriangleup H$.
Under the canonical wreath recursion $\varphi_{i}:{\rm Aut}(\mathsf{T})\to{\rm Aut}\left(\mathsf{T}\right)\wr_{\mathsf{L}_{i}}{\rm Aut}(\mathsf{T}_{i})$,
we write $\gamma=\left(\left(\gamma_{v}\right)_{v\in\mathsf{L}_{i}},\sigma_{f}\right)$
and $g=\left(\left(g_{v}\right)_{v\in\mathsf{L}_{i}},\sigma_{g}\right)$.
Note that by the definitions, any element in $K_{i}\cup H$ fixes
red and blue vertices in $\mathsf{L}_{i}$. It follows that if $g^{\gamma}\in K_{i}\bigtriangleup H$,
then its projection to ${\rm Aut}(\mathsf{T}_{i})$, $\sigma_{\gamma}^{-1}\sigma_{g}\sigma_{\gamma}$,
fixes the set $\mathsf{G}_{i}\cup\mathsf{R}_{i}$ pointwise; and moreover,
at a vertex $v\in\mathsf{G}_{i}\cup\mathsf{R}_{i}$, the section of
$g^{\gamma}$ is
\[
\left(g^{\gamma}\right)_{v}=\gamma_{u}^{-1}g_{u}\gamma_{u},\mbox{ where }u=v\cdot\gamma^{-1}.
\]
We will use a few times the fact that ${\rm Rist}_{c_{0}}(\mathsf{T}_{v})=\gamma^{-1}{\rm Rist}_{c_{0}}(\mathsf{T}_{u})\gamma$,
where $u=v\cdot\gamma^{-1}$. 

We proceed by discussing two cases in the symmetric difference $K_{i}\bigtriangleup H$.
Write $A_{i}=A(\mathsf{B}_{i}(H))$.

Case I: $g^{\gamma}\in K_{i}$ and $g^{\gamma}\notin H$. We need
to show that this case is impossible. By the definition of $K_{i}$,
we have that in this case 
\[
g^{\gamma}=h_{i}r_{i},
\]
where $h_{i}\in H\cap A_{i}$ and $r_{i}$ can be written as a (commuting)
product of $r_{i}^{v}$ over vertices green and blue, where $r_{i}^{u}\in{\rm Rist}_{c_{0}}(\mathsf{T}_{u})$.
Recall that any element in $H$ has trivial sections at green vertices,
and any element in $A_{i}$ has trivial sections at blue vertices.
Thus for $h_{i}\in H\cap A_{i}$, $(h_{i})_{u}=id$ for $u\in\mathsf{B}_{i}\cup\mathsf{G}_{i}$.
It follows that at blue and green vertices, $\left(g^{\gamma}\right)_{u}=r_{i}^{u}\in{\rm Rist}_{c_{0}}(\mathsf{T}_{u})$.
It follows that $g_{v}\in{\rm Rist}_{c_{0}}(\mathsf{T}_{v})$ for
$v=u\cdot\gamma$. Recall that $i\ge n_{g}$ implies that $g_{v}\in\mathcal{N}_{i}$
for all $v\in\mathsf{L}_{i}$. By the choice of $c_{0}$ in Assumption
(C) item 2, we have $\mathcal{N}_{i}\cap{\rm St}_{c_{0}}^{{\rm Aut}(\mathsf{T}_{v})}(\mathsf{T}_{v})=\{id\}$.
It follows that $g_{v}\in\mathcal{N}_{i}\cap{\rm Rist}_{c_{0}}(\mathsf{T}_{w})$
implies $g_{v}=id.$ Then after conjugating by $\gamma$, we have
that $g_{v}=id$ implies $\left(g^{\gamma}\right)_{u}=id$ for $u=v\cdot\gamma^{-1}$.
We conclude that $(g^{\gamma})_{u}=r_{i}^{u}=id$ for all $u\in\mathsf{B}_{i}\cup\mathsf{G}_{i}$,
in particular $r_{i}=id$ and $g^{\gamma}\in H\cap A_{i}$, which
contradicts the assumption that $g^{\gamma}\notin H$. 

Case II: $g^{\gamma}\in H$ and $g^{\gamma}\notin K_{i}$. In this
case, since $H$ fixes the cylinder $C_{w}$ for a green vertex $w$,
we have $(g^{\gamma})_{w}=id$ for $w\in\mathsf{G}_{i}$. If all sections
$(g^{\gamma})_{v}$ at blue vertices are $id$, then by definition
of $A_{i}$ we would have $g^{\gamma}\in A_{i}\cap H$, contradicting
with the assumption that $g^{\gamma}\notin K_{i}$. Therefore there
is a blue vertex $v\in\mathsf{B}_{i}$ with $(g^{\gamma})_{v}\neq id$.
Such a vertex $v$ is a bad blue vertex for $g$. Let $u=v\cdot\gamma^{-1}$,
then $g$ fixes $u$ and the section $g_{u}=\gamma_{u}\left(g^{\gamma}\right)_{v}\gamma_{u}^{-1}$
is nontrivial, that is $u\in A_{g}(i)$. 

\end{proof}

Next we show that the proportion of bad blue vertices controls the
probability $\mathbb{P}_{\mu}\left(g\in H\bigtriangleup K_{i}(H)\right)$. 

\begin{lemma}\label{BB2}

We have 
\[
\mathbb{P}_{\mu}\left(g\in H\bigtriangleup K_{i}(H)\right)\le\left|A_{g}(i)\right|\mathbb{E}_{\mu}\left(\frac{\left|\mathsf{BB}_{i}(H,g)\right|}{\left|\mathsf{L}_{i}\right|}\right).
\]

\end{lemma}

\begin{proof}

Recall that $\Gamma_{i}=\pi_{i}(\Gamma)$, where $\pi_{i}$ is the
projection ${\rm Aut}(\mathsf{T})\to{\rm Aut}(\mathsf{T}^{i})$. Fix
a choice section $s:\Gamma_{i}\to\Gamma$, $\pi_{i}\circ s(\gamma)=\gamma$.
We write $\tilde{\gamma}=s(\gamma)$, $\gamma\in\Gamma_{i}$. By the
conjugation invariance of $\mu$, we have that 
\[
\mathbb{P}_{\mu}\left(g\in H\bigtriangleup K_{i}(H)\right)=\frac{1}{\left|\Gamma_{n}\right|}\sum_{\gamma\in\Gamma_{n}}\mathbb{P}_{\mu}\left(g^{\tilde{\gamma}}\in H\bigtriangleup K_{i}(H)\right).
\]
By Lemma \ref{BB}, we have the containment of events:
\[
\left\{ H:g^{\tilde{\gamma}}\in H\bigtriangleup K_{i}(H)\right\} \subseteq\left\{ H:\mathsf{BB}_{i}(H,g)\cap A_{g}(i)\cdot\tilde{\gamma}\neq\emptyset\right\} .
\]
Note that $A_{g}(i)$ is a subset in level $i$, therefore $A_{g}(i)\cdot\tilde{\gamma}=A_{g}(i)\cdot\gamma$,
that is, it does not depend on the choice of the section. We have
then 
\[
\mathbb{P}_{\mu}\left(g\in H\bigtriangleup K_{i}(H)\right)\le\frac{1}{\left|\Gamma_{n}\right|}\sum_{\gamma\in\Gamma_{n}}\mathbb{P}_{\mu}\left(\mathsf{BB}_{i}(H,g)\cap A_{g}(i)\cdot\gamma\neq\emptyset\right).
\]
Now we take the union bound over active sites:
\begin{align*}
\mathbb{P}_{\mu}\left(\mathsf{BB}_{i}(H,g)\cap A_{g}(i)\cdot\gamma\neq\emptyset\right) & =\mathbb{P}_{\mu}\left(\bigcup_{x\in A_{g}(i)}\left\{ \mathsf{BB}_{i}(H,g)\ni x\cdot\gamma\right\} \right)\\
 & \le\sum_{x\in A_{g}(i)}\mathbb{P}_{\mu}\left(\mathsf{BB}_{i}(H,g)\ni x\cdot\gamma\right).
\end{align*}
Since $\Gamma_{i}$ acts transitively on $\mathsf{L}_{i}$, 
\[
\frac{1}{\left|\Gamma_{n}\right|}\sum_{\gamma\in\Gamma_{n}}\mathbf{1}_{\left\{ \mathsf{BB}_{i}(H,g)\ni x\cdot\gamma\right\} }=\frac{\left|\mathsf{BB}_{i}(H,g)\right|}{\left|\mathsf{L}_{i}\right|}.
\]
Combining these calculations, we obtain 
\begin{align*}
\mathbb{P}_{\mu}\left(g\in H\bigtriangleup K_{i}(H)\right) & \le\frac{1}{\left|\Gamma_{n}\right|}\sum_{\gamma\in\Gamma_{n}}\sum_{x\in A_{g}(i)}\mathbb{P}_{\mu}\left(\mathsf{BB}_{i}(H,g)\ni x\cdot\gamma\right)\\
 & =\left|A_{g}(i)\right|\mathbb{E}_{\mu}\left[\frac{\left|\mathsf{BB}_{i}(H,g)\right|}{\left|\mathsf{L}_{i}\right|}\right].
\end{align*}

\end{proof}

Lemma \ref{BB2} reduces the estimate to counting on each level $i$,
the proportion of bad blue vertices for $g$ with respect to $H$. 

\textbf{\textcolor{black}{Step IV}} (Vanishing proportion of bad blue
vertices). Now we show that Assumption (C) implies that the proportion
of bad blue vertices is vanishing.

\begin{lemma}\label{proportion}

Let $H<\Gamma$ be a subgroup and $g\in\Gamma$. Then the ratios of
bad blue points satisfy
\[
\frac{\left|\mathsf{BB}_{i}(H,g)\right|}{\left|\mathsf{L}_{i}\right|}\to0\mbox{ when }i\to\infty.
\]

\end{lemma}

\begin{proof}[Proof of Lemma \ref{proportion}]

Recall that $c_{0}$ is the integer provided by Assumption (C). Suppose
$v\in\mathsf{BB}_{i}(H,g)$ is a bad blue vertex for $g$. We claim
that further down $c_{0}$ levels, the number of red vertices among
the descendants of $v$ on level $\mathsf{L}_{i+c_{0}}$ is at least
$2$. Indeed, by definition of a bad blue vertex, there exists $\gamma\in\Gamma$
such that $(g^{\gamma})_{v}\in K_{i}(H)\bigtriangleup H$ and $(g^{\gamma})_{v}\neq id$.
Let $u=v\cdot\gamma^{-1}$, then $\left(g^{\gamma}\right)_{v}=\gamma_{u}^{-1}g_{u}\gamma_{u}$
as explained in the proof of the previous lemma. By Lemma \ref{BB},
for $i\ge\max\{i_{0},n_{g}\},$we have that $g^{\gamma}\in H$ and
$g_{u}\in\mathcal{N}_{i}\setminus\{id\}$. By the second item in Assumption
(C), we have that $g_{u}$ cannot fix all descendants of $u$ on the
level $i+c_{0}$ (otherwise $g_{u}$ would project to $id$). Thus
there exist at least two vertices of the form $uz$, where $|z|=c_{0}$,
such that $uz\cdot g_{u}\neq uz$. After conjugation by $\gamma$,
we have that there exists at least two vertices of the form $vz_{1}$
and $vz_{2}$, $|z_{i}|=c_{0}$, such that $vz_{i}\cdot g^{\gamma}\neq vz_{i}$,
$i=1,2$. Since $g^{\gamma}\in H$, it follows that $H$ has no fixed
point in the cylinders $C_{vz_{i}}$, if $vz\cdot g^{f}\neq vz$.
In other words, the two vertices $vz_{1}$ and $vz_{2}$ are colored
red. 

Recall that descendants of red (green resp.) vertices are red (green
resp.). Denote by $q_{b}(i)=\left|\mathsf{B}_{i}(H)\right|/\left|\mathsf{L}_{i}\right|$
the proportion of blue vertices on level $i$, similarly $q_{bb}(i)=\left|\mathsf{BB}_{i}(H,g)\right|/\left|\mathsf{L}_{i}\right|$
the proportion of bad blue vertices. Then we have that the sequence
$(q_{b}(i))_{i=1}^{\infty}$ is non-increasing. By the reasoning in
the previous paragraph, because of red descendants of a bad blue vertex
down $c_{0}$ levels, we have
\[
q_{b}(i+c_{0})\le q_{b}(i)-\frac{2}{d^{c_{0}}}q_{bb}(i),
\]
where $d$ is the maximal valency of the tree $\mathsf{T}$. It follows
that $q_{bb}(i)\to0$ when $i\to\infty$. 

\end{proof}

After these four steps we are now ready to summarize and finish the
proof of Theorem \ref{cosofic-1}]. 

\begin{proof}[Proof of Theorem \ref{cosofic-1}]

Let $\Gamma$ be a group satisfying the assumptions of Theorem \ref{cosofic-1}
and $\mu$ be an IRS of $\Gamma$. 

Let $\mu_{i}$ be the pushforward of $\mu$ under the map $H\mapsto K_{i}(H)$,
where $K_{i}(H)$ is defined in (\ref{eq:K_i}). By Lemma \ref{transversal},
$\mu_{i}$ is an IRS supported on finite index subgroups of $\Gamma$.
As explained at the beginning of Step III, to show that $\mu_{i}$
converges to $\mu$ in the weak$^{\ast}$-topology, it suffices to
show that for any $g\in\Gamma$, $\lim_{i\to\infty}\mathbb{P}_{\mu}\left(g\in H\bigtriangleup K_{i}(H)\right)=0$.
By Lemma \ref{BB2}, we have 
\[
\limsup_{i\to\infty}\mathbb{P}_{\mu}\left(g\in H\bigtriangleup K_{i}(H)\right)\le\sup_{i\in\mathbb{N}}\left|A_{g}(i)\right|\limsup_{i\to\infty}\mathbb{E}_{\mu}\left[\frac{\left|\mathsf{BB}_{i}(H,g)\right|}{\left|\mathsf{L}_{i}\right|}\right].
\]
The bounded activity assumption (iii) exactly means $\sup_{i\in\mathbb{N}}\left|A_{g}(i)\right|<\infty$
for all $g$. By Lemma \ref{proportion}, we have 
\[
\frac{\left|\mathsf{BB}_{i}(H,g)\right|}{\left|\mathsf{L}_{i}\right|}\to0\mbox{ when }i\to\infty\mbox{ for all }H.
\]
Therefore by the bounded convergence theorem, we have 
\[
\lim_{i\to\infty}\mathbb{E}_{\mu}\left[\frac{\left|\mathsf{BB}_{i}(H,g)\right|}{\left|\mathsf{L}_{i}\right|}\right]=0.
\]
It follows that $\mu_{i}$ converges to $\mu$ in the weak{*} topology
when $i\to\infty$. 

\end{proof}

\subsection{Examples to which Theorem \ref{cosofic-1} is applicable\label{subsec:C-examples}}

A large class of groups acting on a regular rooted tree $\mathsf{T}$
satisfying assumptions (ii) and (iii) of Theorem \ref{cosofic-1})
consists of finitely generated subgroups of $\mathfrak{BA}(\mathsf{T})$,
where $\mathfrak{BA}(\mathsf{T})$ is the group of all \emph{bounded
automatic automorphisms} of $\mathsf{T}.$ For the definition of $\mathfrak{BA}(\mathsf{T})$
see \cite[Section 1.C]{BKN}. The main theorem of \cite{BKN} states
that the group $\mathfrak{BA}(\mathsf{T})$ is amenable. An equivalent
description of a finitely generated subgroup $\Gamma$ of $\mathfrak{BA}(\mathsf{T})$
is that $\Gamma$ is finitely generated and acts faithfully on a regular
rooted tree $\mathsf{T}$ such that the action is contracting, self-similar
and elements of $\Gamma$ have bounded activity. 

Among finitely generated subgroups of $\mathfrak{BA}(\mathsf{T})$
there are examples of just-infinite branch groups, see \cite[Section 1.D]{BKN}
or \cite[Section 1.8]{Nek-book}. We mention a few well-known examples:
the first Grigorchuk group $\mathfrak{G}$ (constructed in \cite{Grigorchuk1980});
the Gupta-Sidki $p$-groups (constructed in \cite{Gupta-Sidki});
and the group constructed by P. Neumann in \cite[Section 5]{Neumann}
which is just infinite and regularly branching over itself (the group
is called Example C in \cite{Neumann}, see also \cite[Section 1.6.6]{handbook}). 

It follows from the previous paragraph that Theorem \ref{cosofic-1}
covers the Grigorchuk group and the Gupta-Sidki $p$-group:

\begin{proof}[Proof of Theorem \ref{P-stable}]

The Grigorchuk group $\mathfrak{G}$ is a just-infinite, branch, contracting
self-similar group, see e.g., \cite[Section 1.6]{Nek-book}. Its nucleus
is $\mathcal{N}=\{id,a,b,c,d\}.$ The generators have bounded activity,
indeed, $|A_{i}(a)|=0$ and $\sup_{i}|A_{i}(b)|=\sup_{i}|A_{i}(c)|=\sup_{i}|A_{i}(d)|=2$. 

The Gupta-Sidki $p$-group $G_{p}$ acts on the $p$-regular rooted
tree, where $p$ is a prime at least $3$, it is generated by two
elements $a$ and $t$ with recursion
\[
a=\sigma\mbox{ and }t=\left(a,a^{-1},1,\ldots,1,t\right),
\]
where $a$ is the cyclic permutation of $\{0,1,\ldots,p-1\}$. The
group $G_{p}$ is a just-infinite branch group, see e.g., \cite{Garrido},
and it is a contracting self-similar group with nucleus $\mathcal{N}=\left\{ id,a,\ldots,a^{p-1},t,\ldots,t^{p-1}\right\} $.
The generators have bounded activity, indeed, $|A_{i}(a)|=0$ and
$|A_{i}(t)|=3$ for all $i\in\mathbb{N}$. 

Therefore $\mathfrak{G}$ and $G_{p}$ satisfy the three assumptions
of Theorem \ref{cosofic-1}, the statement follows.

\end{proof}

\begin{remark}

A group $\Gamma$ is \emph{LERF} (also called \emph{subgroup separable})
if each of its finitely generated subgroups is an intersection of
subgroups of finite index; or equivalently, subgroups of finite index
are dense in ${\rm Sub}(\Gamma)$. In the proof of Theorem \ref{P-stable}
we do not invoke the result of Grigorchuk and Wilson \cite{Grigorchuk-Wilson}
that the group $\mathfrak{G}$ is LERF. The Gupta-Sidki group $G_{p}$
is known to be LERF when $p=3$ by Garrido \cite{Garrido}; and to
our knowledge open for other $p$. For $\mathfrak{G}$ and $G_{3}$,
it is not clear how to use the LERF property towards showing co-soficity
of IRSs, as asked in \cite[Question 8.4]{BLT19}. 

\end{remark}

Theorem \ref{cosofic-1} applies to examples beyond just-infinite
branch groups generated by finitely many elements of $\mathfrak{BA}(\mathsf{T})$.
We consider the following sub-collection of Grigorchuk groups. The
family of Grigorchuk groups, indexed by $\omega=\omega_{0}\omega_{1}\ldots\in\{0,1,2\}^{\infty}$,
is constructed in \cite{Girgorchuk1984}. We briefly recall the definition.
Let $\{0,1,2\}$ be the three non-trivial homomorphisms from $\mathbb{Z}/2\mathbb{Z}\times\mathbb{Z}/2\mathbb{Z}=\{id,b,c,d\}$
to $\mathbb{Z}/2\mathbb{Z}=\{id,a\}$. The group $G_{\omega}$ is
generated by $a,b_{\omega},c_{\omega},d_{\omega}$, which are defined
recursively as 
\[
a=\varepsilon,\ b_{\omega}=\left(\omega_{0}(b),b_{\mathfrak{s}\omega}\right),\ c_{\omega}=\left(\omega_{0}(c),c_{\mathfrak{s}\omega}\right),\ d_{\omega}=\left(\omega_{0}(d),d_{\mathfrak{s}\omega}\right),
\]
where $\varepsilon$ is the root permutation which permutes the two
subtrees of the root and $\mathfrak{s}$ is the shift. Let $\Omega'$
be the collection of strings in $\{0,1,2\}^{\infty}$ such that $\omega\in\Omega'$
if and only if there exists $k\in\mathbb{N}$ such that in every $k$-block
$\omega_{k(i-1)}\ldots\omega_{ki-1}$, $i\in\mathbb{N}$, all three
letters $0,1,2$ appear. The set $\Omega'$ contains all periodic
sequences where all three letters $0,1,2$ appear; it also contains,
for example, strings that are concatenations of $0122$ and $0012$. 

\begin{corollary}\label{omega}

For $\omega\in\Omega'$, all IRSs of the Grigorchuk group $G_{\omega}$
are co-sofic. 

\end{corollary}

\begin{proof}

We verify that for $\omega\in\Omega'$, $G_{\omega}$ satisfies the
assumptions of Theorem \ref{cosofic-1}. 

(i) If $\omega$ is a string such that all three letters $0,1,2$
appear infinitely often, then it is shown in \cite{Girgorchuk1984}
that $G_{\omega}$ is an infinite torsion group. The group $G_{\omega}$
is branching, see \cite[Proposition 2.3]{BE2}. Since in a branch
group rigid stabilizers are finitely generated, and a finitely generated
torsion abelian group is finite, it follows that $G_{\omega}$ is
just-infinite.

(ii) Let $\mathcal{N}_{i}=\left\{ id,a,b_{\mathfrak{s}^{i}\omega},c_{\mathfrak{s}^{i}\omega},d_{\mathfrak{s}^{i}\omega}\right\} $.
Then item 1 of Assumption (C) is satisfied. Since $\omega$ is assumed
to be in $\Omega'$, let $k$ be an integer that all three letters
appear in all $k$-blocks. Then item 2 of Assumption (C) is satisfied
with $c_{0}=2k$. 

(iii) By the definition, the generators of $G_{\omega}$ are of bounded
activity for any string $\omega$. 

\end{proof}

In Brieussel \cite{Bri09}, uncountably many pairwise non-isomorphic
amenable groups of non-uniform exponential growth are given. The first
examples of groups of non-uniform exponential growth are constructed
by Wilson in \cite{Wilson}. We refer to \cite{Wilson,Bri09} for
background on non-uniform exponential growth. The groups constructed
in \cite{Wilson} are non-amenable. Similar to Grigorchuk groups considered
in Corollary \ref{omega}, a subcollection of amenable examples in
\cite{Bri09} fit in the setting of Theorem \ref{cosofic-1} and we
deduce:

\begin{corollary}[Corollary to the main theorem of \cite{Bri09} and Theorem \ref{cosofic-1}]\label{non-uni}

There exists an uncountable collection of pairwise non-isomorphic
amenable groups of non-uniform exponential growth such that for any
group $\Gamma$ in the collection, all IRSs of $\Gamma$ are co-sofic. 

\end{corollary}

\begin{proof}

We do not recall the details of the construction, but cite a few key
properties of these groups given in \cite{Bri09} that are sufficient
to verify the conditions in Theorem \ref{cosofic-1}. 

Let ${\bf d}=(d_{i})_{i=0}^{\infty}$ be a bounded valency sequence
satisfying $29\le d_{i}\le D$. Associated to the valency sequence
${\bf d}$, a group $G\left(\mathcal{A}_{d_{0}},\mathcal{A}_{{\bf d}}\right)$
acting on the rooted tree $\mathsf{T}_{d}$ is defined as in \cite[Subsection 7.2]{Bri09}.
By \cite[Proposition 7.2]{Bri09}, $G_{0}=G\left(\mathcal{A}_{d_{0}},A_{{\bf d}}\right)$
is perfect and ${\rm Rist}_{G_{0}}(n)'={\rm St}_{G_{0}}(n)$, therefore
it is a just-infinite branch group. 

By its construction, $G\left(\mathcal{A}_{d_{0}},A_{{\bf d}}\right)$
is generated by rooted automorphisms $\mathcal{A}_{d_{0}}$ and directed
automorphisms $A_{{\bf d}}$. By the definitions of these automorphisms,
we have that for a rooted automorphism $\sigma\in\mathcal{A}_{d_{0}}$,
its sections are trivial for all $v$ with $|v|\ge1$, and for a directed
automorphism $h\in A_{{\bf d}}$, its activity is bounded uniformly
by $D$. Thus $G\left(\mathcal{A}_{d_{0}},A_{{\bf d}}\right)$ is
of bounded activity. 

It remains to ensure that Assumption (C) is satisfied. Denote by $\mathcal{N}_{i}=\mathcal{A}_{d_{i-1}}\cup A_{s^{i}{\bf d}}$,
then for any element $g\in G\left(\mathcal{A}_{d_{0}},A_{{\bf d}}\right)$,
its sections in level $i$ are in $\mathcal{N}_{i}$ for sufficiently
large $i$. We impose a restriction on ${\bf d}$ to guarantee the
second item in Assumption (C). Given ${\bf d}$, denote by $E_{k,j}=\{i\ge1:d_{i}=k,\ d_{i+1}=j\}$,
where $29\le k,j\le D$. Let $B_{{\bf d}}=\{(k,j):E_{k,j}\neq\emptyset\}$
be the pairs which appear as consecutive entries in the valency sequence
${\bf d}$. Then by definition of $A_{{\bf d}}$ in \cite[Subsection 7.2]{Bri09},
we have that if there exists a constant $Q_{0}$ such that for any
pair $(k,j)\in B_{{\bf d}}$, $E_{k,j}\cap[Q_{0}m,Q_{0}(m+1)]\neq\emptyset$
for all $m\in\mathbb{N}$, then condition 2 in Assumption (C) is satisfied
with constant $c_{0}=2Q_{0}$. 

Denote by $\mathcal{D}$ the set of bounded valency sequences ${\bf d}$
with $d_{i}\ge29$, satisfying that there exists a constant $Q_{0}$
such that for any pair $(k,j)\in B_{{\bf d}}$, $E_{k,j}\cap[Q_{0}m,Q_{0}(m+1)]\neq\emptyset$
for all $m\in\mathbb{N}$. For $d\in\mathcal{D}$, we have verified
that its associated group $G\left(\mathcal{A}_{d_{0}},A_{{\bf d}}\right)$
satisfies the three conditions of Theorem \ref{cosofic-1}. By \cite[Corollary 8.2]{Bri09},
for two difference valency sequences, the associated groups are non-isomorphic. 

The set $\mathcal{D}$ is uncountable. For example, take valency sequences
that are concatenations of $abbaa$ and $aabba$, where $a=29$, $b=30$.
Such sequences satisfies the constraint to be in $\mathcal{D}$ with
$Q_{0}=5$. The group $G\left(\mathcal{A}_{d_{0}},A_{{\bf d}}\right)$
is amenable and of non-uniform exponential growth by the main theorem
of \cite{Bri09}. The statement is given by taking the groups $G\left(\mathcal{A}_{d_{0}},A_{{\bf d}}\right)$
associated with ${\bf d}\in\mathcal{D}$. 

\end{proof}

\subsection{Connection to P-stability following \cite{BLT19}}

A finitely generated group $\Gamma$ is permutation stable (P-stable)
if every almost homomorphism $\rho_{k}:\Gamma\to{\rm Sym}(n_{k})$
is close to an actual homomorphism $\varphi_{k}:\Gamma\to{\rm Sym}(n_{k})$.
More precisely, equip the symmetric group ${\rm Sym}(n)$ with the
normalized Hamming distance: 
\[
d_{n}(\sigma,\tau):=\frac{1}{n}\left|\left\{ i\in[n]:\ \sigma(i)\neq\tau(i)\right\} \right|.
\]

\begin{definition}\label{def-P-stable}

A sequence of maps $\sigma_{k}:\Gamma\to{\rm Sym}(n_{k})$, where
$\lim_{k\to\infty}n_{k}=\infty$, is called an \emph{asymptotic homomorphism}
if $\lim_{k\to\infty}d_{n_{k}}\left(\sigma_{k}(gh),\sigma_{k}(g)\sigma_{k}(h)\right)=0$
for any $g,h\in\Gamma$. A group $\Gamma$ is said to be \emph{P-stable}
if for any asymptotic homomorphism $\sigma_{k}:\Gamma\to{\rm Sym}(n_{k})$,
there exists a sequence of homomorphisms $\tau_{k}:\Gamma\to{\rm Sym}(n_{k})$
such that $\lim_{k\to\infty}d_{n_{k}}\left(\sigma_{n}(g),\tau_{n}(g)\right)=0$
for every $g\in\Gamma$. 

\end{definition}

By \cite[Theorem 1.3]{BLT19}, if $\Gamma$ is a finitely generated
amenable group, then $\Gamma$ is P-stable if and only if every IRS
of $G$ is co-sofic. Therefore finitely generated amenable groups
satisfying the assumptions Theorem \ref{cosofic-1} are P-stable. 

\appendix

\section{Invariant measures on the space of closed subsets under the action
of an LDA-group}

Let $\mathcal{H}$ be an AF-groupoid associated with a simple Bratteli
diagram $B$. The unit space of $\mathcal{H}$ can be identified with
the space $X_{B}$ of infinite paths of the diagram $B$. In the appendix
we apply the pointwise ergodic theorem to understand $\mathsf{A}(\mathcal{H})$-invariant
measures on $F(X_{B})$. The alternating full group $\mathsf{A}(\mathcal{H})$
is isomorphic to the direct limit of direct product $\Gamma_{n}=\prod_{v\in V_{n}}{\rm Alt}(E(v_{0},v))$.
Apply the pointwise ergodic theorem from \cite{Vershik74,Olshanski-Vershik}
(or the general pointwise ergodic theorem for amenable groups in \cite{Lindenstrauss})
to the locally finite group $\mathsf{A}(\mathcal{H})=\cup_{n=1}^{\infty}\Gamma_{n}$,
we have that if $\nu$ is an ergodic $\mathsf{A}(\mathcal{H})$-invariant
probability measure on $F(X_{B})$, then for any measurable function
$f:F(X_{B})\to\mathbb{R}$, 
\[
\int fd\nu=\lim_{n\to\infty}\frac{1}{|\Gamma_{n}|}\sum_{g\in\Gamma_{n}}f(C\cdot g),
\]
for $\nu$-a.e. $C$. 

\begin{lemma}\label{kset}

Let $\mathcal{H}$ be a minimal AF groupoid with unit space $X=\mathcal{H}^{(0)}$
homeomorphic to the Cantor set. Let $\nu$ be an $\mathsf{A}(\mathcal{H})$-invariant
ergodic probability measure on $F(X)$, $\nu\neq\delta_{\emptyset},\delta_{X}$.
Then there exists a constant $k\in\mathbb{N}$ such that $\nu$ is
supported on the subspace $X^{(k)}$ of finite sets of cardinality
$k$. 

\end{lemma}

\begin{proof}

Since $\mathcal{H}$ is assumed to be a minimal AF-groupoid, it is
associated with a simple Bratteli diagram $B=(V,E)$ with path space
$X_{B}$ homeomorphic to $\mathcal{H}^{(0)}$. We use notations and
terminologies as reviewed in Example \ref{AF}. Telescoping if necessary,
we may assume $\left|E(v_{0},v)\right|\ge5$ for all $v\in V_{n}$,
$n\ge1$. let $\nu$ be a $\mathsf{A}(\mathcal{H})$-invariant ergodic
probability measure on $F(X_{B})$, $\nu\neq\delta_{\emptyset},\delta_{X}$. 

We first show that if $\nu$ is supported on infinite subsets of $X_{B}$,
then $\nu(C\subseteq U)=0$ for any clopen set $U\neq X$. To see
this, let $f=f_{U}:F(X_{B})\to\mathbb{R}$ be the indicator function
$f(C)=\mathbf{1}_{\{C\subseteq U\}}$. Then by the pointwise ergodic
theorem, for $\nu$-a.e. $C$, 
\begin{align*}
\int fd\nu & =\lim_{n\to\infty}\frac{1}{|\Gamma_{n}|}\sum_{g\in\Gamma_{n}}f(C\cdot g)\\
 & =\lim_{n\to\infty}\frac{\left|\left\{ g\in\Gamma_{n}:\ C\cdot g\subseteq U\right\} \right|}{\left|\Gamma_{n}\right|}.
\end{align*}

We introduce the following notations. For a given clopen set $U$,
let $n_{0}=n_{0}(U)$ be the minimal number such that $U$ can be
expressed as a disjoint union of cylinder sets of the form $U(e_{1},\ldots,e_{n_{0}})$.
For $v\in V_{n}$ and a set $A\subseteq X_{B}$, let $E(v_{0},v:A)$
be the collection of $n$-paths
\begin{equation}
E(v_{0},v:A)=\left\{ \left(e_{1},\ldots,e_{n}\right):\ r(e_{n})=v\mbox{ and }U(e_{1},\ldots,e_{n})\cap A\neq\emptyset\right\} .\label{eq:EA}
\end{equation}
Given a closed set $C$ and level $n\ge n_{0}$, write
\[
\tilde{C}_{n}=\cup\left\{ U(e_{1},\ldots,e_{n}):\ (e_{1},\ldots,e_{n})\in E(v_{0},v:A)\right\} .
\]

Recall that $\Gamma_{n}=\prod_{v\in V_{n}}{\rm Alt}(E(v_{0},v))$.
In particular, the action of $\Gamma_{n}$ on $X_{B}$ is by permuting
the $n$-prefix of the infinite paths. Since for $n\ge n_{0}(U)$,
$U$ can be expressed as a disjoint union of cylinder sets indexed
by $n$-paths, it is clear that for $g\in\Gamma_{n}$, $C\cdot g\subseteq U$
if and only if $\tilde{C}_{n}\cdot g\subseteq U$. Let $\tilde{P}_{n}=\{g\in\Gamma_{n}:\ \tilde{C}_{n}\cdot g=\tilde{C}_{n}\}$
be the setwise stabilizer of $\tilde{C}_{n}$ in $\Gamma_{n}$. Then
\begin{align}
\frac{\left|\left\{ g\in\Gamma_{n}:\ C\cdot g\subseteq U\right\} \right|}{\left|\Gamma_{n}\right|} & =\frac{\left|\left\{ g\in\Gamma_{n}:\ \tilde{C}_{n}\cdot g\subseteq U\right\} \right|}{\left|\Gamma_{n}\right|}\nonumber \\
 & =\frac{\left|\{K\in\tilde{P}_{n}\setminus\Gamma_{n}:\ K\subseteq U\}\right|}{\left|\Gamma_{n}:\tilde{P}_{n}\right|},\label{eq:coset1}
\end{align}
where we identity the coset space $\tilde{P}_{n}\setminus\Gamma_{n}$
as translates of the set $\tilde{C}_{n}$ of the form $\tilde{C}_{n}\cdot g$. 

Recall that the action of the alternating group ${\rm Alt}(m)$ is
transitive on the collection of subsets of fixed size $k$, where
$1\le k\le m$, $m\ge3$. Since $\Gamma_{n}=\prod_{v\in V_{n}}{\rm Alt}(E(v_{0},v))$,
the ratio in (\ref{eq:coset1}) can be computed as 

\begin{align}
\frac{\left|\{K\in\tilde{P}_{n}\setminus\Gamma_{n}:\ K\subseteq U\}\right|}{|\Gamma_{n}:\tilde{P}_{n}|} & =\prod_{v\in V_{n}}\frac{\left(\begin{array}{c}
\left|E(v_{0},v:U)\right|\\
\left|E(v_{0},v:C)\right|
\end{array}\right)}{\left(\begin{array}{c}
\left|E(v_{0},v)\right|\\
\left|E(v_{0},v:C)\right|
\end{array}\right)}.\nonumber \\
 & \le\prod_{v\in V_{n}}\left(\frac{\left|E(v_{0},v:U)\right|}{\left|E(v_{0},v)\right|}\right)^{\left|E(v_{0},v:C)\right|}\label{eq:ratio}
\end{align}
To estimate such ratios we use the following elementary fact about
path counting in a simple Bratteli diagram $B$:

\begin{fact}\label{markov}

Let $U(e_{1},\ldots,e_{k})$ be the cylinder set indexed by finite
path $\left(e_{1}\ldots e_{k}\right)$ and $u=r(e_{k})$ be the end
point of the path. Suppose that $m>k$ is a level such that all vertices
in $V_{k}$ and $V_{m}$ are connected. Then for any $n\ge m$ and
$v\in V_{n}$,
\[
\frac{N(v;e_{1},\ldots,e_{k})}{\left|E(v_{0},v)\right|}\ge\frac{1}{\max_{w\in V_{m}}|E(v_{0},w)|},
\]
where $N(v;e_{1},\ldots,e_{k}):=\left|\left\{ (f_{1},\ldots f_{n}):\ f_{i}=e_{i}\ \mbox{for }1\le i\le k\mbox{ and }r(f_{n})=v\right\} \right|$. 

\end{fact}

\begin{proof}[Proof of Fact \ref{markov}]

Recall that $E(w,v)$ denotes the set of finite paths with start point
$w$ and end point $v$. Path counting in the Bratteli diagram has
the Markov property:
\begin{align*}
\left|E(v_{0},v)\right| & =\sum_{w\in V_{m}}\left|E(v_{0},w)\right|\cdot\left|E(w,v)\right|,\\
N(v;e_{1},\ldots,e_{k}) & =\left|\left\{ (f_{1},\ldots f_{n}):\ f_{i}=e_{i}\ \mbox{for }1\le i\le k\mbox{ and }r(f_{n})=v\right\} \right|\\
 & =\sum_{w\in V_{m}}N(w;e_{1},\ldots,e_{k})\left|E(w,v)\right|.
\end{align*}
By assumption of the lemma we have $N(w;e_{1},\ldots,e_{k})\ge1$
for any $w\in V_{m}$ and finite path $(e_{1},\ldots,e_{k})$. Comparing
the two equations we have
\[
\frac{N(v;e_{1},\ldots,e_{k})}{|E(v_{0},v)|}\ge\frac{1}{\max_{w\in V_{m}}\left|E(v_{0},w)\right|}.
\]

\end{proof}

Now suppose the clopen set $U$ has non-empty compliment $U^{c}=X_{B}\setminus U$.
Take a cylinder set $W=U(f_{1},\ldots,f_{\ell})$ contained in $U^{c}$.
Let $m>\ell$ be a level such that any vertex in $V_{\ell}$ is connected
to any vertex in $V_{m}$, such a level $m$ exists because of the
minimality assumption on $\mathcal{H}$. It follows from the Fact
that for all $n\ge m$, for $v\in V_{n}$, 
\begin{align*}
\frac{\left|E(v_{0},v:U)\right|}{\left|E(v_{0},v)\right|} & \le1-\frac{N(v;f_{1},\ldots,f_{k})}{\left|E(v_{0},v)\right|}.\\
 & \le1-\frac{1}{\max_{w\in V_{m}}\left|E(v_{0},w)\right|}.
\end{align*}
Let $c$ be the constant that $e^{-c}=1-\frac{1}{\max_{w\in V_{m}}\left|E(v_{0},w)\right|}$.
Plug into (\ref{eq:ratio}), we conclude that for $n\ge m$,
\[
\frac{\left|\left\{ g\in\Gamma_{n}:\ C\cdot g\subseteq U\right\} \right|}{\left|\Gamma_{n}\right|}\le e^{-c\sum_{v\in V_{n}}\left|E(v_{0},v:C)\right|}.
\]
If $C$ is an infinite closed set, then $\sum_{v\in V_{n}}\left|E(v_{0},v:C)\right|\to\infty$
as $n\to\infty$. Therefore by the pointwise ergodic theorem, if $\nu(\{C:\ |C|=\infty\})=1$,
then for any clopen set $U$ such that $U\neq X$, $\nu$ a.e. $C$,
we have
\begin{align*}
\nu(\left\{ C:\ C\subseteq U\right\} ) & =\lim_{n\to\infty}\frac{\left|\left\{ g\in\Gamma_{n}:\ C\cdot g\subseteq U\right\} \right|}{\left|\Gamma_{n}\right|}\\
 & \le\lim_{n\to\infty}e^{-c\sum_{v\in V_{n}}\left|E(v_{0},v:C)\right|}=0.
\end{align*}

Since the sets $\left\{ C\in F(X_{B}):\ C\subseteq U\right\} $ where
$U$ goes over all clopen proper subset of $X$ form a countable cover
of $\{C\in F(X_{B}):C\neq X,\ |C|=\infty\}$, we conclude that $\nu(\{C:\ |C|=\infty,\ C\neq X\})=0$.
This means that $\nu$ is $\delta$-mass at $X$, which is a contradiction.
By ergodicity of $\nu$ we conclude that it must be supported on subsets
of fixed size $k$, for some $k\in\mathbb{N}$.

\end{proof}

As in the statement of the previous lemma, let $X^{(k)}$ be the space
of subsets of size $k$ of $X$. We apply the pointwise ergodic theorem
again to obtain further information on $\mathsf{A}(\mathcal{H})$-invariant
probability measures on $X^{(k)}$. Recall that $M(X,\mathcal{H})$
denotes $\mathcal{H}$-invariant probability measures on $X$. 

\begin{lemma}\label{invariant measure}

Let $\mathcal{H}$ be a minimal AF groupoid with unit space $X$ homeomorphic
to the Cantor set. Let $\nu$ be an ergodic $\mathsf{A}(\mathcal{H})$-invariant
probability measure on $X^{(k)}$. Then there exists an $k$-tuple
of ergodic measures $(\mu_{1},\ldots,\mu_{k})\in M(X,\mathcal{H})^{k}$
such that for any Borel set $U\subseteq X$,
\[
\nu(\{C\in X^{(k)}:C\subseteq U\})=\mu_{1}(U)\ldots\mu_{k}(U).
\]

\end{lemma}

\begin{proof}

Let $\nu$ be an ergodic $\mathsf{A}(\mathcal{H})$-invariant probability
measure on $X^{(k)}$ as given and $\eta$ be the measure on $X$
defined by 
\[
\eta=\int_{X^{(k)}}{\bf u}_{C}d\nu(C),
\]
where ${\bf u}_{C}$ is the uniform measure on the finite set $C$.
It's clear by definition that $\eta$ is an $\mathsf{A}(\mathcal{H})$-invariant
probability measure on $X$. Write $\eta=\int_{Y}\mu_{y}d\theta(y)$
for its ergodic decomposition. Note that $\eta$ has at most $k$
ergodic components. Indeed, if this claim is not true, then we can
find $k+1$ disjoint measurable sets $B_{1},\ldots,B_{k+1}$ in $Y$
such that $\theta(B_{i})>0$ for all $1\le i\le k+1$. Let $\eta_{i}=\frac{1}{\theta(B_{i})}\int_{B_{i}}\mu d\theta(\mu)$.
Since the sets $B_{i}$'s are disjoint, we have that the collection
of measures $\eta_{i}$, $1\le i\le k+1$ are mutually singular. Let
$A_{1},A_{2},\ldots,A_{k+1}$ be a choice of disjoint sets in the
$\mathsf{A}(\mathcal{H})$-invariant $\sigma$-field $\mathcal{I}$
on $X$ such that $\eta_{i}(A_{i})=1$ for each $i$. Now consider
the function $\psi:X^{(k)}\to\{0,1\}^{k+1}$ defined as $(\psi(C))_{i}={\bf 1}_{\{C\cap A_{i}\neq\emptyset\}}$.
Since each $A_{i}\in\mathcal{I}$, we have that $\psi$ is a measurable
function invariant under $\mathsf{A}(\mathcal{H})$. By ergodicity
of $\nu$, it follows that $\nu$-a.e. $\psi$ is a constant. Because
$C$ only contains $k$ points, there must exists a coordinate $j$
such that $(\psi(C))_{j}=0$. It follows that there is $j\in\{1,\ldots,k+1\}$
such that $\nu(\{C:C\cap A_{j}\neq\emptyset\})=0$. However since
$\eta=\sum_{i=1}^{k+1}\theta(B_{i})\eta_{i}$, by relation of $\eta$
to $\nu$, we have $\nu(\{C:C\cap A_{j}\neq\emptyset\})\ge\frac{1}{k}\theta(B_{j})>0$,
which is a contradiction.

We have seen that there exists $\ell\in\{1,\ldots,k\}$ and ergodic
$\mathsf{A}(\mathcal{H})$-invariant probability measures $\mu_{1},\ldots,\mu_{\ell}$
on $X$ such that the ergodic decomposition of $\eta$ is given by
$\eta=\sum_{j=1}^{\ell}c_{j}\mu_{j}$, $c_{j}>0$. Let $\Omega_{j}$
be a subset of $X$ provided by the pointwise ergodic theorem such
that $\mu_{j}(\Omega_{j})=1$ and for any $x\in\Omega_{j}$ and any
clopen set $U\subseteq X$, we have
\[
\mu_{j}(U)=\lim_{n\to\infty}\frac{\left|\left\{ g\in\Gamma_{n}:x\cdot g\in U\right\} \right|}{|\Gamma_{n}|}.
\]
Let $\Omega$ be the subset of $X^{(k)}$ defined as $\Omega=\left\{ C\in X^{(k)}:\ C\subseteq\cup_{j=1}^{\ell}\Omega_{j}\right\} $.
Note that $\nu(\Omega^{c})\le k\eta\left(\left(\cup_{j=1}^{\ell}\Omega_{j}\right)^{c}\right)=0$. 

Let $\Omega'$ be the subset of $X^{(k)}$ of full $\nu$-measure
provided by the pointwise ergodic theorem that is for any $K\in\Omega'$
and any clopen set $U\subseteq X$,
\[
\nu(\{C:C\subseteq U\})=\lim_{n\to\infty}\frac{\left|\left\{ g\in\Gamma_{n}:\ K\cdot g\subseteq U\right\} \right|}{\left|\Gamma_{n}\right|}.
\]

Now take a set $K=\{z_{1},\ldots,z_{k}\}\in\Omega\cap\Omega'$. Let
$U$ be a non-empty clopen set. Let $n\ge n_{0}(U)$ be sufficiently
large such that all elements of $K$ have distinct $n$-prefix. We
use the notation $E(v_{0},v:A)$ defined in (\ref{eq:EA}) in the
proof of the previous lemma. As calculated in (\ref{eq:coset1}) and
(\ref{eq:ratio}), we have 

\[
\frac{\left|\left\{ g\in\Gamma_{n}:\ K\cdot g\subseteq U\right\} \right|}{\left|\Gamma_{n}\right|}=\prod_{v\in V_{n}}\frac{\left(\begin{array}{c}
\left|E(v_{0},v:U)\right|\\
\left|E(v_{0},v:K)\right|
\end{array}\right)}{\left(\begin{array}{c}
\left|E(v_{0},v)\right|\\
\left|E(v_{0},v:K)\right|
\end{array}\right)},
\]
and 
\[
\prod_{i=1}^{k}\frac{\left|\left\{ g\in\Gamma_{n}:\ z_{i}\cdot g\subseteq U\right\} \right|}{\left|\Gamma_{n}\right|}=\prod_{v\in V_{n}}\prod_{i=1}^{k}\frac{\left(\begin{array}{c}
\left|E(v_{0},v:U)\right|\\
\left|E(v_{0},v:\{z_{i}\})\right|
\end{array}\right)}{\left(\begin{array}{c}
\left|E(v_{0},v_{i})\right|\\
\left|E(v_{0},v:\{z_{i}\})\right|
\end{array}\right)}.
\]
Note that $\left|E(v_{0},v:\{z_{i}\})\right|$ takes value in $\{0,1\}$
and $\left|E(v_{0},v:K)\right|=\sum_{i=1}^{k}\left|E(v_{0},v:\{z_{i}\})\right|$.
Thus the ratio between the two is
\begin{align*}
\frac{\left|\left\{ g\in\Gamma_{n}:\ K\cdot g\subseteq U\right\} \right|/|\Gamma_{n}|}{\prod_{i=1}^{k}\left(\left|g\in\Gamma_{n}:z_{i}\cdot g\in U\right|/|\Gamma_{n}|\right)} & =\prod_{v\in V_{n}}\frac{\left(\begin{array}{c}
\left|E(v_{0},v:U)\right|\\
\left|E(v_{0},v:K)\right|
\end{array}\right)\left|E(v_{0},v:U)\right|^{-\left|E(v_{0},v:K)\right|}}{\left(\begin{array}{c}
\left|E(v_{0},v)\right|\\
\left|E(v_{0},v:K)\right|
\end{array}\right)\left|E(v_{0},v)\right|^{-\left|E(v_{0},v:K)\right|}}\\
 & =\prod_{v\in V_{n}}\frac{\prod_{j=1}^{\left|E(v_{0},v:K)\right|}\left(1-\frac{j-1}{\left|E(v_{0},v:U)\right|}\right)}{\prod_{j=1}^{\left|E(v_{0},v:K)\right|}\left(1-\frac{j-1}{\left|E(v_{0},v)\right|}\right)}\\
 & \ge\prod_{v\in V_{n}}\left(1-\frac{\left|E(v_{0},v:K)\right|\left(\left|E(v_{0},v:K)\right|-1\right)}{2\left|E(v_{0},v:U)\right|}\right)\\
 & \ge1-\frac{k^{3}}{\min_{v\in V_{n}}\left|E(v_{0},v:U)\right|}.
\end{align*}
It is also clear from the formula that this ratio is bounded from
above by $1$. 

Note that $\min_{v\in V_{n}}\left|E(v_{0},v:U)\right|\to\infty$ as
$n\to\infty$ since $B$ is a simple Bratteli diagram. Indeed, the
sequences $\min_{v\in V_{n}}\left|E(v_{0},v)\right|$, $\max_{v\in V_{n}}\left|E(v_{0},v)\right|$
are non-decreasing in $n$. If every vertex in level $n$ is connected
to every vertex in level $m$ for some $m>n$, then it follows $\max_{v\in V_{n}}\left|E(v_{0},v)\right|\le\min_{w\in V_{m}}|E(v_{0},w)|$.
Therefore since $B$ is simple, we have that $\lim_{n\to\infty}\min_{v\in V_{n}}\left|E(v_{0},v)\right|<\infty$
is equivalent to $\lim_{n\to\infty}\max_{v\in V_{n}}\left|E(v_{0},v)\right|<\infty$.
In the case that $\max_{v\in V_{n}}\left|E(v_{0},v)\right|$ is bounded,
the path space $X_{B}$ is either finite (if $|V_{n}|$ is bounded)
or countably discrete (if $|V_{n}|$ is unbounded). By Fact \ref{markov}
we conclude that $\min_{v\in V_{n}}\left|E(v_{0},v:U)\right|\to\infty$
as $n\to\infty$ if $X_{B}$ is homeomorphic to the Cantor set. Therefore
\[
\lim_{n\to\infty}\frac{\left|\left\{ g\in\Gamma_{n}:\ K\cdot g\subseteq U\right\} \right|/|\Gamma_{n}|}{\prod_{i=1}^{k}\left(\left|g\in\Gamma_{n}:z_{i}\cdot g\in U\right|/|\Gamma_{n}|\right)}=1.
\]

Since $K\in\Omega\cap\Omega'$, from the way $\Omega$ and $\Omega'$
are chosen, we have for each $1\le i\le k$, there is an index $j_{i}\in\{1,2,\ldots,\ell\}$
such that for any clopen set $U$,
\begin{align*}
\lim_{n\to\infty}\left|\left\{ g\in\Gamma_{n}:\ K\cdot g\subseteq U\right\} \right|/|\Gamma_{n}| & =\nu(C:C\subseteq U),\\
\lim_{n\to\infty}\left|g\in\Gamma_{n}:z_{i}\cdot g\in U\right|/|\Gamma_{n}| & =\mu_{j_{i}}(U).
\end{align*}
We conclude that 
\[
\nu(C:C\subseteq U)=\mu_{j_{1}}(U)\ldots\mu_{j_{k}}(U).
\]

\end{proof}

\begin{remark}

Lemma \ref{kset} and \ref{invariant measure} imply that an ergodic
$\mathsf{A}(\mathcal{H})$-invariant measure on $F(X)$ other than
$\delta_{\emptyset}$ or $\delta_{X}$ must be pushforward of a product
measure $\mu_{1}\times\ldots\times\mu_{k}$ on $X^{k}$ under $(x_{1},\ldots,x_{k})\mapsto\{x_{1},\ldots,x_{k}\}$
for some $k$-tuple of ergodic invariant measures. 
\begin{description}
\item [{(i)}] Since the group $\mathsf{A}(\mathcal{H})$ is an infinite
simple locally finite group, by \cite[Theorem 2.4]{Thomas-Tucker-Drob1}
the ergodic measures in $M(X,\mathcal{H})$ are weakly mixing. Therefore
the converse direction of Lemma \ref{invariant measure} is true:
for any tuple $(\mu_{1},\ldots,\mu_{k})\in M(X,\mathcal{H})^{k}$
of ergodic invariant measures, the measure $\nu$ defined by $\nu(\{C\in X^{(k)}:C\subseteq U\})=\mu_{1}(U)\ldots\mu_{k}(U)$
is an ergodic $\mathsf{A}(\mathcal{H})$-invariant measure on $X^{(k)}$. 
\item [{(ii)}] When there are only finitely many ergodic invariant measures
in $M(X,\mathcal{H})$, one can also deduce Lemma \ref{kset} and
\ref{invariant measure} as a consequence of the classification of
IRSs of $\mathsf{A}(\mathcal{H})$ in \cite{Dudko-Medynets3}. 
\item [{(iii)}] Invariant measures on Bratteli diagrams have been investigated
extensively in \cite{BKMS10,BKMS13}. For instance, sufficient conditions
for unique ergodicity or admitting finitely many ergodic invariant
measures are provided there.
\end{description}
\end{remark}

\bibliographystyle{alpha}
\bibliography{IRS}

\end{document}